\documentclass[12pt,a4paper]{amsart}
\usepackage{amscd, amssymb}
\usepackage{amsthm}
\usepackage[all]{xy}
\usepackage{tikz}
\usepackage{float}
\usetikzlibrary{arrows,backgrounds,decorations.markings}
\usepackage[total={6.6in,9.2in}, top=1.5in, left=0.85in, includefoot]{geometry}
\usepackage[colorlinks=true, pdfstartview=FitV, linkcolor=blue, citecolor=black, urlcolor=black,filecolor=magenta]{hyperref}
\usepackage{graphicx, verbatim,amsmath,amssymb,subfigure,enumerate}

\allowdisplaybreaks

\def\N{\mathbb{N}}
\def\Z{\mathbb{Z}}

\def\CC{\mathcal{C}}

\renewcommand{\int}{\operatorname{int}}
\renewcommand{\mid}{:}

\newtheorem{thm}{Theorem}[section]

\newtheorem{lemma}[thm]{Lemma}
\newtheorem{prop}[thm]{Proposition}

\newtheorem{Thm}{Theorem}

\theoremstyle{definition}
\newtheorem{definition}[thm]{Definition}

\theoremstyle{remark}
\newtheorem{remark}[thm]{Remark}

\newtheorem*{Acknowledgements}{Acknowledgements}

\numberwithin{equation}{section}

\tikzstyle{vertex}=[circle,thick]
\tikzstyle{goto}=[->,shorten >=1pt,>=stealth,semithick]

\pgfpathcircle{\pgfpoint{5cm}{0cm}} {2pt}

 \title[A dendritic monotile]{An aperiodic monotile that forces nonperiodicity through dendrites}

 \author{Michael Mampusti}
 \address{Michael Mampusti \\ School of Mathematics and Applied Statistics  \\ The University of Wollongong\\ NSW  2522\\ Australia} 
 
 \email{michael.mampusti@gmail.com}
 
\author[Michael F. Whittaker]{Michael F. Whittaker}
\address{Michael F. Whittaker, School of Mathematics and Statistics, University of Glasgow, University Place, Glasgow Q12 8QQ, United Kingdom}
\email{Mike.Whittaker@glasgow.ac.uk}
 
 \thanks{This research was partially supported by EPSRC grant EP/R013691/1, ARC Discovery Project DP150101595, and the Australian Government Research Training Program Scholarship.}
 

\begin{document}

\begin{abstract}
We introduce a new type of aperiodic hexagonal monotile; a prototile that admits infinitely many tilings of the plane, but any such tiling lacks any translational symmetry. Adding a copy of our monotile to a patch of tiles must satisfy two rules that apply only to adjacent tiles. The first is inspired by the Socolar--Taylor monotile, but can be realised by shape alone. The second is a dendrite rule; a direct isometry of our monotile can be added to any patch of tiles provided that a tree on the monotile connects continuously with a tree on one of its neighbouring tiles. This condition forces tilings to grow along dendrites, which ultimately results in nonperiodic tilings. Our dendrite rule initiates a new method to produce tilings of the plane.
\end{abstract}

\maketitle

\section{Introduction}

Almost 60 years ago, Hao Wang posed the Domino Problem \cite{Wang}: is there an algorithm that determines whether a given set of square prototiles, with specified matching rules, can tile the plane? Robert Berger \cite{Ber} proved the Domino Problem is undecidable by producing an aperiodic set of 20,426 prototiles, a collection of prototiles that tile the plane but only nonperiodically (lacks any translational periodicity). This remarkable discovery began the search for other (not necessarily square) aperiodic prototile sets. In the 1970s, there were two stunning results giving examples of very small aperiodic prototile sets. The first was by Raphael Robinson who found a set of six square prototiles \cite{Rob}. The second was by Roger Penrose who reduced this number to two \cite{GS,Pen}. Penrose's discovery led to the planar einstein (one-stone) problem: is there a single aperiodic prototile?

In a crowning achievement of tiling theory, the existence of an aperiodic monotile was resolved almost a decade ago by Joshua Socolar and Joan Taylor \cite{ST,Tay}. Several candidates had been put forth prior to their monotile, but the experts immediately recognised the importance of Socolar and Taylor's discovery \cite{AL,BGG,BG,Lee,LM}. The Socolar--Taylor monotile is a hexagonal tile with two local rules that enforce aperiodicity. The first rule forces tiles to arrange themselves into collections of triangles, and the second rule ensures these triangles are nested, thereby forcing the resulting tiling to be nonperiodic. One limitation of the Socolar--Taylor monotile is that the second local rule applies to pairs of non-adjacent tiles, so aperiodicity is not enforced by adjacencies. Another limitation is that reflected copies of the monotile are required to tile the plane. The search for an aperiodic monotile with local rules that only apply to adjacent tiles or does not require reflections has been a driving force of research in tiling theory since Socolar and Taylor's amazing discovery.

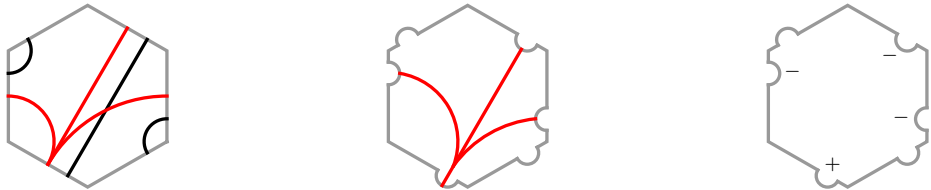
\begin{figure}[h]
\begin{center}
\begin{tikzpicture}
\node (tilea) at (0,0) {\includegraphics[width=0.13\textwidth]{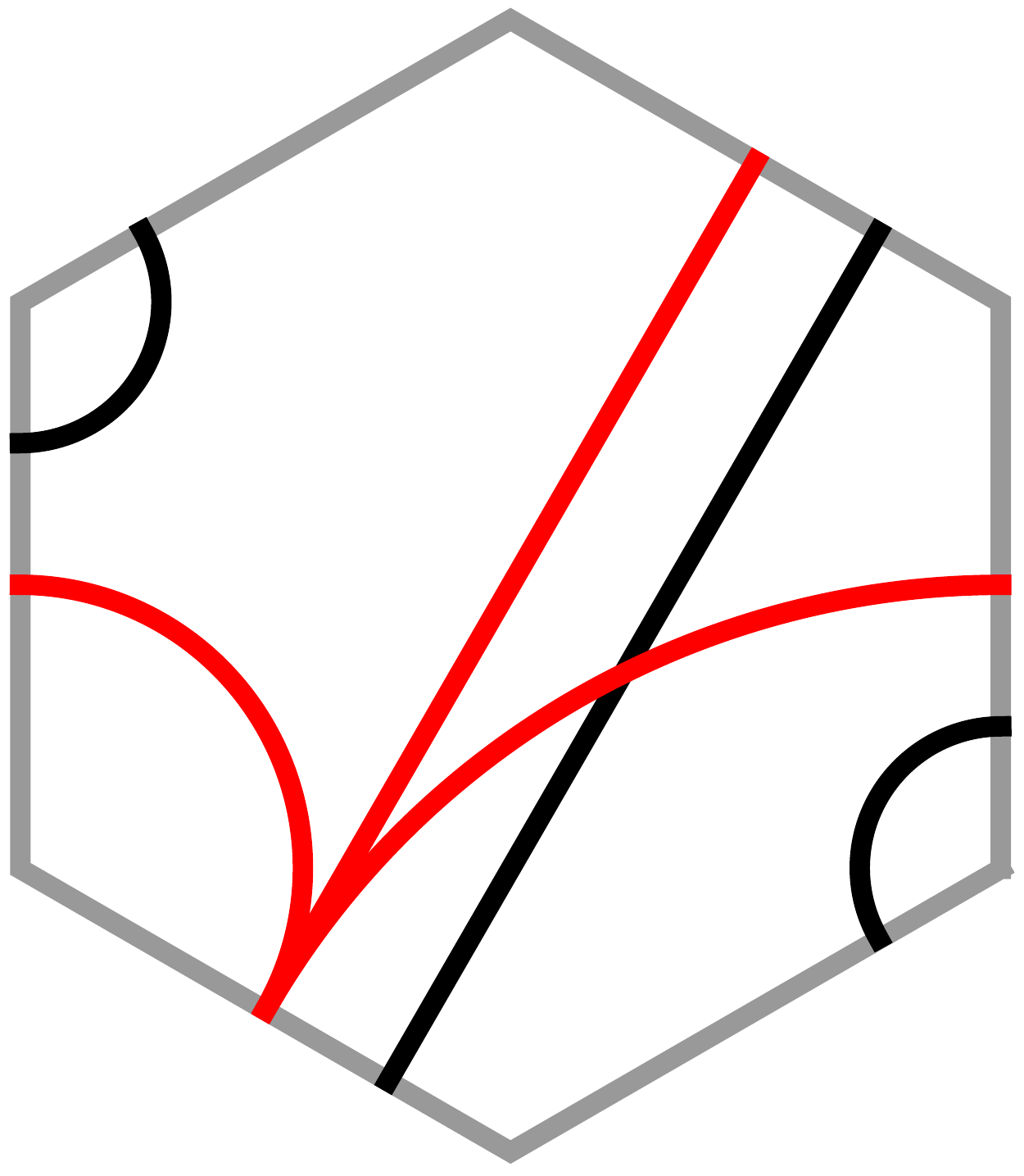}};
\node (tileb) at (5,0) {\includegraphics[width=0.13\textwidth]{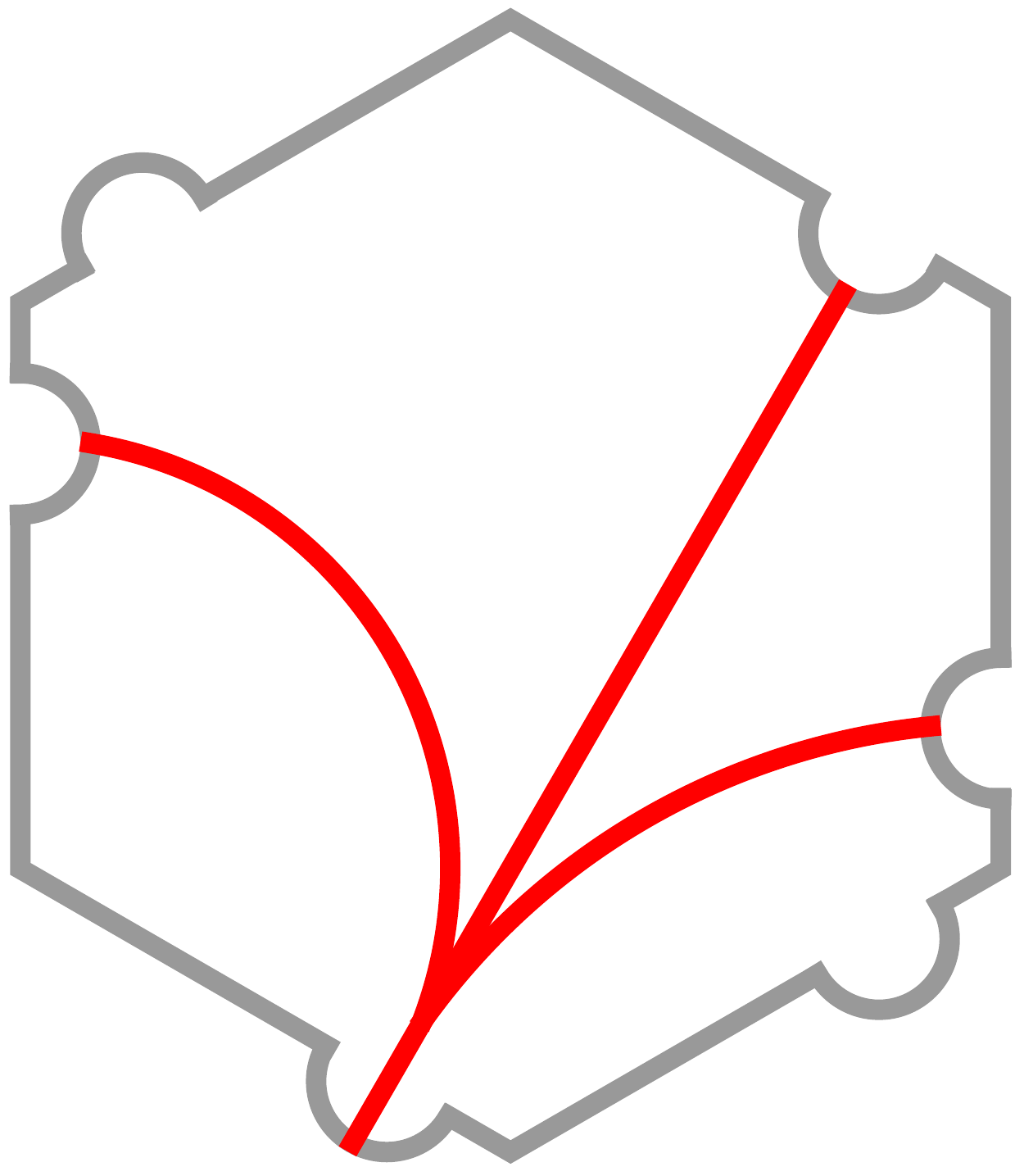}};
\node (tilec) at (10,0) {\includegraphics[width=0.13\textwidth]{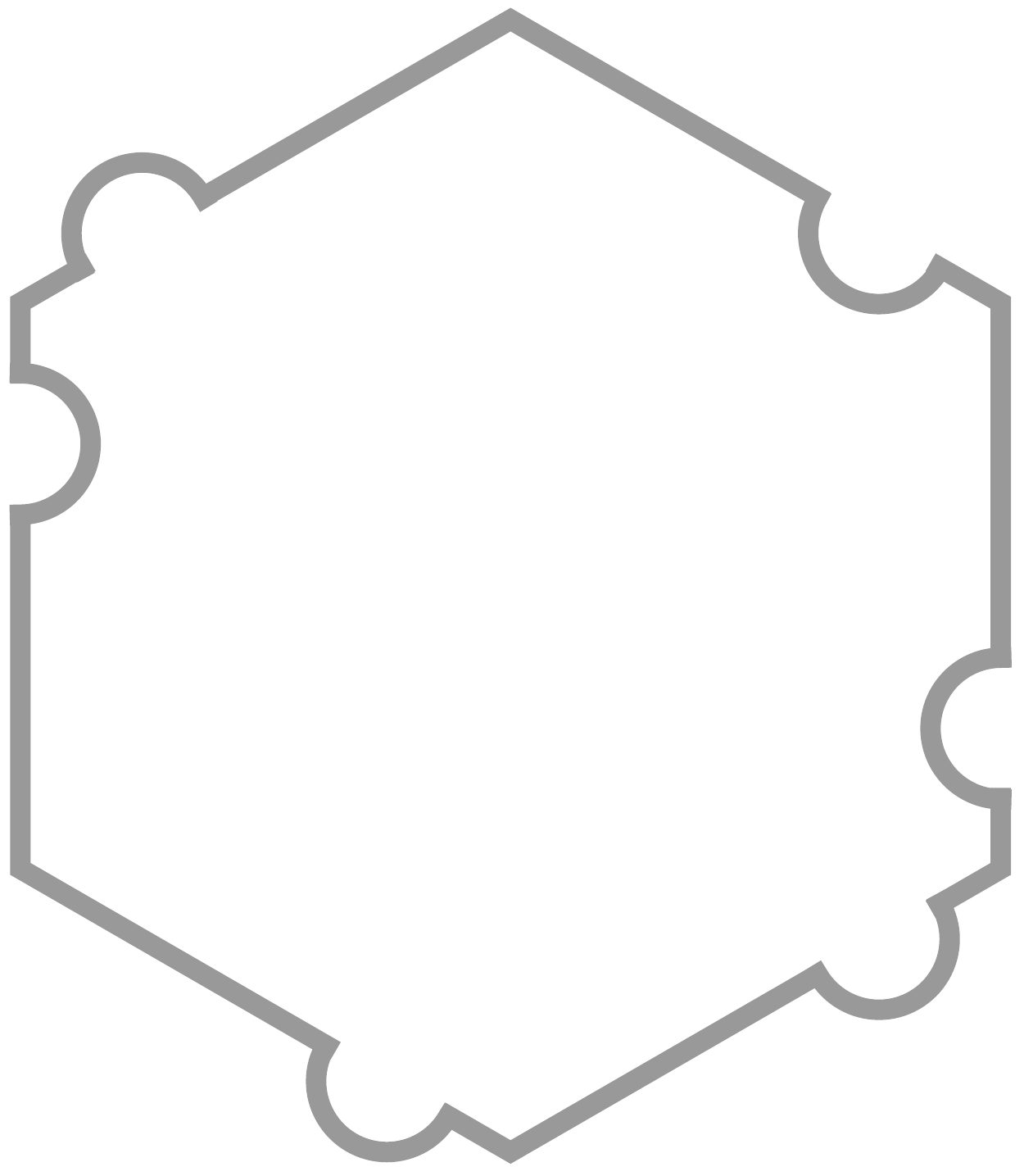}};
\node at (10.55,0.55) {\Tiny $-$};
\node at (10.7,-0.28) {\Tiny $-$};
\node at (9.27,0.34) {\Tiny $-$};
\node at (9.8,-0.9) {\Tiny $+$};
\end{tikzpicture}
\end{center}
\caption{Representations of our monotile: the R1-curves appear as black curves or edge decorations,  and the R2-tree as red trees or magnetic dipoles.}
\label{the tile}
\end{figure}

In this paper, we put forward a new type of aperiodic monotile that does not require a reflection, and has rules that only apply to adjacent tiles. We start with a hexagonal tile satisfying Socolar and Taylor's first local rule, and add a rule that only allows finite patches of tiles to connect along a dendrite. Our second rule is motivated by the proposed growth of certain quasicrystals \cite{IKG,ZULPDH}. Interestingly, a consequence of not requiring a reflected copy of our monotile is that the first rule can be enforced by shape alone, while this does not hold for the Socolar--Taylor monotile \cite[p.22]{ST2}. Although we are comparing our monotile with Socolar and Taylor's construction, these two monotiles are different in character. The Socolar--Taylor monotile is defined using local rules, whilst our monotile is defined by pairing a local rule and a dendritic rule. Indeed, our dendrite rule is local in the sense of building tilings, but is not local as a rule on tilings. So the monotile we present here is not an einstein in the technical sense, but rather a variant that requires tilings to be constructed starting from a seed tile. Three representations of our monotile appear in Figure \ref{the tile}, and each of these representations must satisfy the local rule \textbf{R1} and the dendrite rule \textbf{R2} outlined below.

\smallskip

Before introducing our monotile, we briefly define the terminology used in the paper. A {\em tiling} is a covering of the plane by closed topological discs, called tiles, that only intersect on their boundaries. A \emph{patch} is a finite connected collection of tiles that only intersect on their boundary. The building blocks of a tiling are the \emph{prototiles}: a finite set of tiles with the property that every tile is a direct isometry (an orientation preserving isometry) of a prototile. If the prototile set consists of a single tile, or a single tile and its reflection, we call it a \emph{monotile}. A tiling is said to be {\em nonperiodic} if it lacks any translational periodicity, and a set of prototiles is called \emph{aperiodic} if it can only form nonperiodic tilings.

Our monotile has two distinct features: a disconnected set of three curves that meet tile edges off-centre, and a connected tree that meets itself whenever two edges with a tree intersect. Once a single tile has been placed, a direct isometry of our monotile can be added to the plane provided the resulting collection of tiles is a patch, and
\begin{itemize}
\item[\textbf{R1}:] the black off-centre lines and curves must be continuous across tiles (c.f. \cite[\textbf{R1}]{ST}) and
\item[\textbf{R2}:] the new tile's red tree continuously connects with at least one tree of an adjacent tile.
\end{itemize}
We note that \textbf{R1} can be realised by shape alone, represented as puzzle like edge contours, while \textbf{R2} can be represented by magnetic dipole-dipole coupling. These representations appear in Figure \ref{the tile}.

In what follows, we will refer to the off-centre decorations that determine \textbf{R1} as R1-curves. These combine to form R1-triangles. Similarly, we will refer to the decorations that determine \textbf{R2} as R2-trees. We say that a tiling $T$ satisfies rule $R$ if every patch in $T$ is contained in a patch that can be constructed following rule $R$. Therefore, a tiling $T$ satisfies \textbf{R2} if and only if the union of R2-trees in $T$ is connected. Two legal patches satisfying \textbf{R1} and \textbf{R2} appear in Figure \ref{legal patch}. From this point forward we will use the representation from the left hand side of Figure \ref{the tile}, since we use R1-triangles heavily in the arguments that follow.

\begin{figure}[ht]
\[
\includegraphics[width=0.4\textwidth]{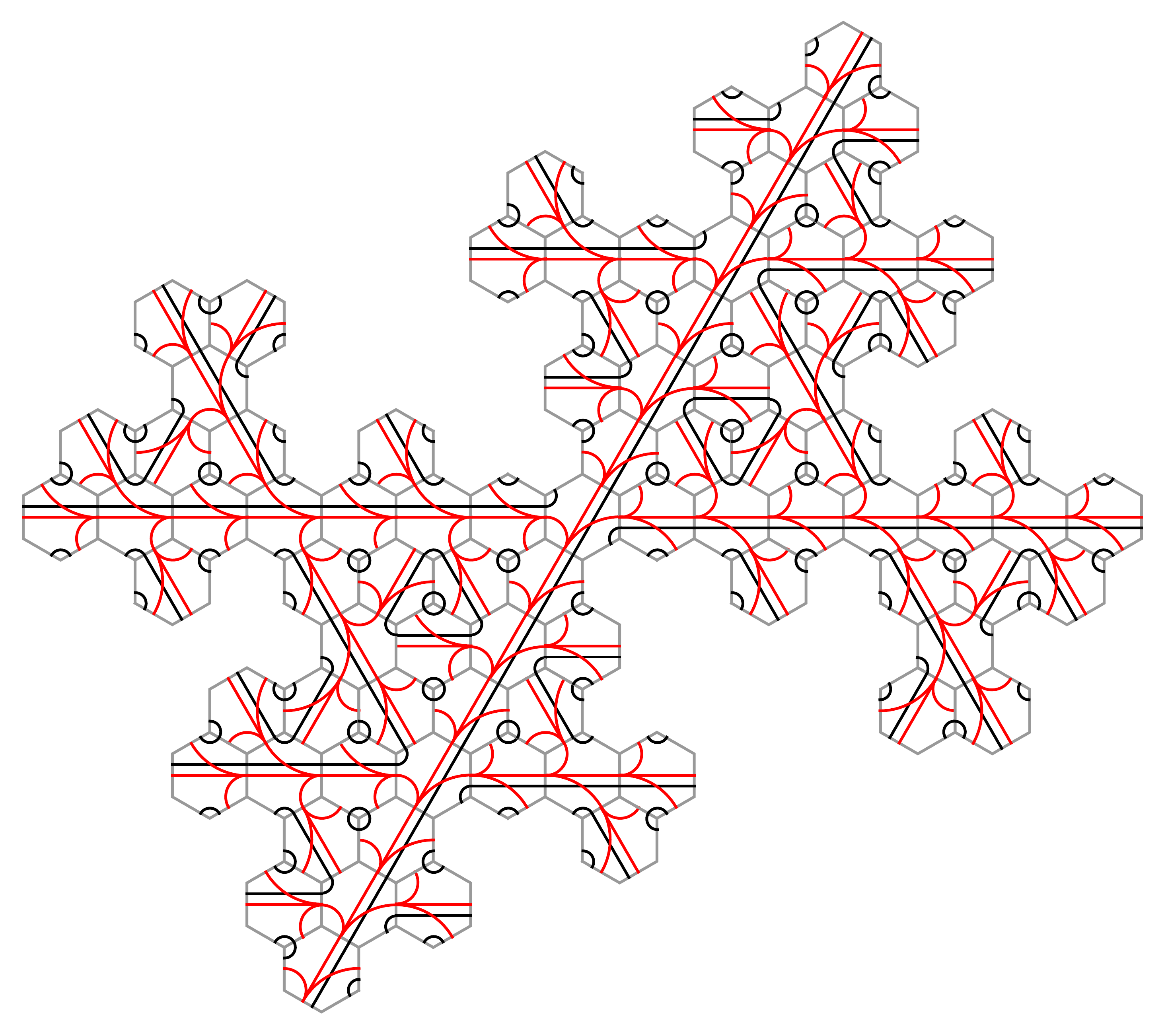} \qquad 
\includegraphics[width=0.4\textwidth]{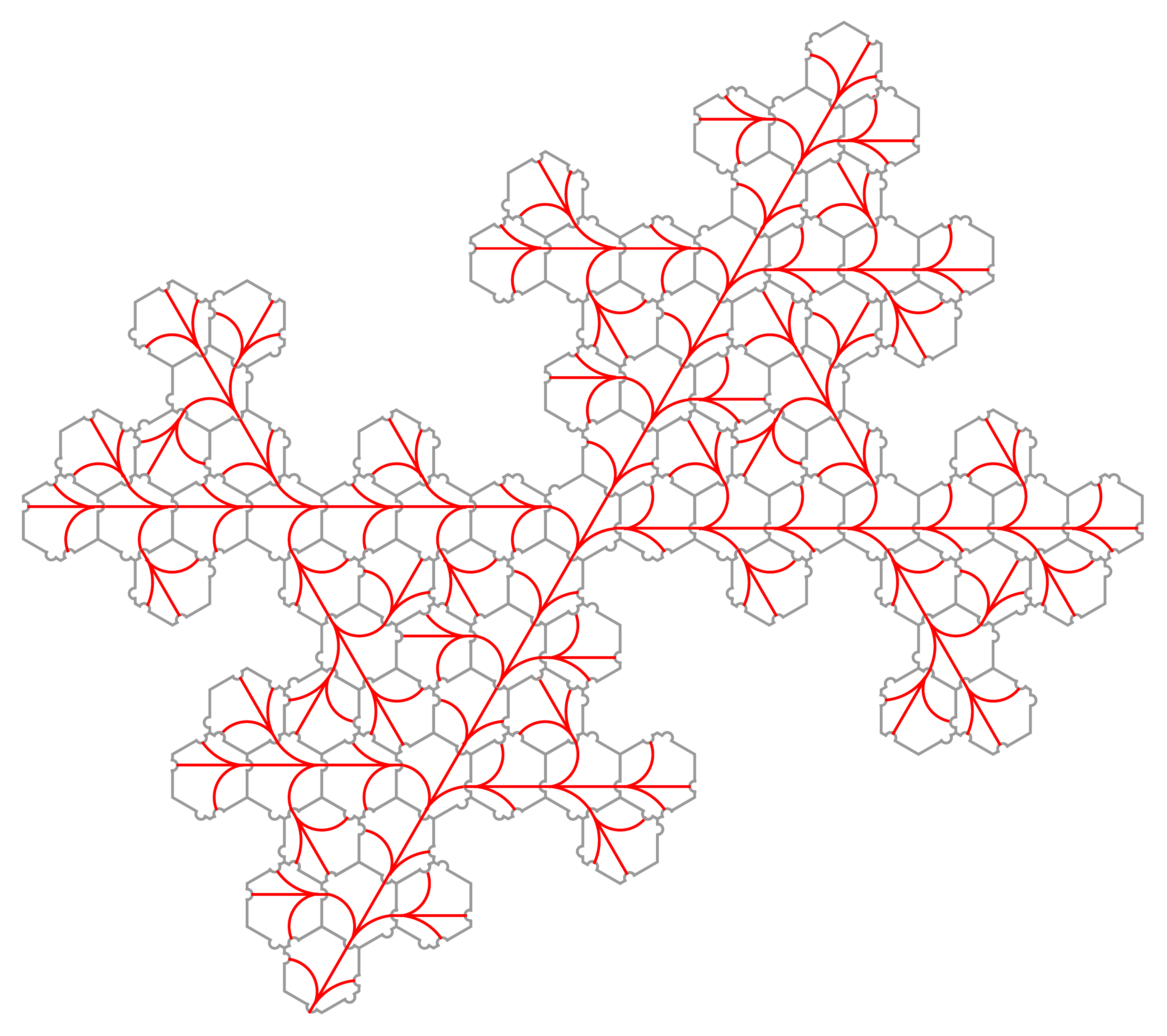}
\]
\[
\includegraphics[width=0.4\textwidth]{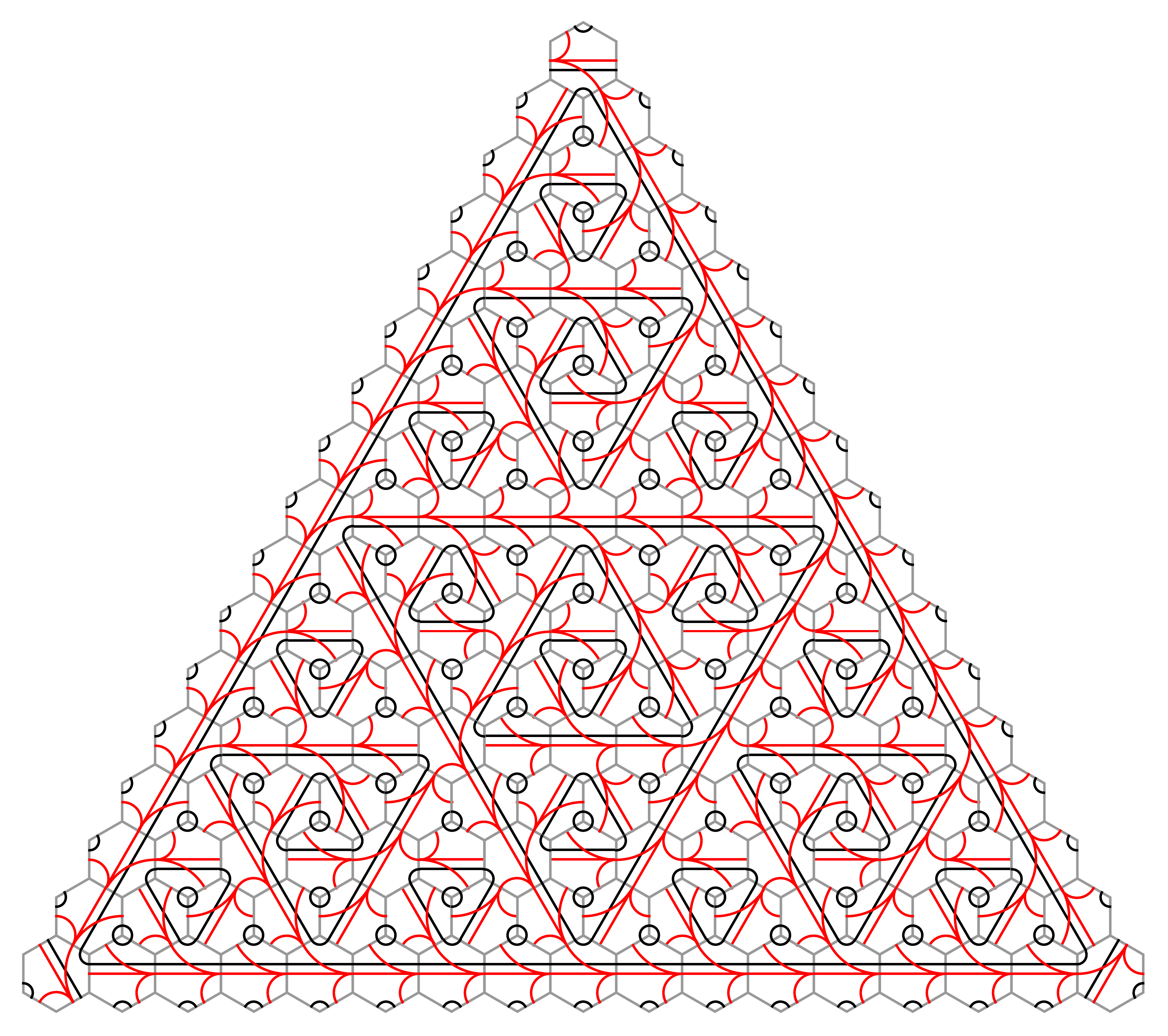} \qquad 
\includegraphics[width=0.4\textwidth]{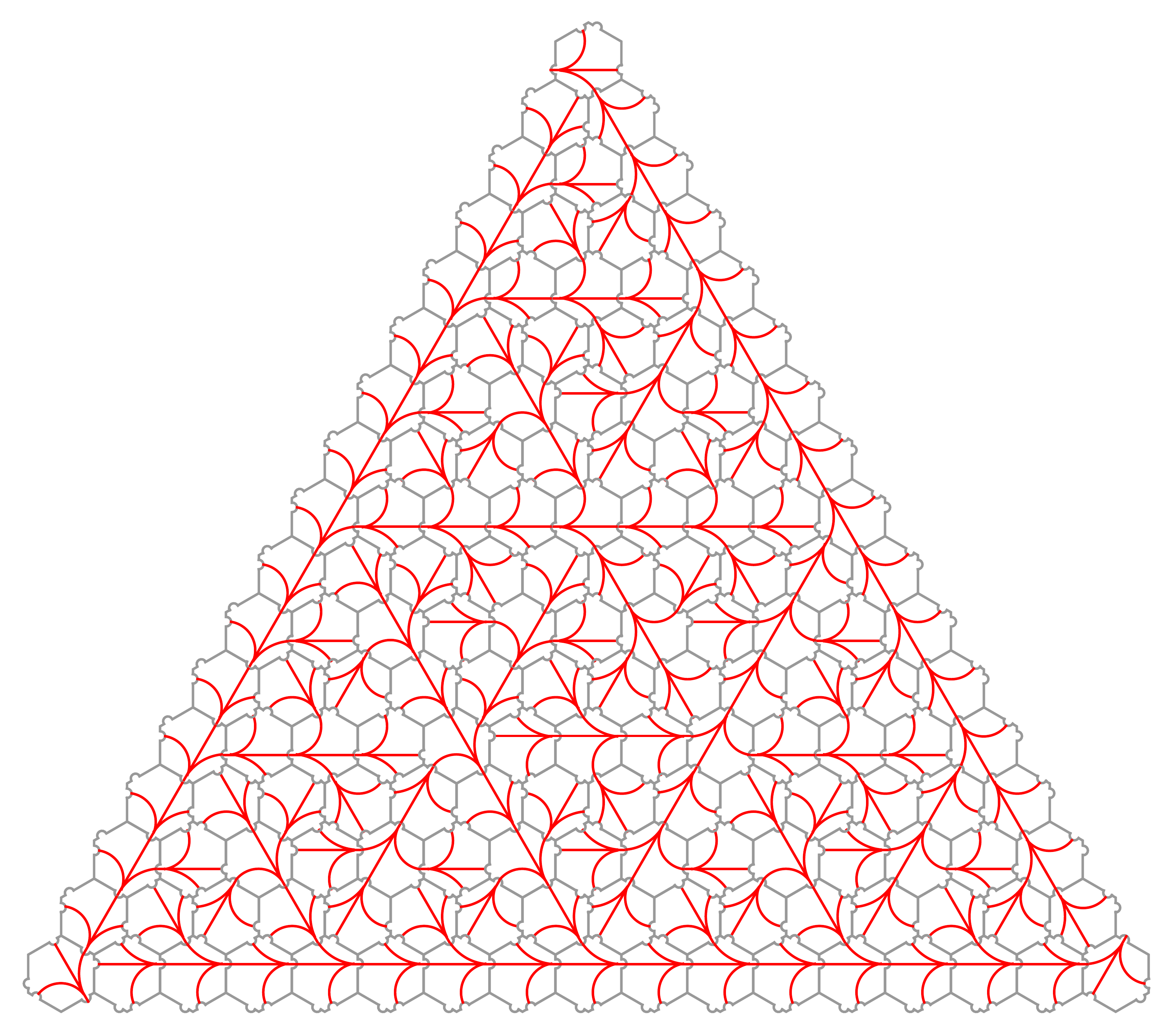}
\]
\caption{Two sets of equivalent legal patches. The patches on the top can be constructed directly using \textbf{R1} and \textbf{R2}, while the patches on the bottom cannot, but are subsets of legal patches}
\label{legal patch}
\end{figure}

The goal of the paper is to prove the following theorem.

\begin{Thm}\label{main result monotile}
The monotile in Figure \ref{the tile} is aperiodic. That is, there is a tiling of the plane using only direct isometries of the monotile that satisfies \textbf{R1} and \textbf{R2}, and any such tiling $T$ is nonperiodic.
\end{Thm}

The proof of Theorem \ref{main result monotile} is essentially the contents of the rest of this paper. Since the proof is quite involved, we now provide a brief sketch. In Section \ref{tile properties}, we identify two classes, $\CC_0$ and $\CC_1$, of tilings satisfying \textbf{R1} and \textbf{R2}. We prove that both classes are non-empty and only contain nonperiodic tilings. To prove $\CC_0$ contains only nonperiodic tilings, we use a clever construction of Socolar and Taylor to build tilings satisfying \textbf{R1}, but not necessarily \textbf{R2}. We then build a tiling in $\CC_0$ by recursively constructing a spiral fixed point of R2-trees about the origin. We prove that $\CC_1$ is non-empty and that every tiling is nonperiodic by building all possible tilings in the class. The key to this construction is Lemma \ref{no infinite triangles}, which shows that no legal tiling can contain an infinite R1-triangle (two R1-rays connected by an R1-corner). In Section \ref{tile uniqueness}, we show that \textbf{R2} rules out two patterns of R2-trees, which we call R2-cycles and R2-anticycles. These patterns are exactly those that are formed between R1-triangles when they are arranged into a periodic lattice, as illustrated in Figure \ref{periodic pattern}. A nice consequence of ruling these patterns out is that the union of R2-trees in any tiling is always a connected \emph{tree}.  We are left to show that every tiling satisfying \textbf{R1} and \textbf{R2} must be in $ \CC_0 $ or $ \CC_1 $. This is achieved by proving that the absence of R2-cycles and R2-anticycles implies that every tiling fits into one of these two classes.

After proving Theorem \ref{main result monotile} we discuss the continuous hull of tilings arising from our monotile. We show that all tilings in the continuous hull satisfy \textbf{R1} and a weak version of \textbf{R2}.

Our use of dendrites in constructing the rules for our monotile was motivated by growth in quasicrystals. Several papers hypothesise that dendritic structures in molecules, particularly in soft-matter quasicrystals, are the mechanism that force nonperiodicity, see \cite{IKG,ZULPDH}. In \cite{ZTX}, magnetic dipole-dipole coupling of quasicrystals is mentioned, which can also be modelled by our monotile as indicated in Figure \ref{the tile}. According to the recent survey paper of Steurer \cite{Ste}, one of the most pressing questions in quasicrystal theory is understanding how they form, and when they grow periodically and quasiperiodically. The dendrite rules in this paper show that dendritic growth can lead to aperiodic tile sets.

\begin{Acknowledgements}
We are grateful to Michael Baake, Franz G\"{a}hler, Chaim Goodman-Strauss, Jamie Walton, and Stuart White for their helpful comments and mathematical insights.
\end{Acknowledgements}

\section{Classes of tilings arising from the monotile}\label{tile properties}

In this section, we consider two classes of tilings satisfing \textbf{R1} and \textbf{R2}. We show that each class is non-empty and only contains nonperiodic tilings. In the following section, we show that these classes exhaust the possible tilings that can be constructed from our monotile.

We begin with a closer look at the R1-triangles. Notice that the small R1-curves only occur at angle $\pi/3$, so any R1-triangle must be equilateral. We also observe that the R1-curves can only give rise to nested R1-triangles, an infinite R1-triangle (two infinite R1-rays connected by an R1-curve) or a bi-infinite R1-line. Let us denote the length of a finite R1-triangle by the number of tiles comprising the straight R1-line segment of any given side, as in Figure \ref{R1-triangles}.
\begin{figure}[ht]
\begin{center}
\begin{tikzpicture}
\node (tilea) at (-1,0) [label=below:{Length $0$}] {
\includegraphics[width=0.08\textwidth]{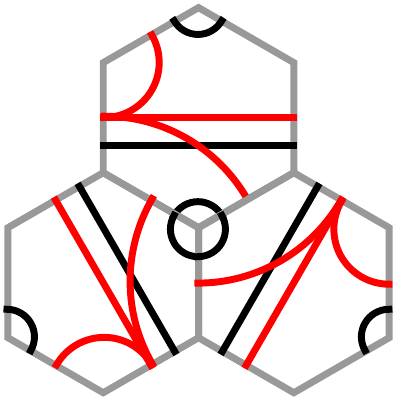}
};

\node (tileb) at (1.5,0) [label=below:{Length $1$}] {
\includegraphics[width=0.12\textwidth]{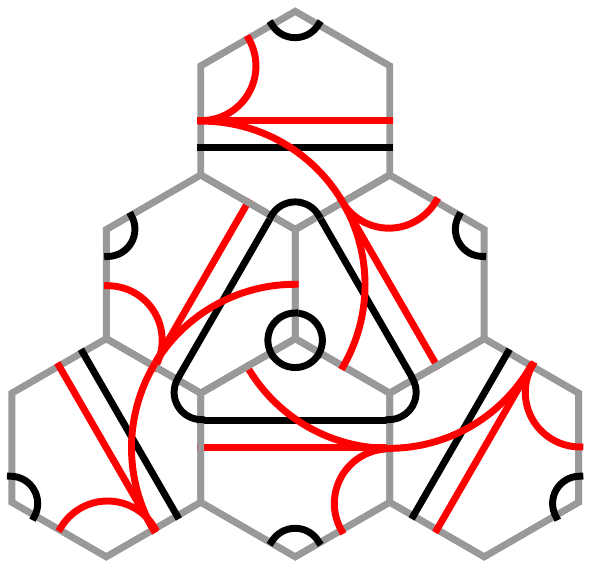}
};

\node (tilec) at (5,0) [label=below:{Length $3$}] {
\includegraphics[width=0.2\textwidth]{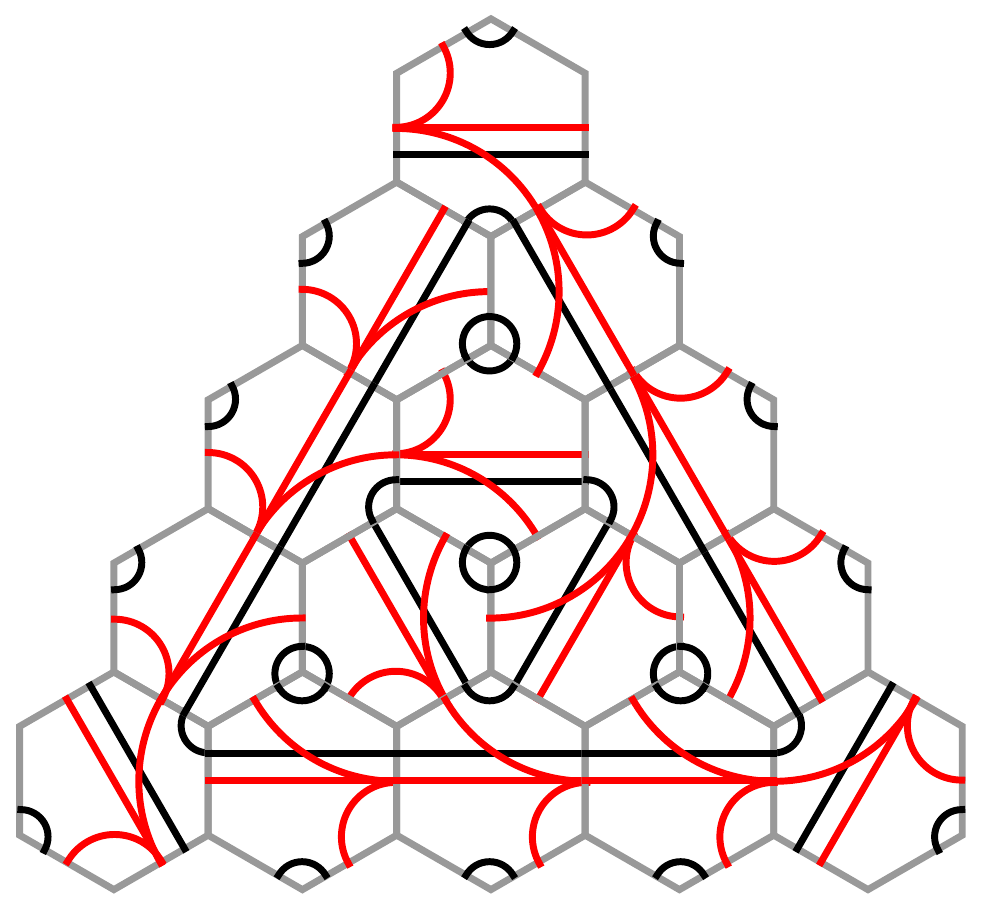}
};

\node (tiled) at (10,0) [label=below:{Length $7$}] {
\includegraphics[width=0.36\textwidth]{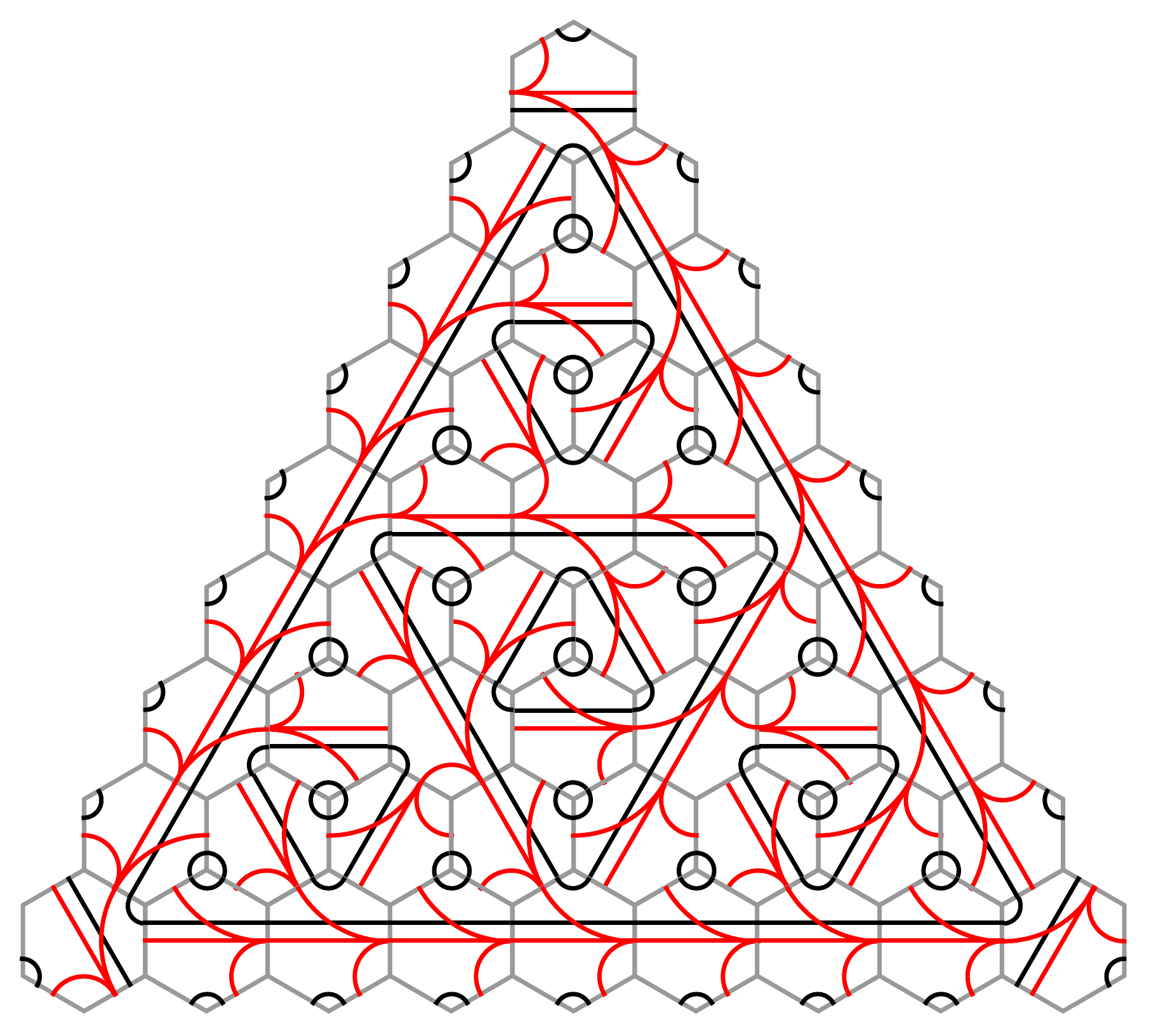}
};
\end{tikzpicture}
\end{center}
\caption{R1-triangles have length $2^n-1$ for $n=0,1,2,\ldots$.}
\label{R1-triangles}
\end{figure}

The following lemma is easily deduced from \textbf{R1} and the geometry of the R1-curves.

\begin{lemma}\label{triangle sizes}
Any nested R1-triangle has length $2^n -1$ for some $n=0,1,2,\ldots$.
\end{lemma}

\begin{definition}\label{classes}
Let $\CC$ denote the collection of tilings whose prototile set is the monotile from Figure \ref{the tile} satisfying \textbf{R1} and \textbf{R2}. Consider the subcollections of $\CC$ defined by the properties:
\begin{description}
\item[$\CC_0$] if the corners of a pair of R1-triangles meet at a common tile in $T$, then these R1-triangles have the same length;
\item[$\CC_1$] $T$ contains a bi-infinite R1-line.
\end{description}
\end{definition}

We first consider the collection $\CC_0$. Let us introduce a convention that will be used in the proof of the following proposition. Define $R_\theta$ to be the rotation operator that rotates a patch counterclockwise around the origin by $\theta$.

\begin{prop}\label{existence}
The collection $\CC_0$ is non-empty and only contains nonperiodic tilings.
\end{prop}

\begin{figure}[ht]
\[
\begin{tikzpicture}
\node (grid_0) at (0,0) [label=below:{zeroth R1-grid}] {\includegraphics[width=4cm,angle=0]{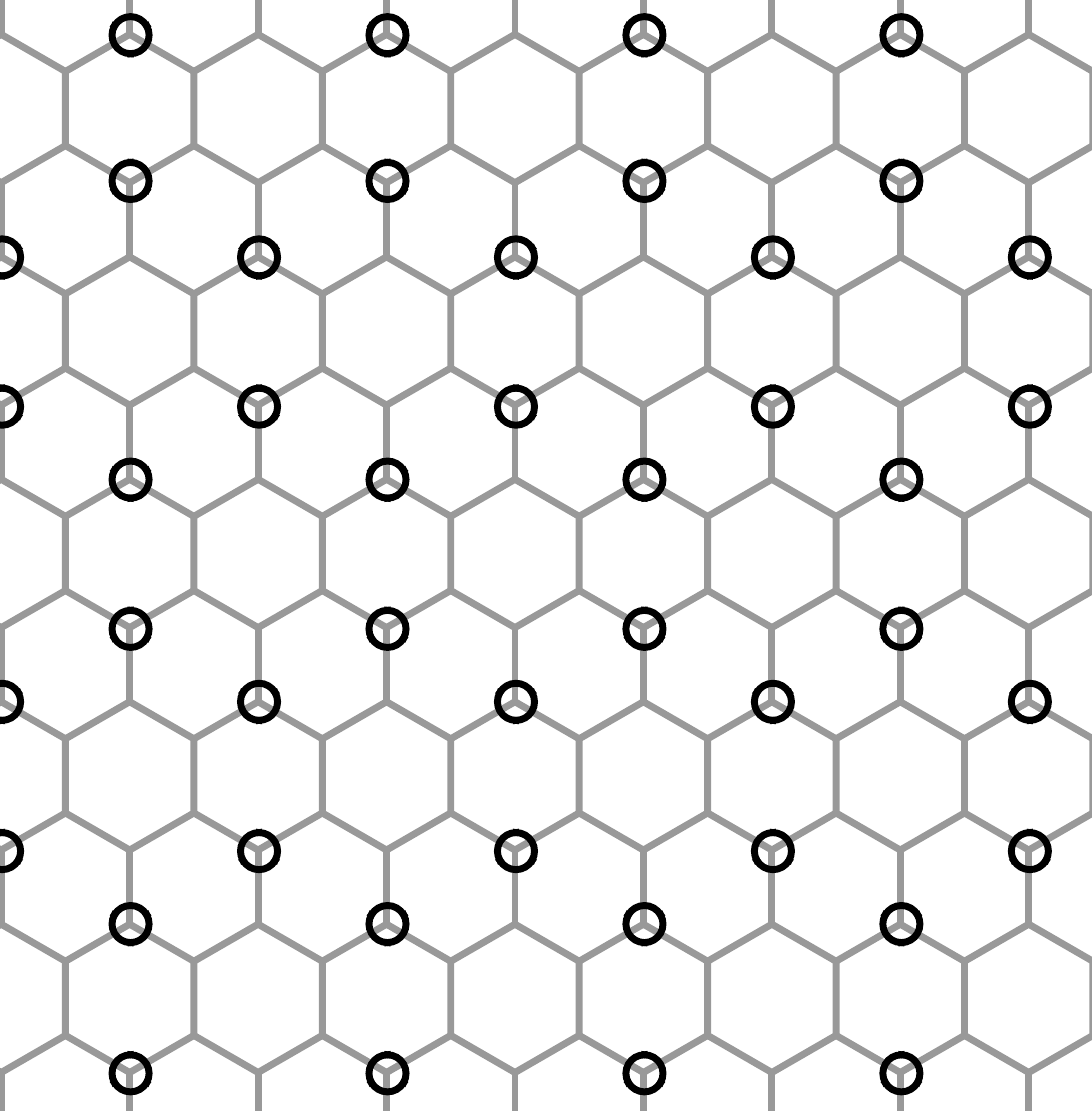}};
\node (grid_1) at (5.5,0) [label=below:{first R1-grid}] {\includegraphics[width=4cm,angle=0]{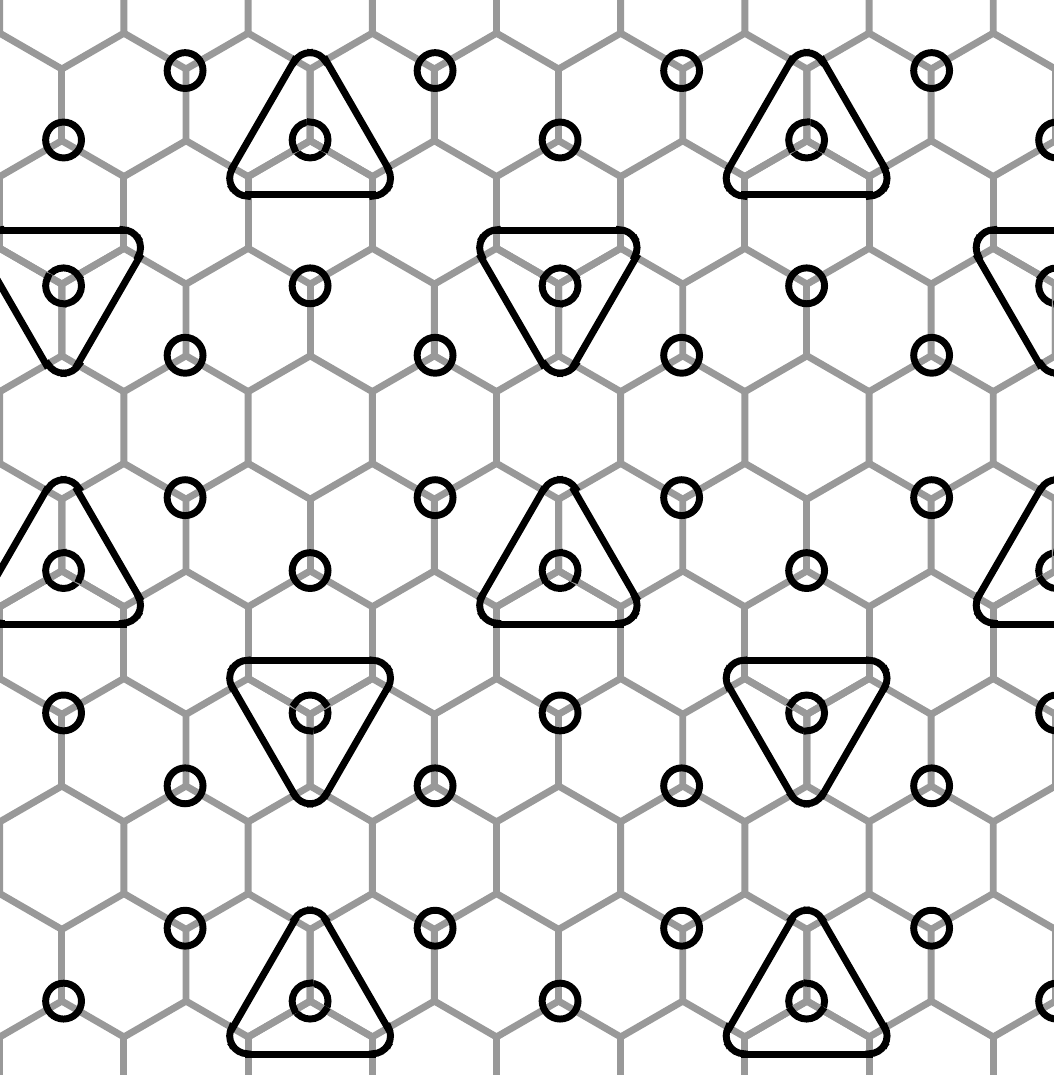}}
	edge[<-] (grid_0);	
\node (grid_n) at (11,0) [label=below:{$n^{th}$ R1-grid}] {\includegraphics[width=4cm,angle=0]{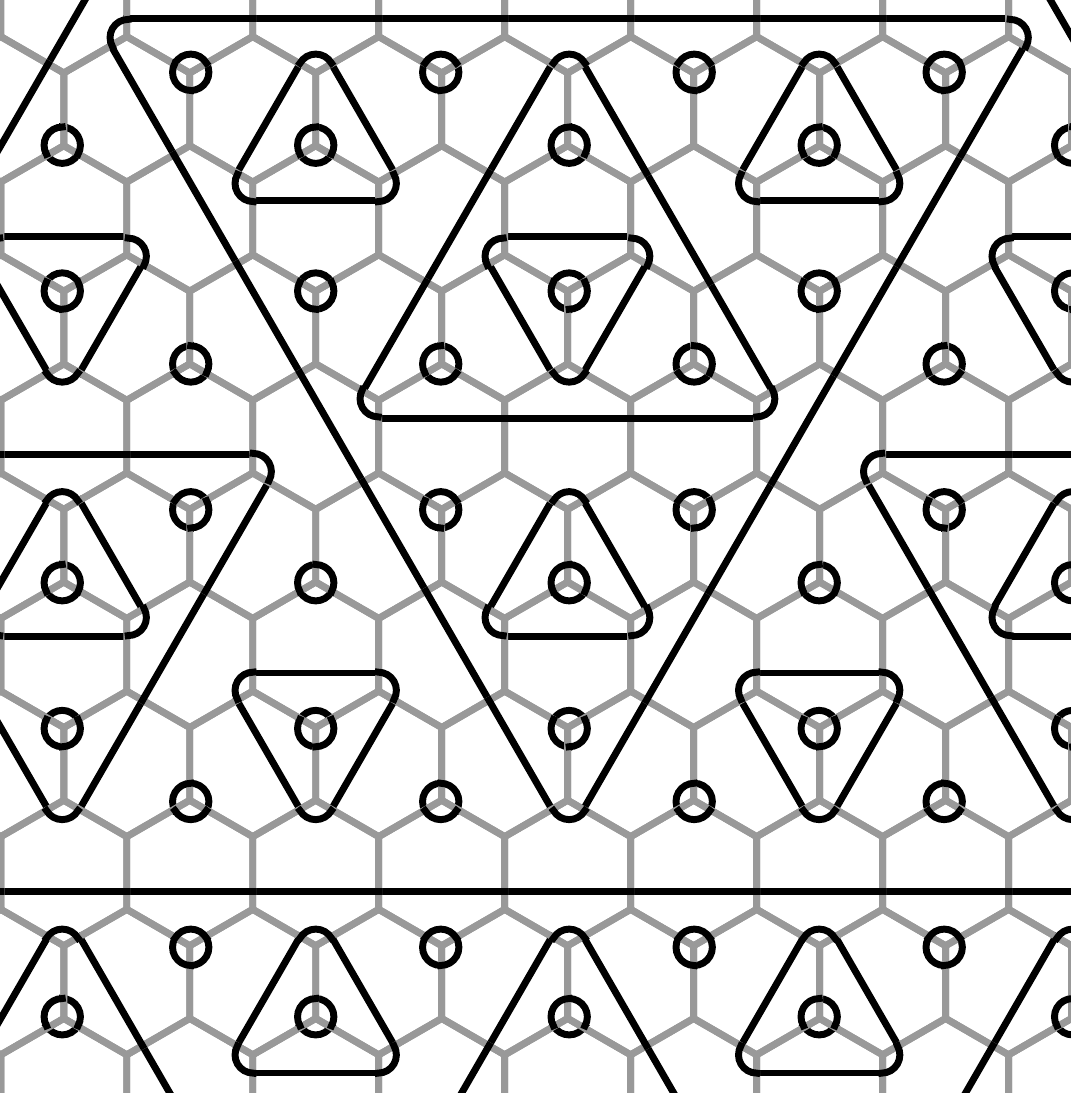}}
	edge[<-,dashed] (grid_1);
\end{tikzpicture}
\]
\caption{The interlaced honeycombs of R1-triangles in the Socolar--Taylor construction}
\label{grids}
\end{figure}

\begin{proof}
To see that $ \CC_0 $ contains only nonperiodic tilings, we appeal to a construction of Socolar and Taylor \cite[Theorem 1]{ST}, which we now summarise. Tilings satisfying \textbf{R1} are produced by adding markings to tiles that form successively larger hexagonal grids of interlaced R1-triangles, see Figure \ref{grids} for a pictorial representation of their construction. To force these honeycomb lattices of length $2^n-1$ R1-triangles, Socolar and Taylor use their second local rule to deduce the condition that all R1-triangles whose corners meet at a common tile have the same size. Since we have restricted ourselves to tilings in $\CC_0$, this condition is one of our hypotheses. The honeycomb lattices of R1-triangles have no largest translational periodicity constant, so that all of the infinite tilings produced must be nonperiodic.

\begin{figure}[ht]
\begin{center}
\scalebox{0.98}{
\begin{tikzpicture}
\node (tilea) at (-1,0) {
\includegraphics[scale=0.33]{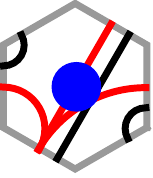}
};
\node at (-1,-0.7) {$P_0$};

\node (tileb) at (0.5,0) {
\includegraphics[scale=0.33]{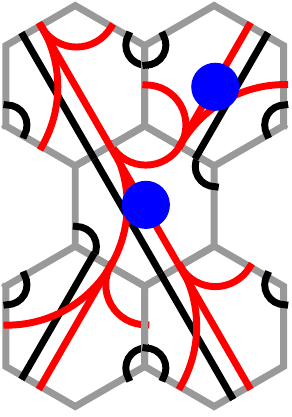}
};
\node at (0.5,-1.0) {$P_1$};

\node (tilec) at (3,0) {
\includegraphics[scale=0.33]{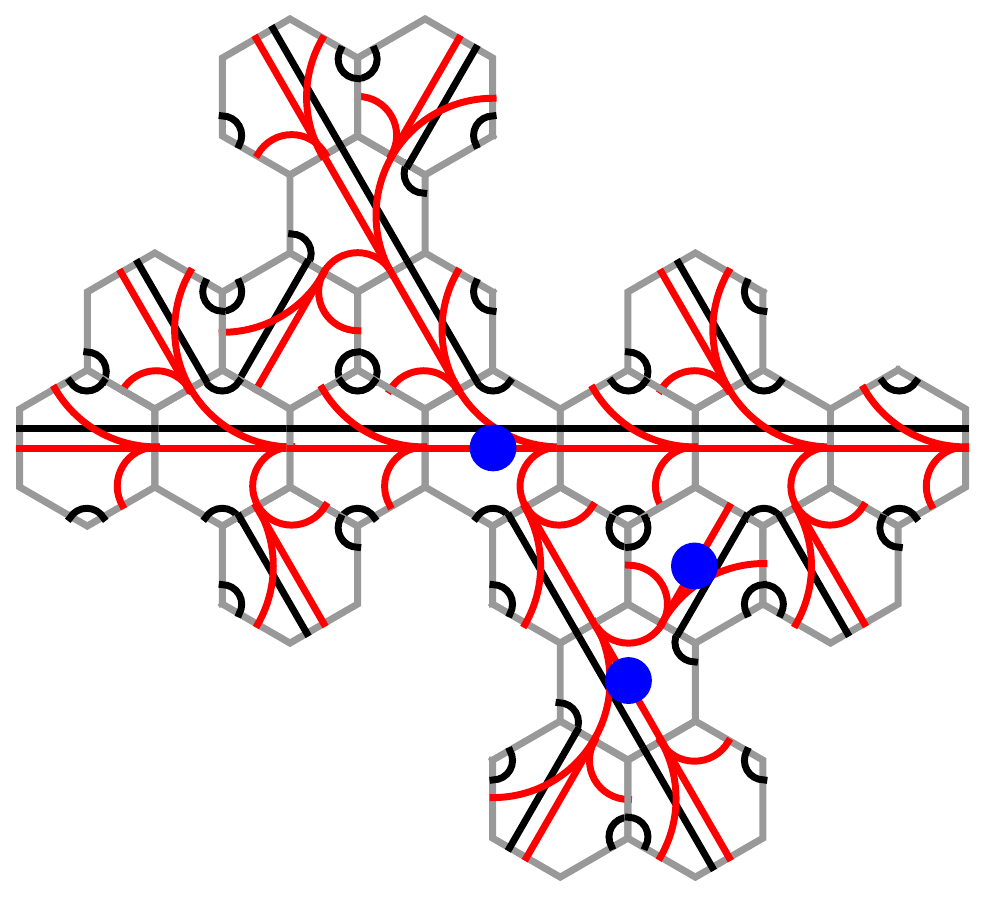}
};
\node at (2.5,-1.0) {$P_2$};

\node (tiled) at (1.5,-5.25) {
\includegraphics[scale=0.33]{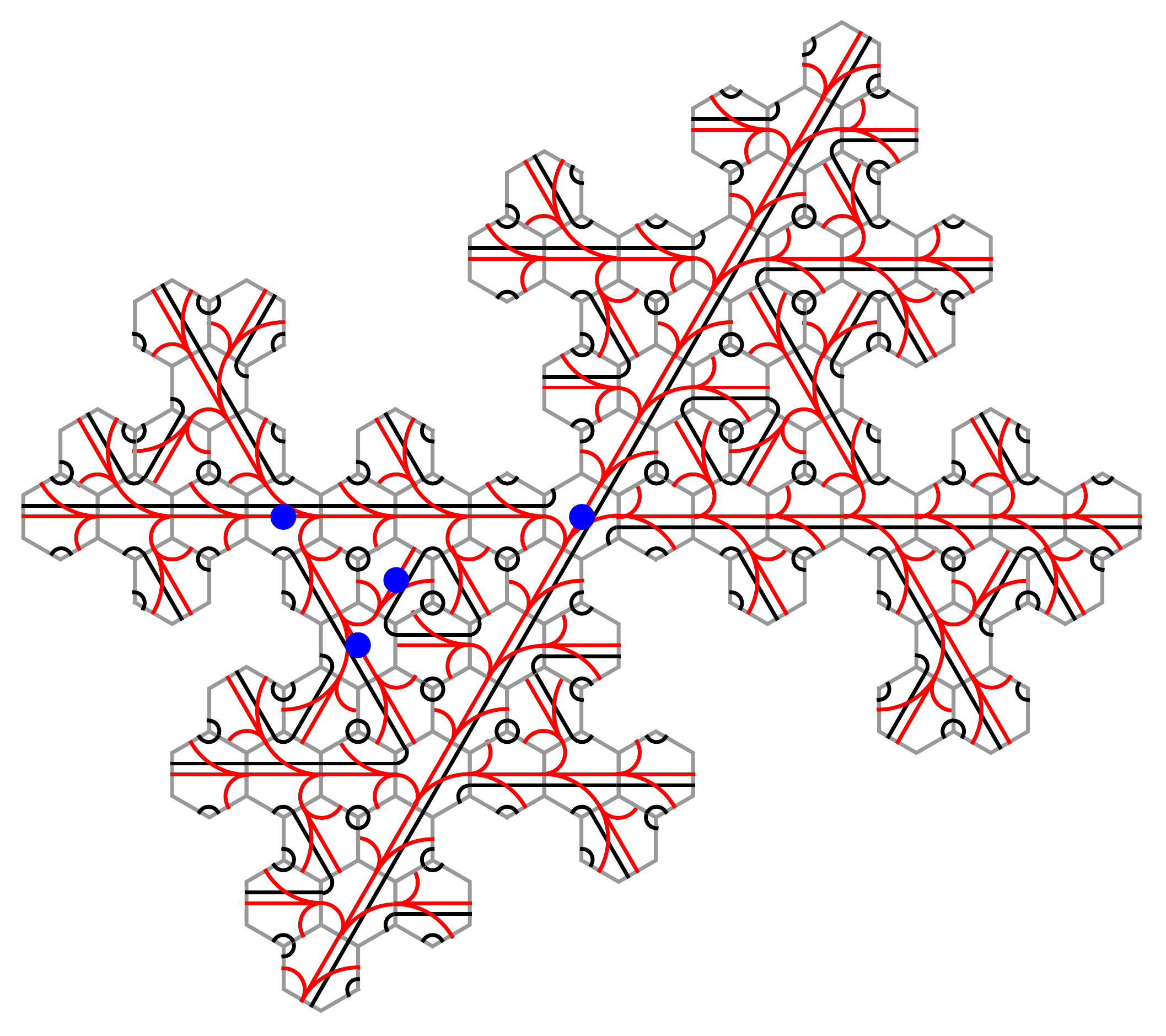}
};
\node at (2.5,-6.4) {$P_3$};

\node (tiled) at (9,-5) {
\includegraphics[scale=0.33]{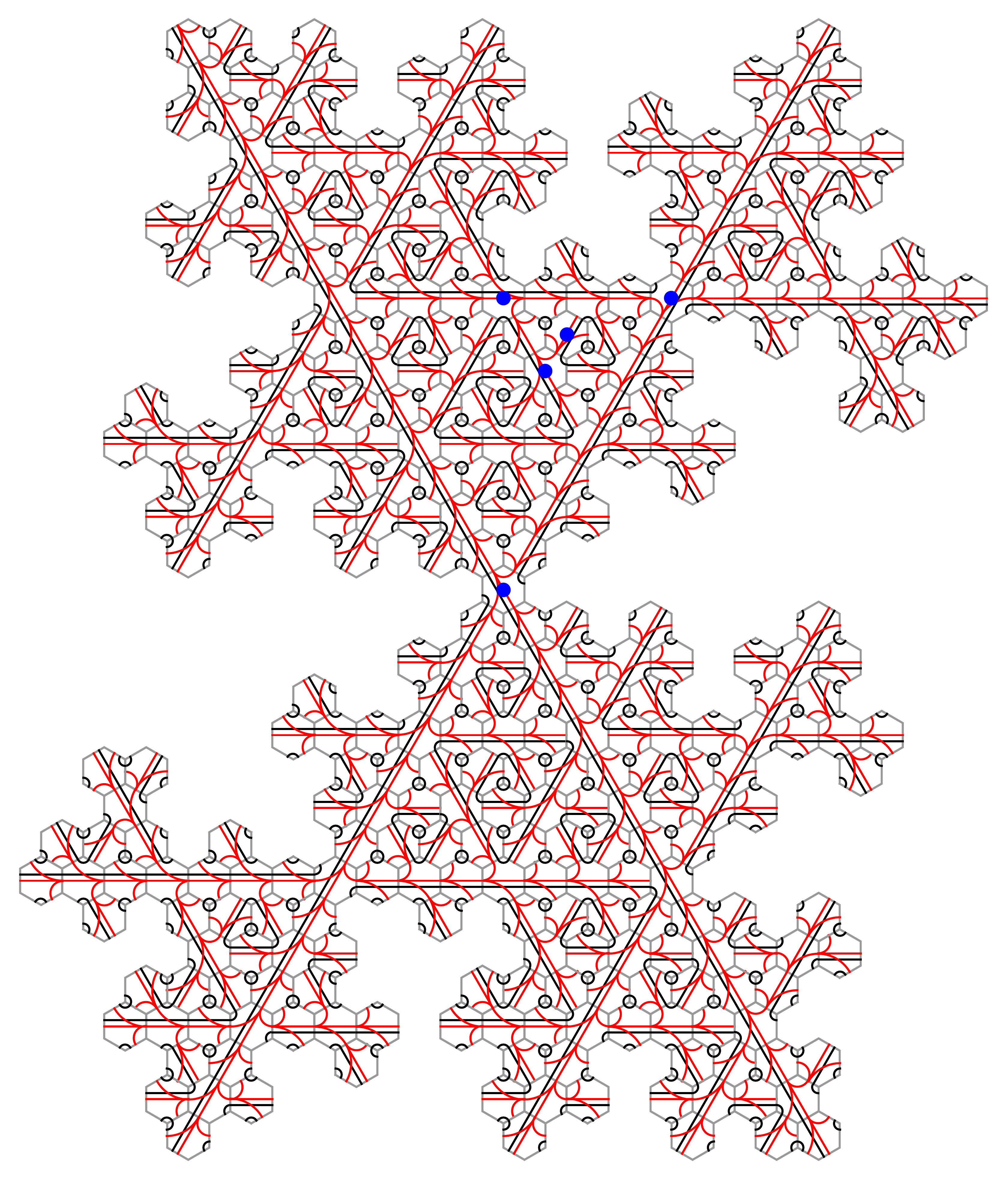}
};
\node at (8,-11.0) {$P_4$};
\end{tikzpicture}}
\end{center}
\caption{Constructing an infinite tiling by piecing together R2-trees}
\label{constructing tiling}
\end{figure}
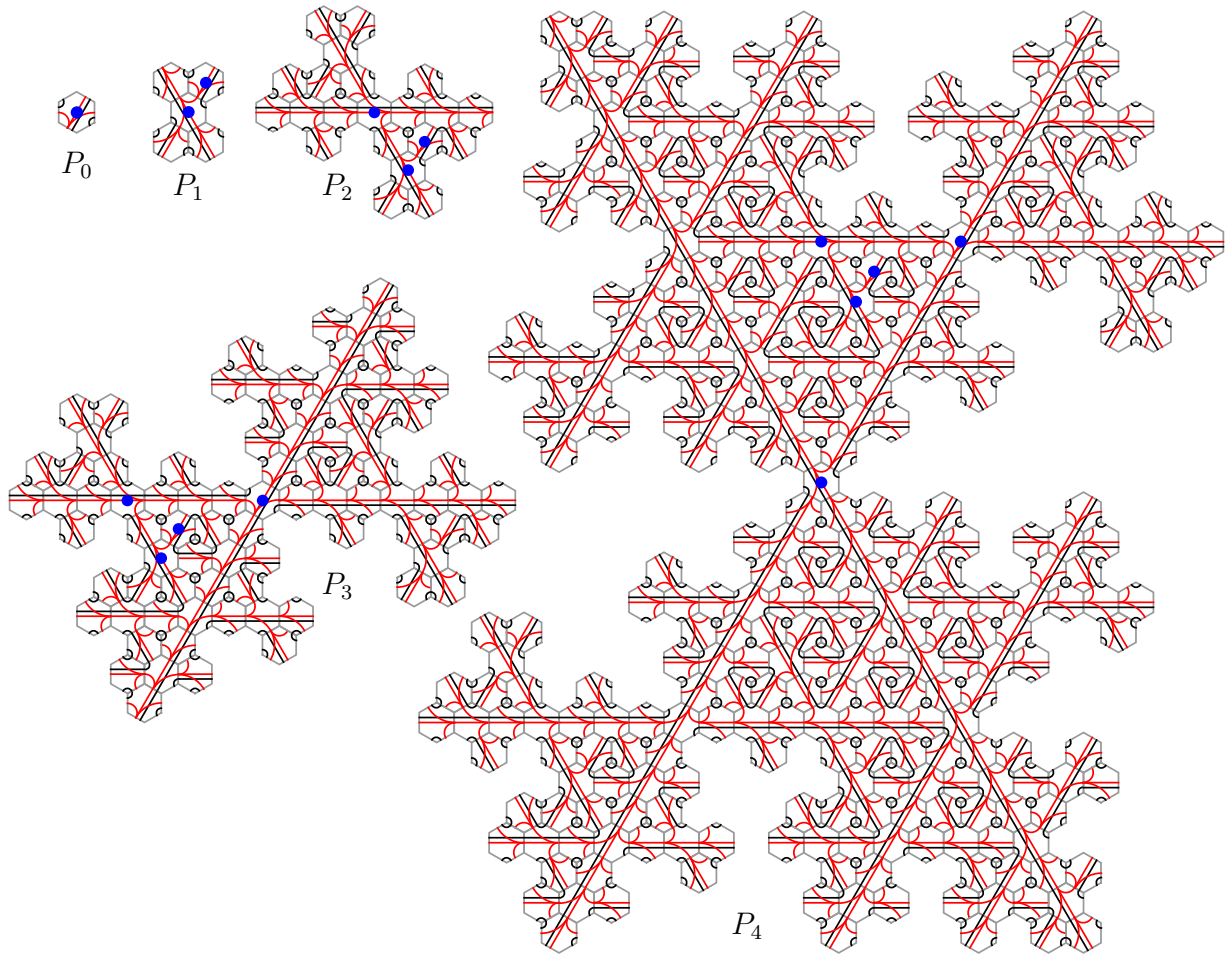

\begin{figure}
\[
\includegraphics[width=0.9\textwidth]{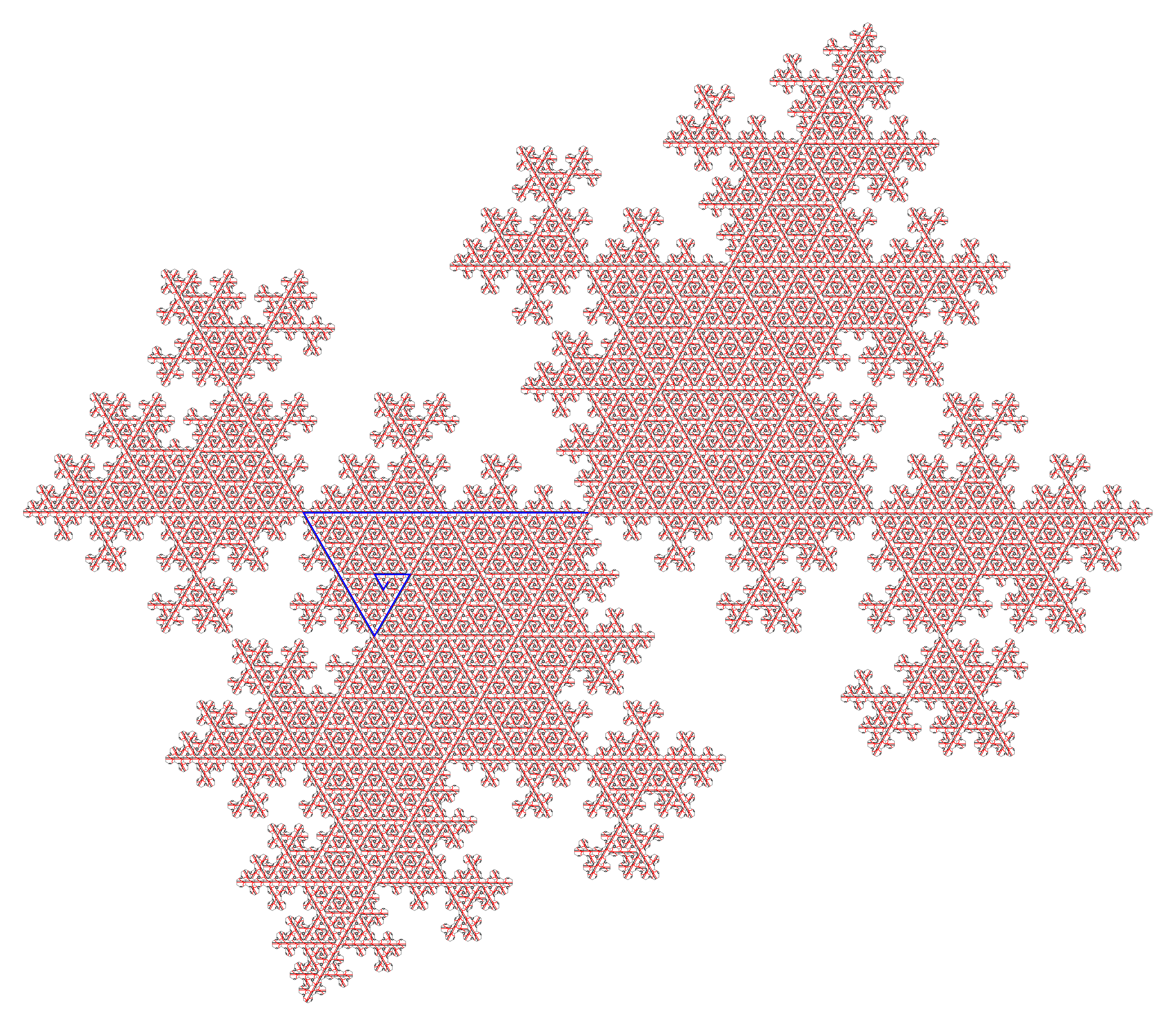}
\]
\caption{Part of the patch $P_6$ with the spiral branch of the R2-tree highlighted}
\label{spiral}
\end{figure}

We are left to show that $\CC_0$ is non-empty. To construct a tiling $T_0 \in \CC_0$, we recursively define patches $P_n$ with three key properties:
\begin{enumerate}
	\item $P_n$ satisfies \textbf{R1} and \textbf{R2},
	\item the patch $P_n$ is a strict subset of $P_{n+1}$, and
	\item as $n$ increases, the patch around the origin in $P_n$ increases exponentially.
\end{enumerate}
The union of patches $P_n$ is a tiling of the plane, see \eqref{constructing T0 formula} below.

To define the recursive algorithm, we carefully look at how patches grow with respect to \textbf{R2}. The first few steps of our algorithm appear in Figure \ref{constructing tiling}, which should help decipher the recursive definition below. Essentially, the patch $P_n$ is constructed from $P_{n-1}$ by gluing $P_{n-1}$ and three direct isometries of $P_{n-1}$ together by a single connecting tile in the centre of the new patch $P_n$. These central tiles have centre at a point $x_n$ (explicitly described below), and each such $x_n$ is marked by a dot in Figure \ref{constructing tiling}, which helps to see $P_n$ in $P_{n+1}$.

Let $P_0$ be our monotile in exactly the orientation appearing in Figure \ref{the tile}, placed with its centre on the origin. Using polar coordinates $(r,\theta)$, the recursive formula for patch $P_n$ is defined by points
\[
x_0:= (0,0) \quad \text{ and } \quad x_n := \sum_{i=1}^n \left(2^{i-1},4 i \pi/3 \right),
\]
along with patches
\begin{align}
\label{recursive patch}
P_n:= R_{\frac{4n\pi}{3}}\left(P_{0} + x_n\right) &\bigsqcup P_{n-1} \bigsqcup R_{\frac{4\pi}{3}}\left(P_{n-1} - x_{n-1}\right)+x_n+\left(2^{n-1},4n\pi/3-2\pi/3\right)\\
\nonumber
&\bigsqcup R_{\pi}\left(P_{n-1} - x_{n-1}\right)+x_n+\left(2^{n-1},4n\pi/3\right)\\
\nonumber
&\bigsqcup R_{\frac{4\pi}{3}}\left(P_{n-1} - x_{n-1}\right)+x_n+\left(2^{n-1},4n\pi/3+\pi/3\right).
\end{align}
Note that the points $x_n$ appear at each corner of the superimposed spiral in Figure \ref{spiral}.

The method we have used to build $ P_n $ ensures that both \textbf{R1} and \textbf{R2} are satisfied. Moreover, the patches $P_n$ overlap where they intersect, and are space filling in a spiral pattern that successively connects the points $x_n$ around the origin, see the Figure \ref{spiral}. Thus, the union 
\begin{equation}\label{constructing T0 formula}
T_0 := \bigcup_{n=0} P_n
\end{equation}
is a tiling of the plane satisfying both \textbf{R1} and \textbf{R2}. 

To finish the proof, we show that $T_0$ satisfies the defining condition of $\CC_0$. Observe that each patch $P_n$ has two R2-tree straight segments of length $2^n-1$ extending from the tile containing the point $x_n$ at angles $\frac{n\pi}{3}$ and $\frac{(n+3)\pi}{3}$. The recursive definition extending $P_n$ into $P_{n+1}$ ensures that the straight segments of the R2-tree arms terminate at length $2^n-1$, which implies that the lengths of the R1-triangles (realised in $P_{n+3}$) along those arms also have length $2^n-1$, and occur on opposite sides of an R1-line segment. These R1-triangles force all R1-triangles of smaller length to have the same length if their corners meet in a common tile. Thus, the tiling $T_0$ is in $\CC_0$.
\end{proof}

\begin{figure}[ht]
\begin{center}
\begin{tikzpicture}
\node at (0,0) {\includegraphics[width=0.9\textwidth]{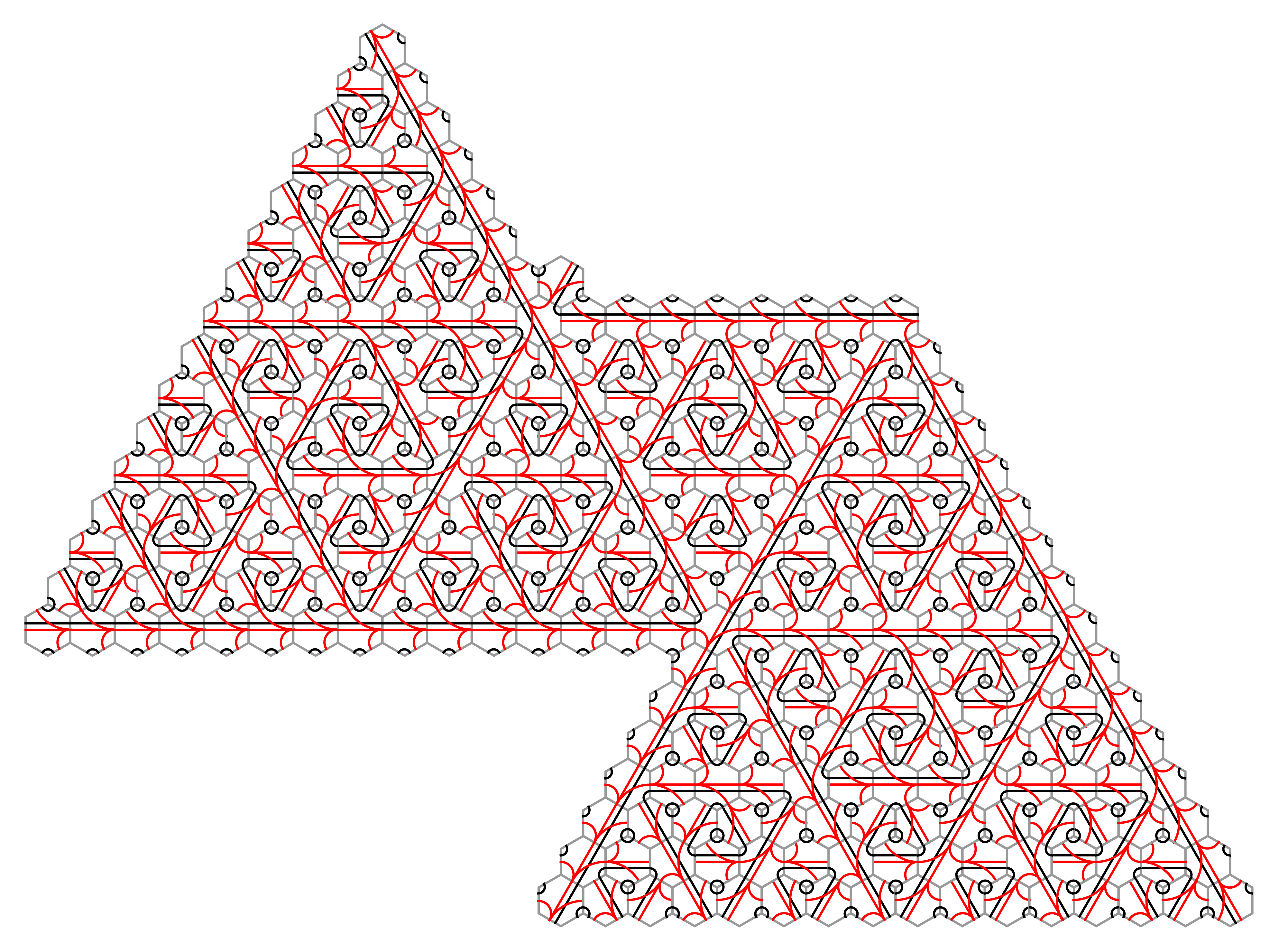}};
\node at (-1.2,-5.4) {$ A $};
\node at (7.1,-5.4) {$ B $};
\end{tikzpicture}
\end{center}
\caption{The pictorial argument for Lemma \ref{no infinite triangles}}
\label{infinite triangle fig}
\end{figure}

We now consider the collection $\CC_1$, but first a lemma.

\begin{lemma}\label{no infinite triangles}
Any tiling satisfying \textbf{R1} and containing an infinite R1-triangle (two R1-rays connected by an R1-corner) does not satisfy \textbf{R2}.
\end{lemma}

\begin{proof}
Suppose $T$ is a tiling satisfying \textbf{R1} and containing an infinite R1-triangle. Let $A$ and $B$ be the respective connected components of the union of R2-trees associated with each side of the infinite R1-triangle, see Figure \ref{infinite triangle fig}. We will argue that $A$ and $B$ cannot be connected.  Notice that the branches of $A$ and $B$ into the interior of the infinite R1-triangle never reach the opposite side of the R1-triangle, and are disjoint. At the infinite R1-triangle corner, one of the trees extends into the corner tile, while the other does not. Let us assume $A$ does not extend. The only remaining connection possible between $A$ and $B$ is along R1-branches of $A$ growing outside the infinite R1-triangle. Every such R2-branch extends along the side of an R1-triangle, and terminates at the corner of the R1-triangle, if the R1-triangle is finite, or extends infinitely if the R1-triangle is infinite. However, as noted above, R2-branches into the interior of an R1-triangle never reach the opposite side of the R1-triangle and are disjoint from any R2-branches extending from the opposite side. It follows that $A$ never meets $B$, and so $T$ cannot satisfy \textbf{R2}.
\end{proof}

\begin{prop}\label{fault line}
The collection $\CC_1$ is non-empty and only contains nonperiodic tilings.
\end{prop}

\begin{figure}[ht]
\[
\includegraphics[width=0.6\textwidth]{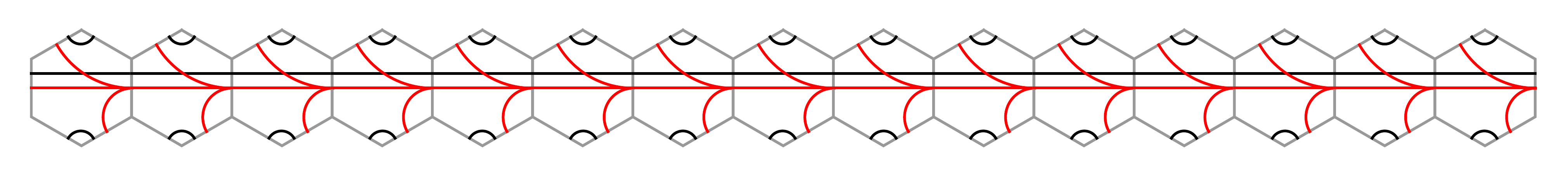}
\]
\caption{The bi-infinite string of tiles forming an R1-line in Proposition \ref{fault line}.}
\label{fault line fig}
\end{figure}

\begin{proof}
We will construct a tiling in $\CC_1$, starting with a bi-infinite string of tiles forming an R1-line. Note that the union of R2-trees along this string is connected. In order to simplify the argument, we fix the orientation of this string of tiles, as in Figure \ref{fault line fig}, and will refer to the top, bottom, left and right as per the orientation depicted. Our construction will produce every possible tiling in $\CC_1$ up to direct isometry.

We begin by adding tiles above the R1-line. Lemma \ref{no infinite triangles} implies that we can never add an infinite R1-triangle with a corner meeting the R1-line. So every R1-triangle meeting the R1-line must have length $2^n-1$ for some $n \in \N$ by Lemma \ref{triangle sizes}.

\begin{figure}[ht]
\[
\includegraphics[width=0.6\textwidth]{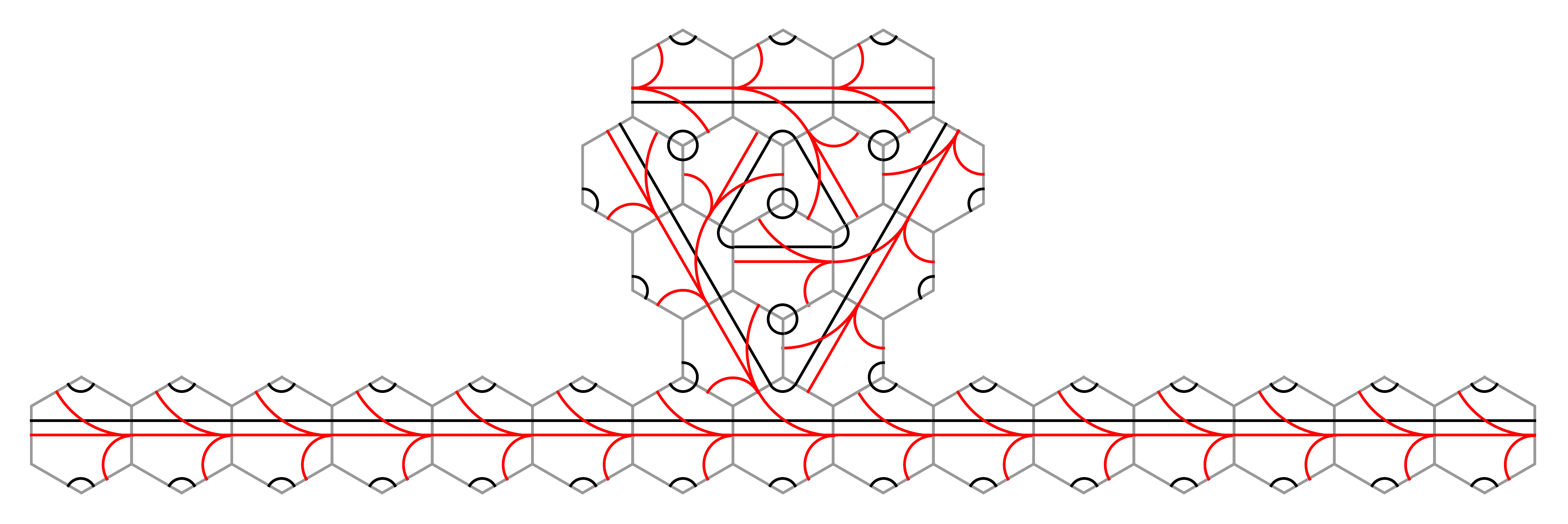}
\]
\caption{Adding tiles to attach a length $3$ R1-triangle to the R1-line}
\label{fault line fig 2}
\end{figure}

\begin{figure}[ht]
\[
\includegraphics[width=0.6\textwidth]{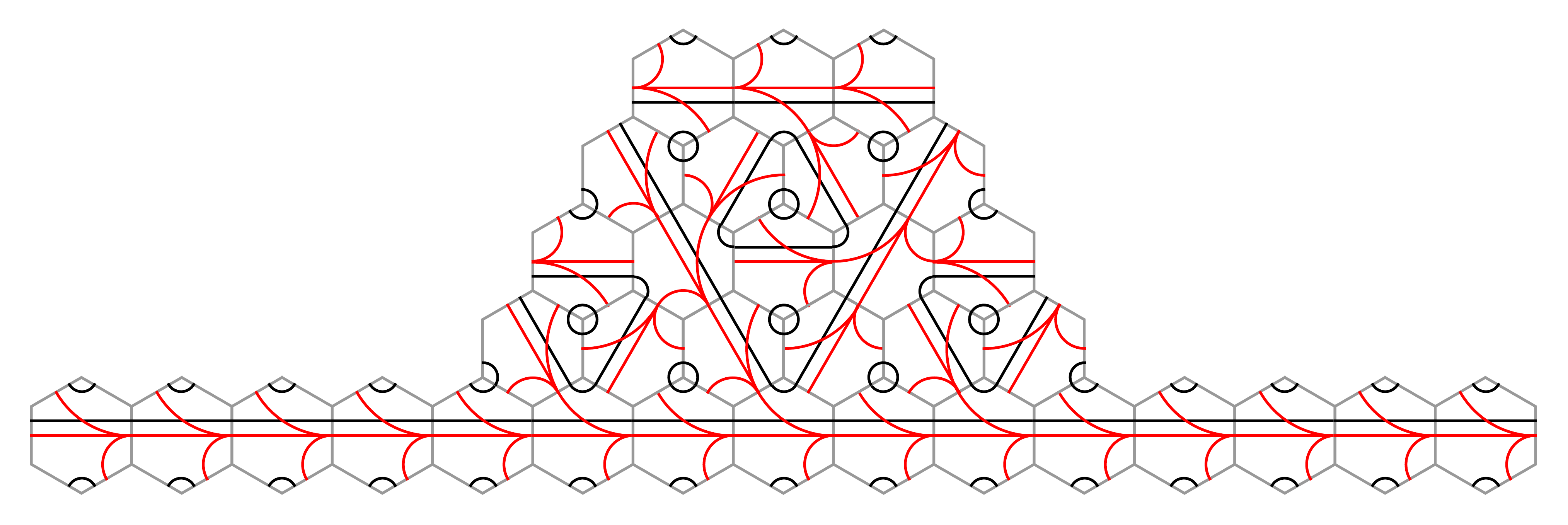}
\]
\caption{Adding the forced tiles between the R1-triangle and the R1-line}
\label{fault line fig 3}
\end{figure}

Suppose we add an R1-triangle of length $2^{n}-1$ whose bottom corner is the R1-corner of a tile in the R1-line, as in Figure \ref{fault line fig 2} with $n=2$. We note that the union of R2-trees is no longer connected, but this will be rectified shortly. Between the length $ 2^n - 1 $ R1-triangle and the R1-line, R1-triangles of all shorter legal lengths are forced, as depicted in Figure \ref{fault line fig 3}.

\begin{figure}[ht]
\[
\includegraphics[width=0.9\textwidth]{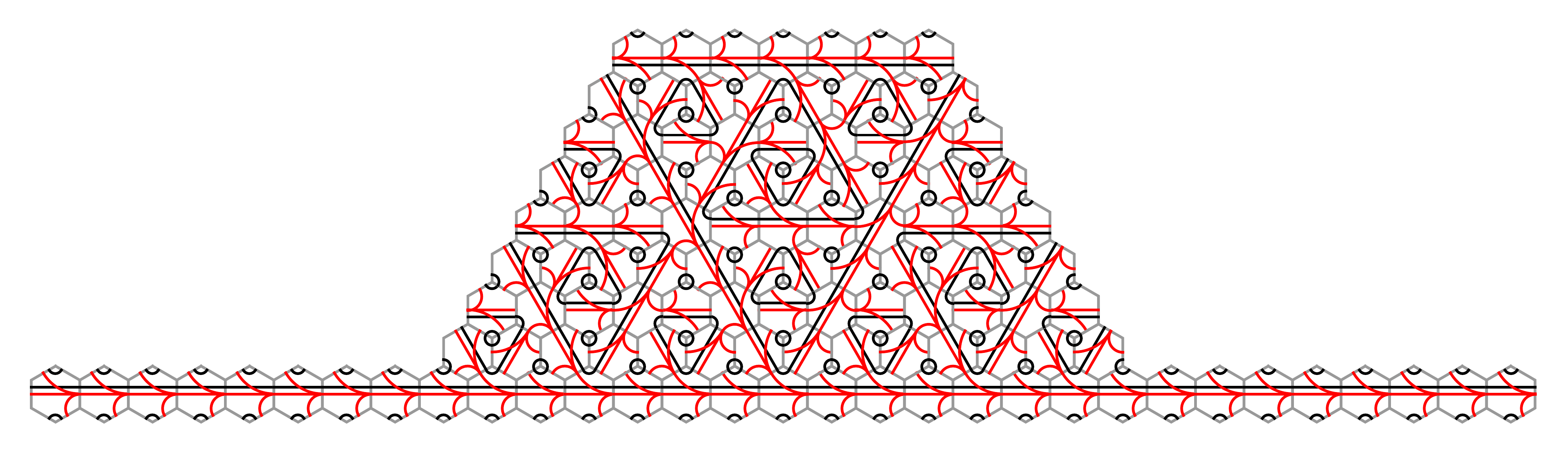}
\]
\caption{A length $7$ R1-triangle is forced to one side of a length $3$ R1-triangle}
\label{fault line fig 4}
\end{figure}

We now consider the possible tiles we may add above the tiles occurring $ 2^n $ tiles to the left or right of the bottom corner of the $ 2^n - 1 $ R1-triangle along the R1-line. The geometry of the situation forces one of these tiles to be the corner of a length $2^{n+1}-1$ R1-triangle. As above, this triangle forces the corner of a length $ 2^{n+2} - 1 $ triangle to the left or right of the bottom corner of the $ 2^{n+1} - 1 $ R1-triangle along the R1-line. This process of adding successively larger R1-triangles whose corners occur at distance $2^{n+j}$ tiles along the R1-line from its successor carries on ad infinitum. Lemma \ref{no infinite triangles} implies that tilings in $\CC$ cannot contain infinite triangles, so we must change the direction of our choice an infinite number of times so that every tile on the R1-line contains the corner of some finite length R1-triangle. Moreover, these triangles are forced to occur periodically. That is, placing a length $2^m-1$ R1-triangle of tiles at tile position $l$ on the R1-line yields a length $2^m-1$ R1-triangle of tiles at positions $l + k2^{m+1}$ for all $k \in \Z$, and at least one R1-triangle of length $2^m-1$ appears in any string of tiles along the R1-line of length $2^{m+1}$. This construction leads to a half-plane of tiles that satisfies \textbf{R1}. 

\begin{figure}[ht]
\[
\includegraphics[width=0.9\textwidth]{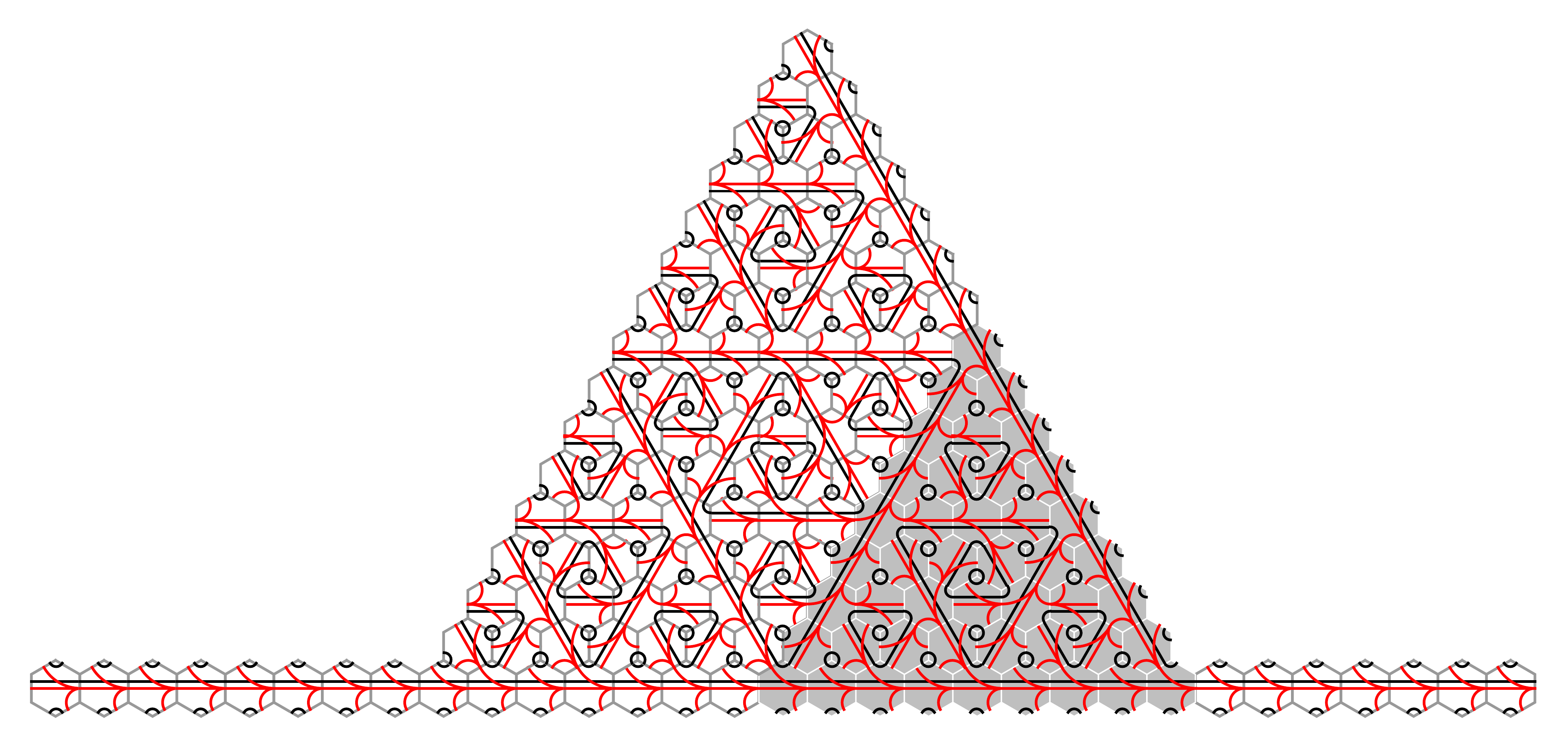}
\]
\caption{To the right of an R1-triangle, the R2-trees of the shaded tiles are connected to the R2-trees along the infinite string}
\label{fault line fig 5}
\end{figure}

We now argue that the half-plane of tiles constructed above also satisfies \textbf{R2}. Indeed, the union of R2-trees along the bi-infinite string of tiles is connected. Given a length $2^m-1$ R1-triangle whose corner meets the R1-line, the union of R2-trees in a triangular arrangement of tiles between its right side and the R1-line is connected to the union of R2-trees along the R1-line, the shaded region in Figure \ref{fault line fig 5} is an example of such a patch. Since there are no infinite R1-triangles, every tile in the upper half-plane is to the right of some R1-triangle whose corner meets the R1-line. Thus, the union of R2-trees in the upper half-plane is connected, so the upper half-plane is a patch satisfying \textbf{R1} and \textbf{R2}.

An analogous argument implies that all tiles in the lower half-plane satisfy \textbf{R1} and \textbf{R2}. Since the upper and lower half-planes intersect along the bi-infinite R1-line, the resulting tiling is in $\CC_1$. Since there are arbitrarily large R1-triangles arranged in interlaced periodic patterns whose corners meet the R1-line, the resulting tiling is nonperiodic, giving the desired result.
\end{proof}

We note that the classes $\CC_0$ and $\CC_1$ have non-trivial intersection. Indeed, if the R1-triangles on opposite sides of the bi-infinite R1-line of a tiling in $\CC_1$ have the same length, then it is also in $\CC_0$.

\section{Proof of Theorem \ref{main result monotile}}\label{tile uniqueness}

We have shown in Propositions \ref{existence} and \ref{fault line}, that the classes $\CC_0$ and $\CC_1$ from Definition \ref{classes} are non-empty and contain only nonperiodic tilings. In this section, we will prove that $\CC=\CC_0 \cup \CC_1$, which will prove Theorem \ref{main result monotile}. The key is to prove that any tiling in $\CC$ that does not belong to $\CC_0$ must have an infinite R1-line through it, so it belongs to $\CC_1$.

\begin{figure}[ht]
\begin{center}
\begin{tikzpicture}
\node (tilea) [label=below:{$n=2$}] at (-1,0) {\includegraphics[scale=0.33]{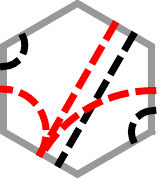}};
\node (tileaa) at (0.9,0) {\includegraphics[scale=0.33,angle=0]{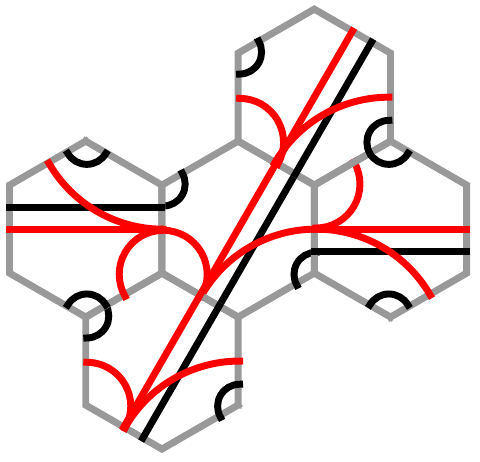}};
\draw[->] (-0.6,0) -- (-0.1,0);
\node (tileb) [label=below:{$n=3$}] at (2.5-3,-3) {\includegraphics[scale=0.33]{tree_no_dot_0}};
\node (tilebb) at (5.2-2.8,-3) {\includegraphics[scale=0.33,angle=0]{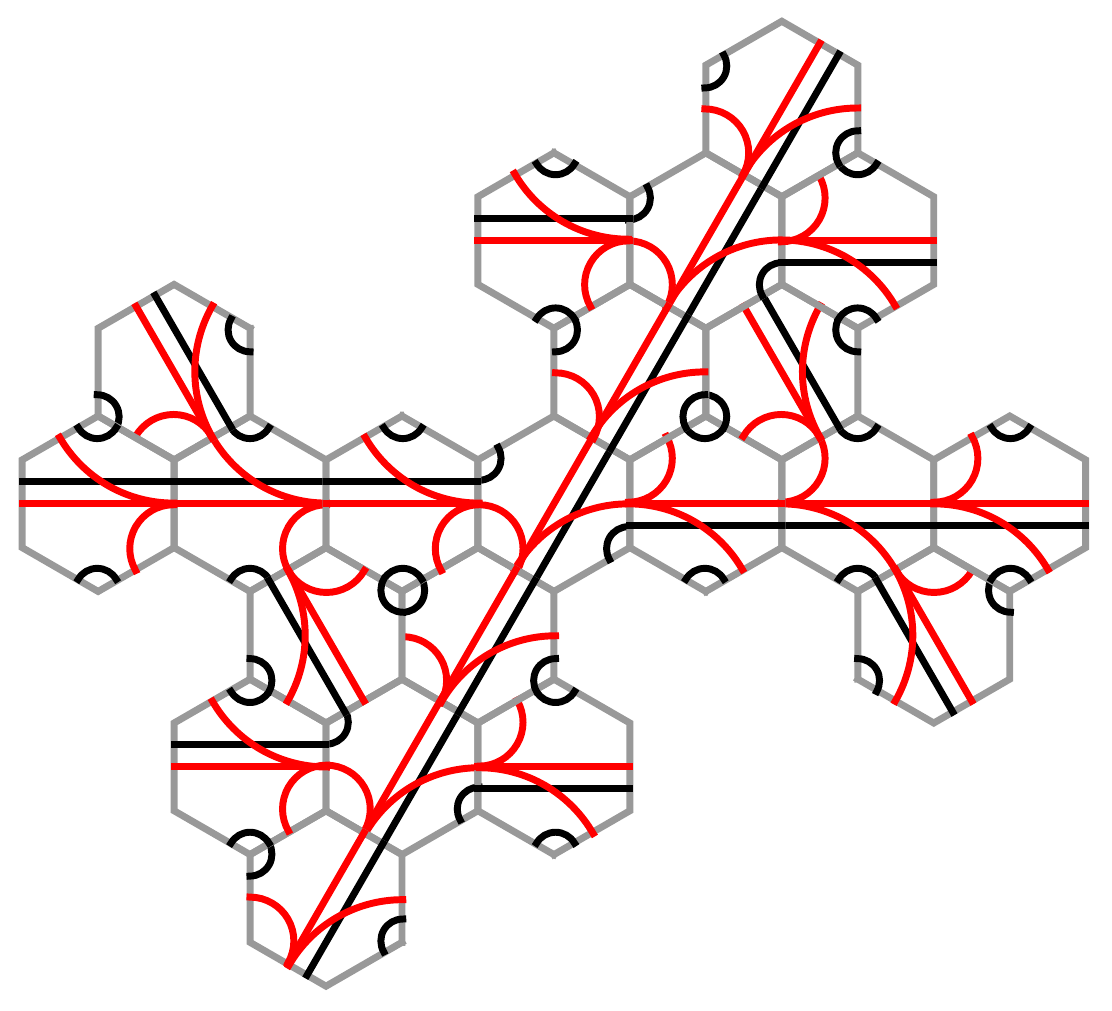}};
\draw[->] (2.9-3,-3) -- (3.4-3,-3);
\node (tilec) [label=below:{$n=4$}] at (5,-1) {\includegraphics[scale=0.33]{tree_no_dot_0}};
\node (tilecc) at (10,-1) {\includegraphics[scale=0.33]{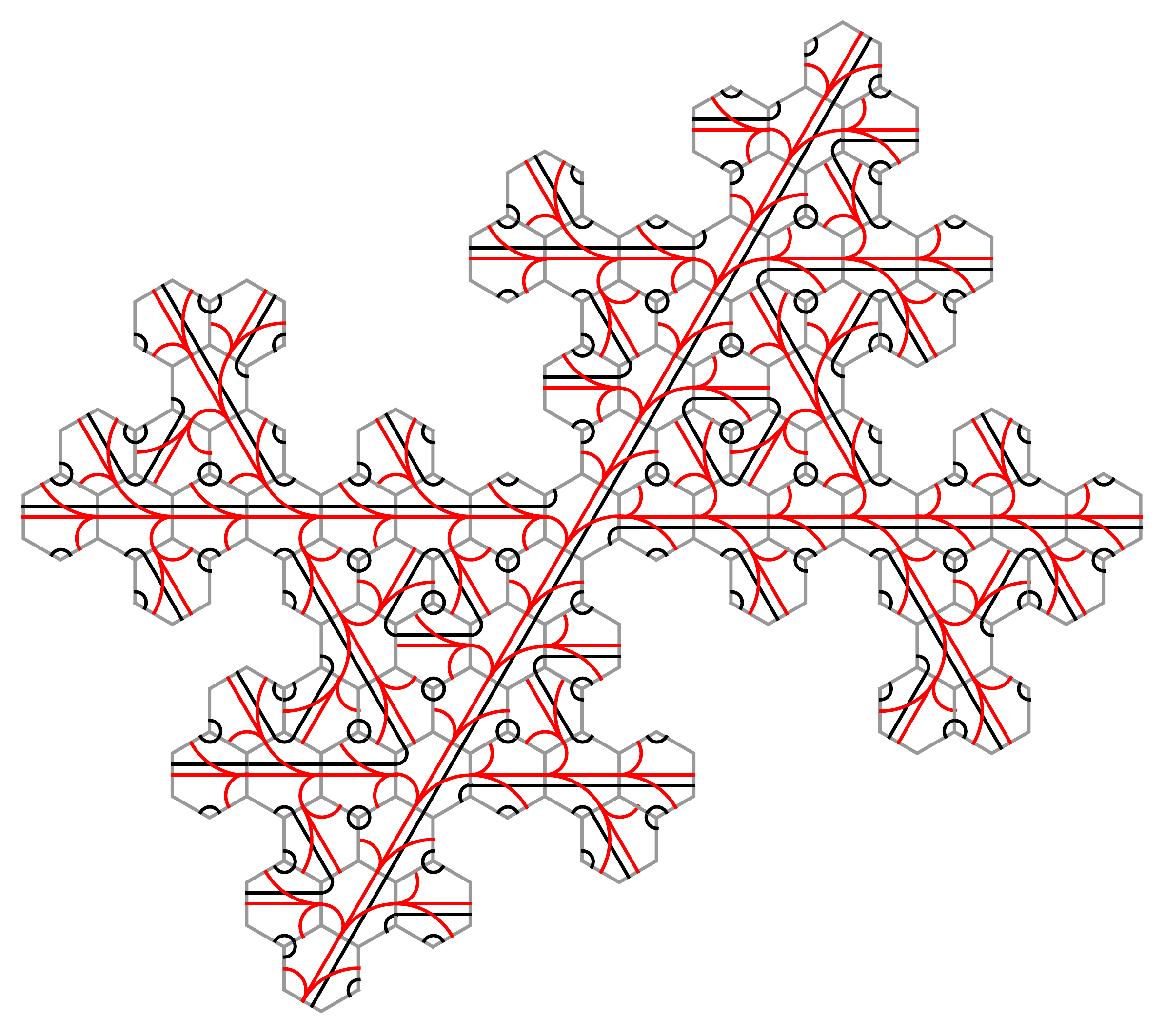}};
\draw[->] (5.4,-1) -- (5.9,-1);
\end{tikzpicture}
\end{center}
\caption{Monotiles with dashed lines represent legal patches of length $2^n-1$}
\label{dotted patch}
\end{figure}

\begin{figure}[ht]
\begin{center}
\begin{tikzpicture}
\node (tilea) {\includegraphics[scale=0.22,angle=0]{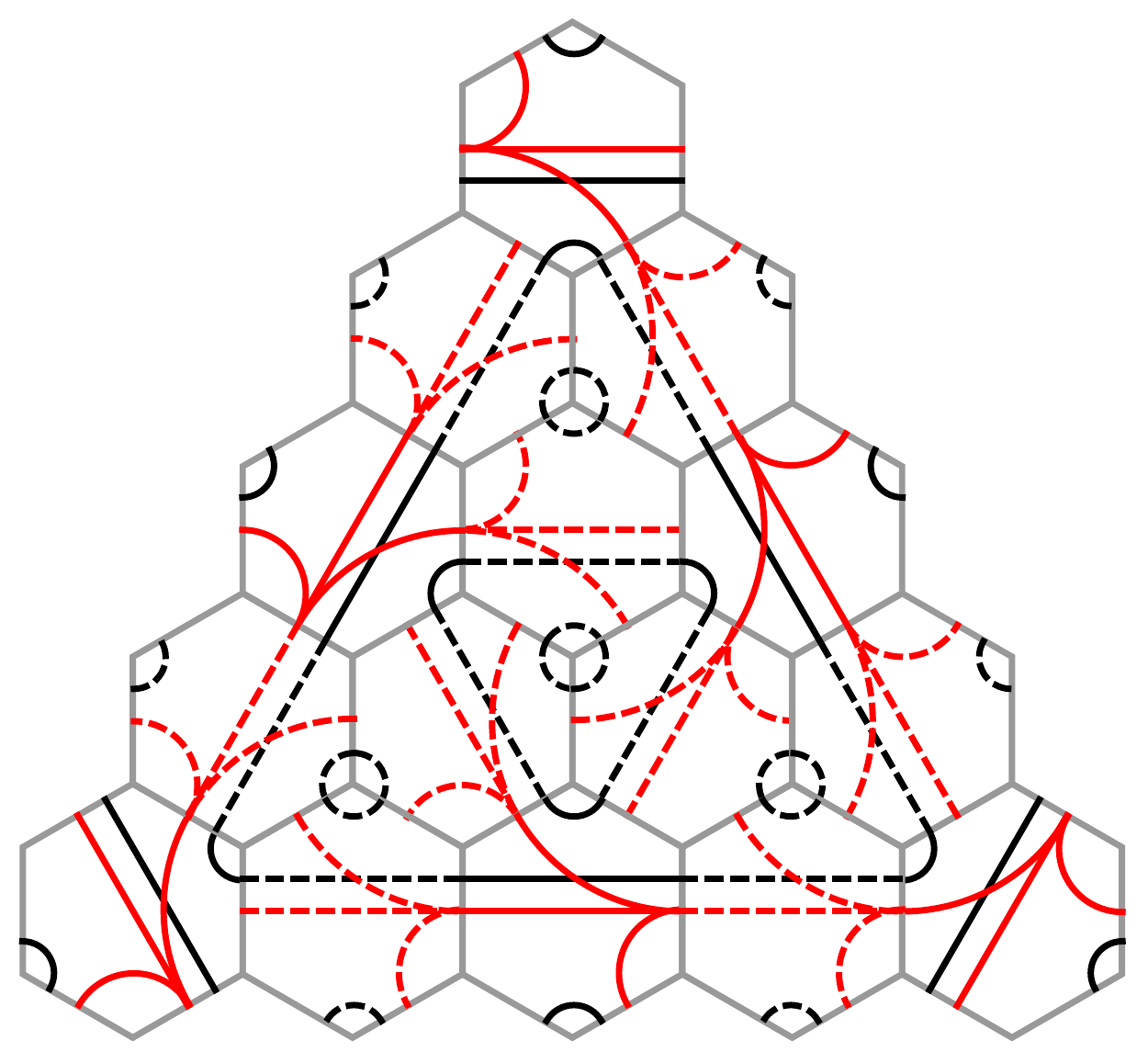}};
\node (tileaa) at (-4,0) {\includegraphics[scale=0.3,angle=0]{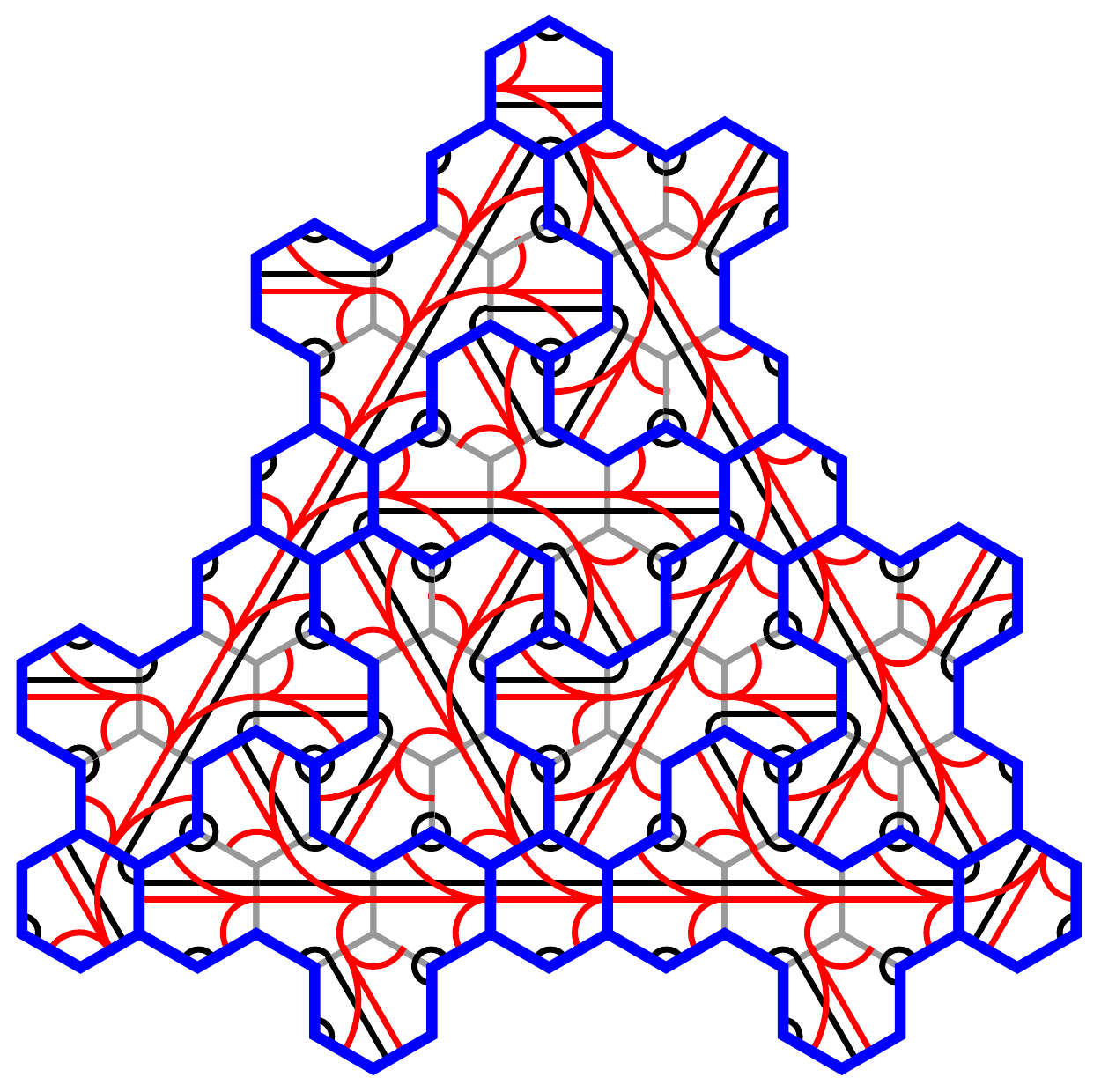}};
\draw[->] (1.1,0) -- node[above] {\scriptsize $n=3$} (2.1,0);
\node (tileaa) at (6,0) {\includegraphics[scale=0.3,angle=0]{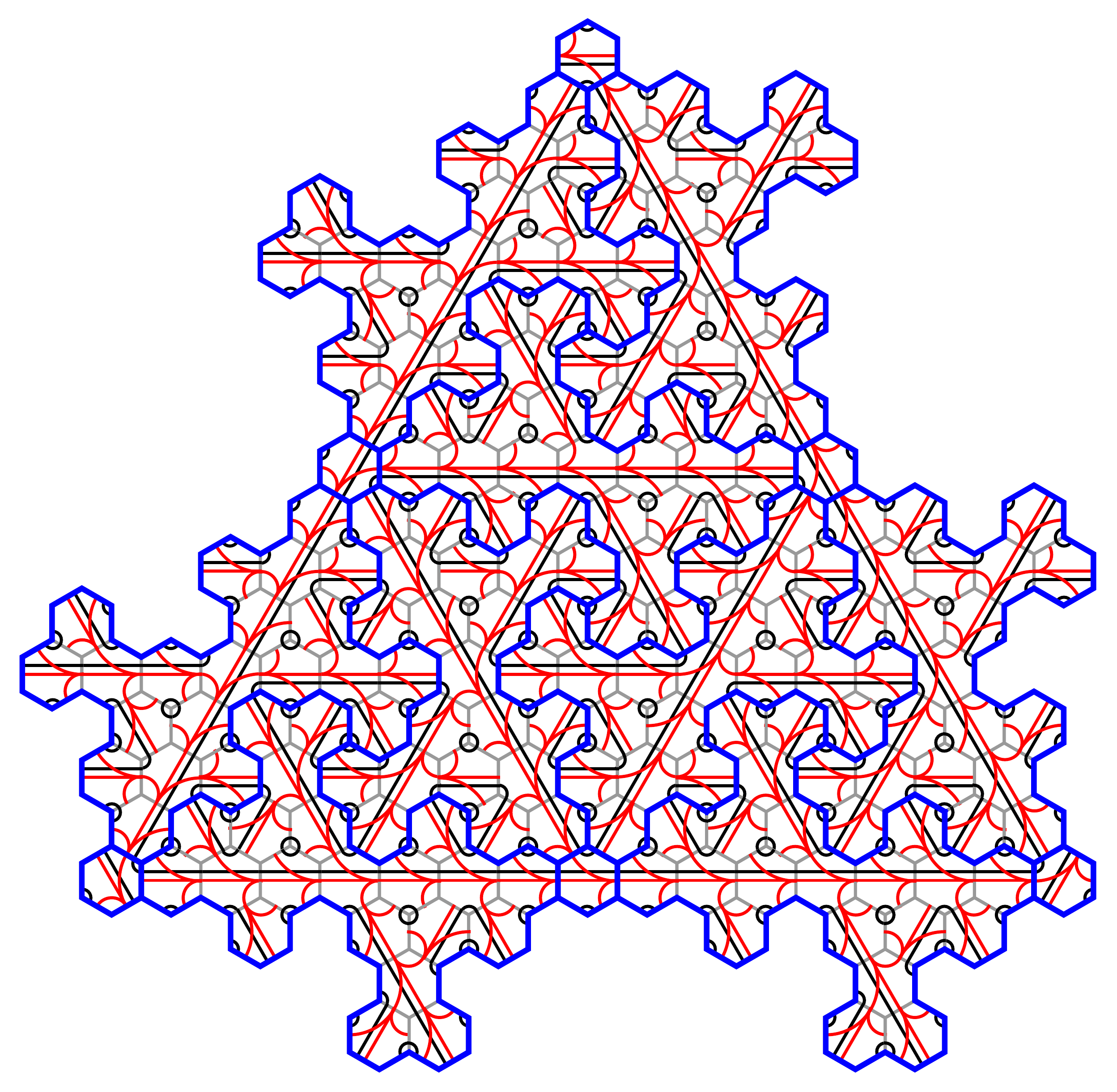}};
\draw[<-] (-2.1,0) -- node[above] {\scriptsize $n=2$} (-1.1,0);
\end{tikzpicture}
\end{center}
\caption{For $n=2$ and $n=3$, the patches in Figure \ref{dotted patch} fit together as above}
\label{dotted fit}
\end{figure}

In order to provide general arguments, we introduce pictorial notation. For $n \in \N$, a monotile with dashed lines represents a patch of tiles depicted in Figure \ref{dotted patch}, where the main diameter of R2-trees has length $2^n-1$. We note that these are the patches $P_n$ that appeared when we constructed a tiling in Proposition \ref{existence}. Moreover, notice that these patches fit together in the manner depicted in Figure \ref{dotted fit}.

\begin{figure}[ht]
\begin{tikzpicture}
\node (tile7_1) at (0,0) [label=below:{an R2-cycle of length $2^n-1$}] {
\includegraphics[width=5cm]{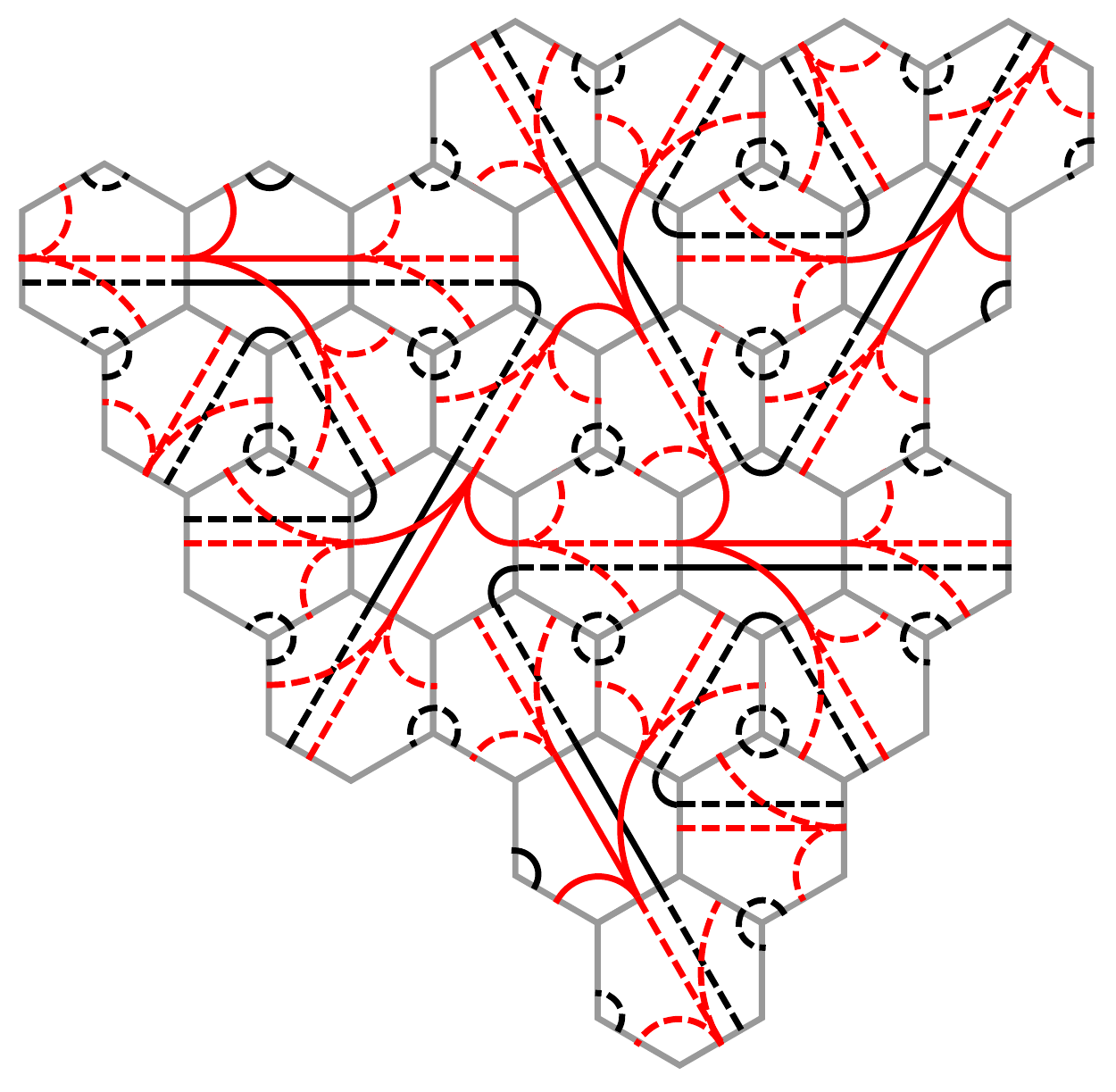}
};
\node (tile7_2) at (8,0) [label=below:{an R2-anticycle of length $2^n-1$}] {
	\includegraphics[width=5cm]{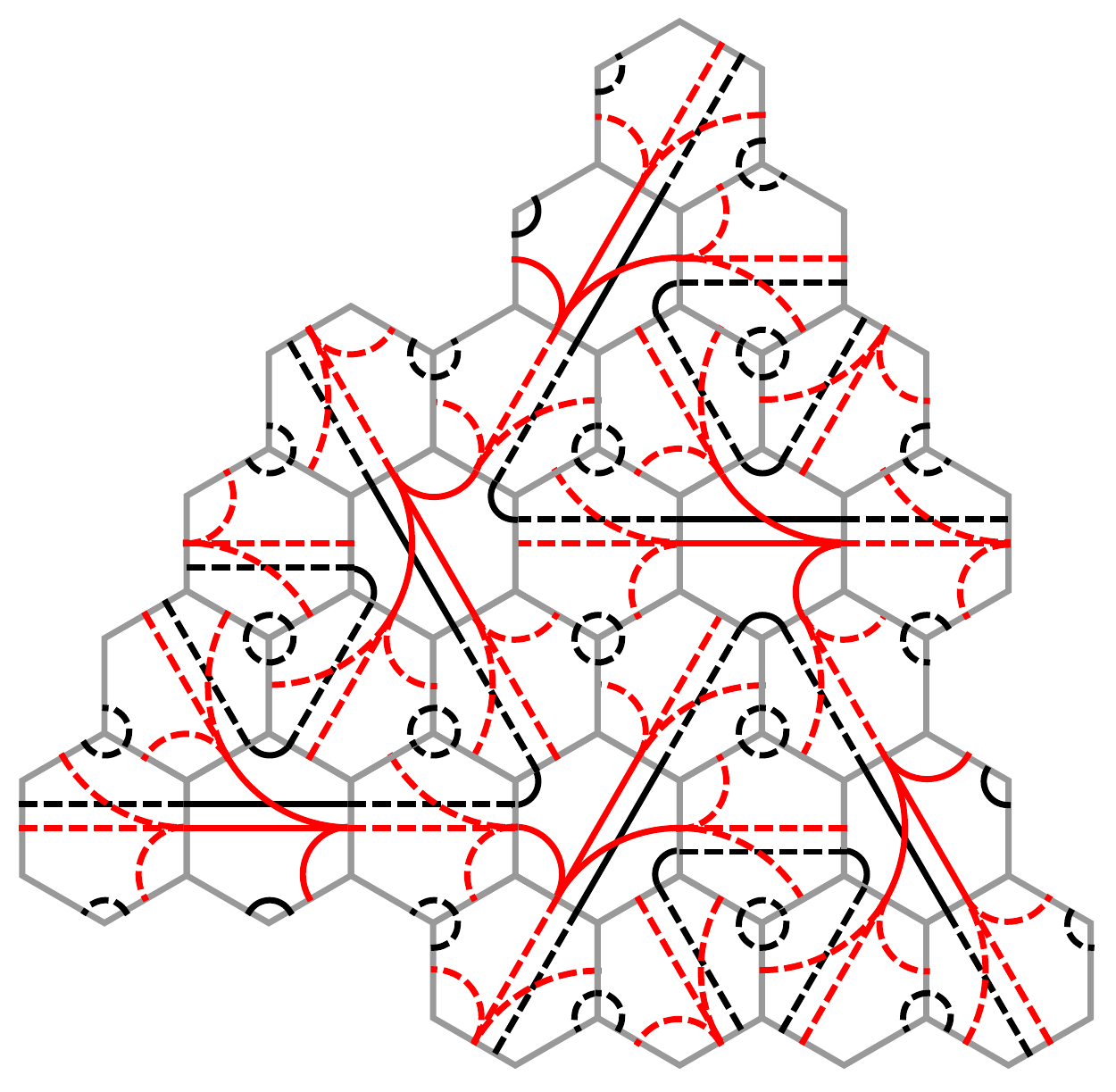}
};
\end{tikzpicture}
\caption{The illegal patches appearing in Lemmas \ref{cyclic connected tree} and \ref{cyclic disconnected tree}}
\label{illegal patches}
\end{figure}

The fundamental tool of this section is to prove that \textbf{R2} rules out three R1-triangles meeting corners to sides in the cyclic fashion appearing in Figure \ref{illegal patches}, where the solid lines have length one, and the dotted lines have length $2^n -1$ for $n\in\{0,1,2,\ldots\}$. Due to the behaviour of the R2-trees in the centre of the cyclic R1-triangles, we will refer to these configurations as \emph{R2-cycles} and \emph{R2-anticycles}, respectively. We note that Lemma \ref{triangle sizes} implies that R2-cycles and R2-anticycles only occur with side length $2^n -1$.

\begin{figure}[ht]
	\centering
	\includegraphics[width=0.5\textwidth]{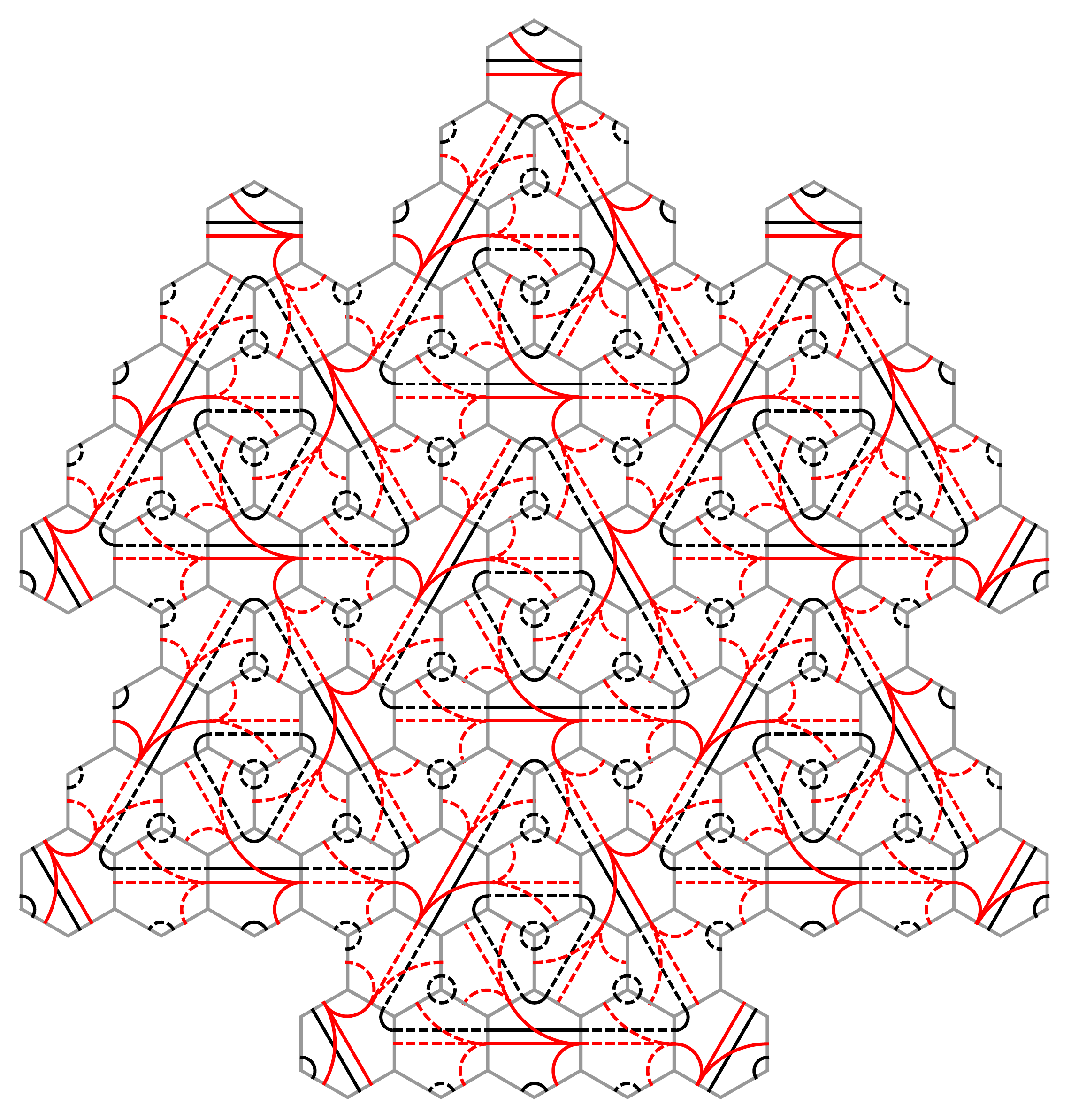}	
	\caption{A periodic lattice of R1-triangles}
	\label{periodic pattern}
\end{figure}

An immediate consequence of ruling out R2-cycles and R2-anticycles is that periodic lattices of R1-triangles, as depicted in Figure \ref{periodic pattern}, is no longer possible. Of course, it is clear that such lattices do not satisfy \textbf{R2}. However, a growth rule that disallows these periodic lattices was the key to the results of this paper.

\begin{figure}[ht]
\begin{tikzpicture}[>=stealth]
\node at (0,0) {\includegraphics[width=0.65\textwidth]{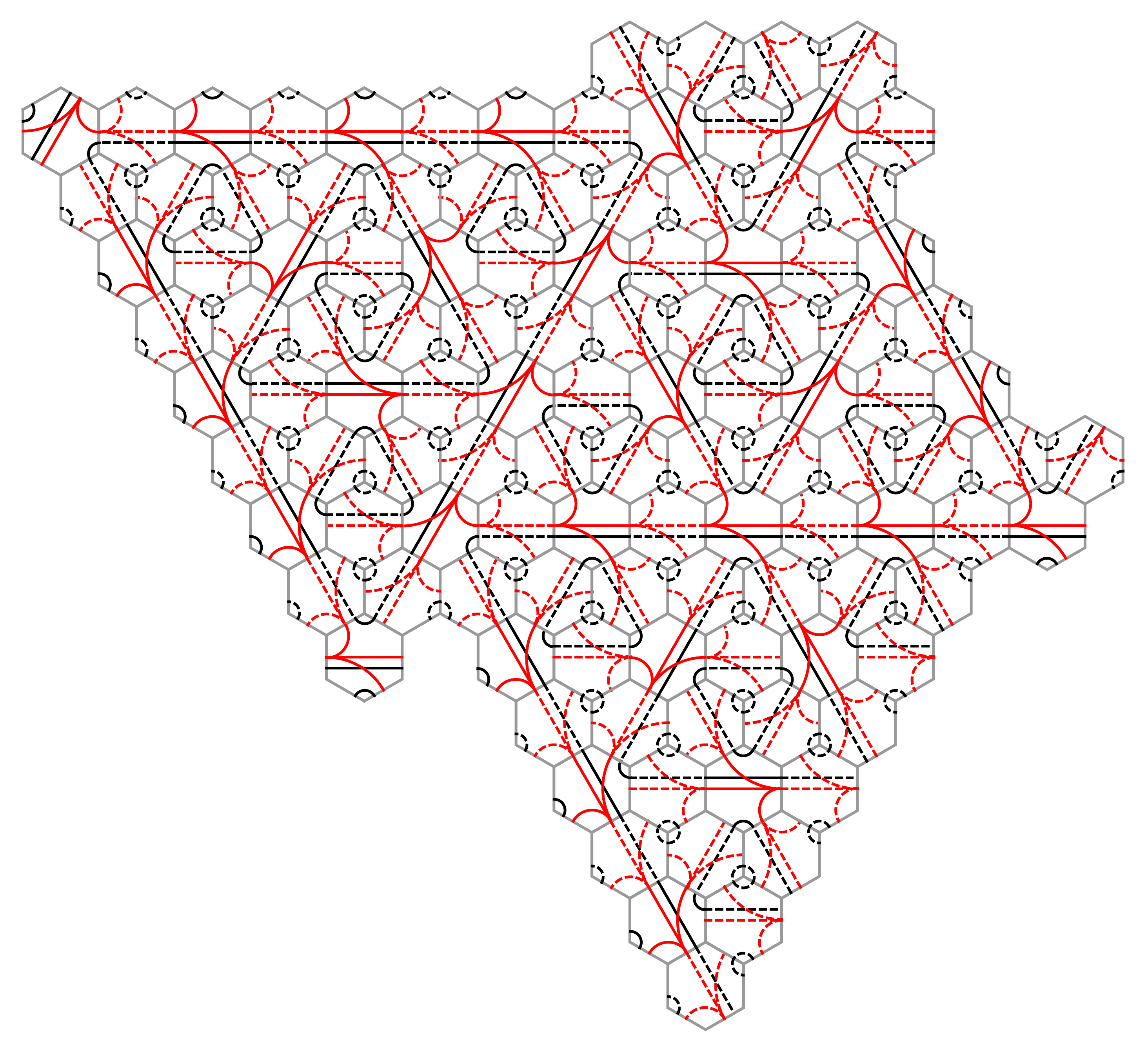}};
\node at (-0.0,3.1) {\textbf{A}};
\node at (1.6,3.6) {\textbf{B}};
\node at (-3.3,-1.5) {$ m=1 $};
\node[align=center] (a) at (7.5,2.5) {
	An R1-triangle \\
	with length \\
	$ 2^n (2^2 - (2^1-1)) -1$ \\
	(i.e. $m=k=1$),\\
	which is impossible\\
	by Lemma \ref{triangle sizes}};
\node at (4,2.5) {} edge[<-] (a);
\end{tikzpicture}
\caption{The case for $n \in \N$ and $m=1$ in the induction from Lemma \ref{cyclic connected tree}}
\label{cyclic tree fig}
\end{figure}

\begin{figure}[ht]
\begin{tikzpicture}
\node at (0,0) {\includegraphics[scale=0.3]{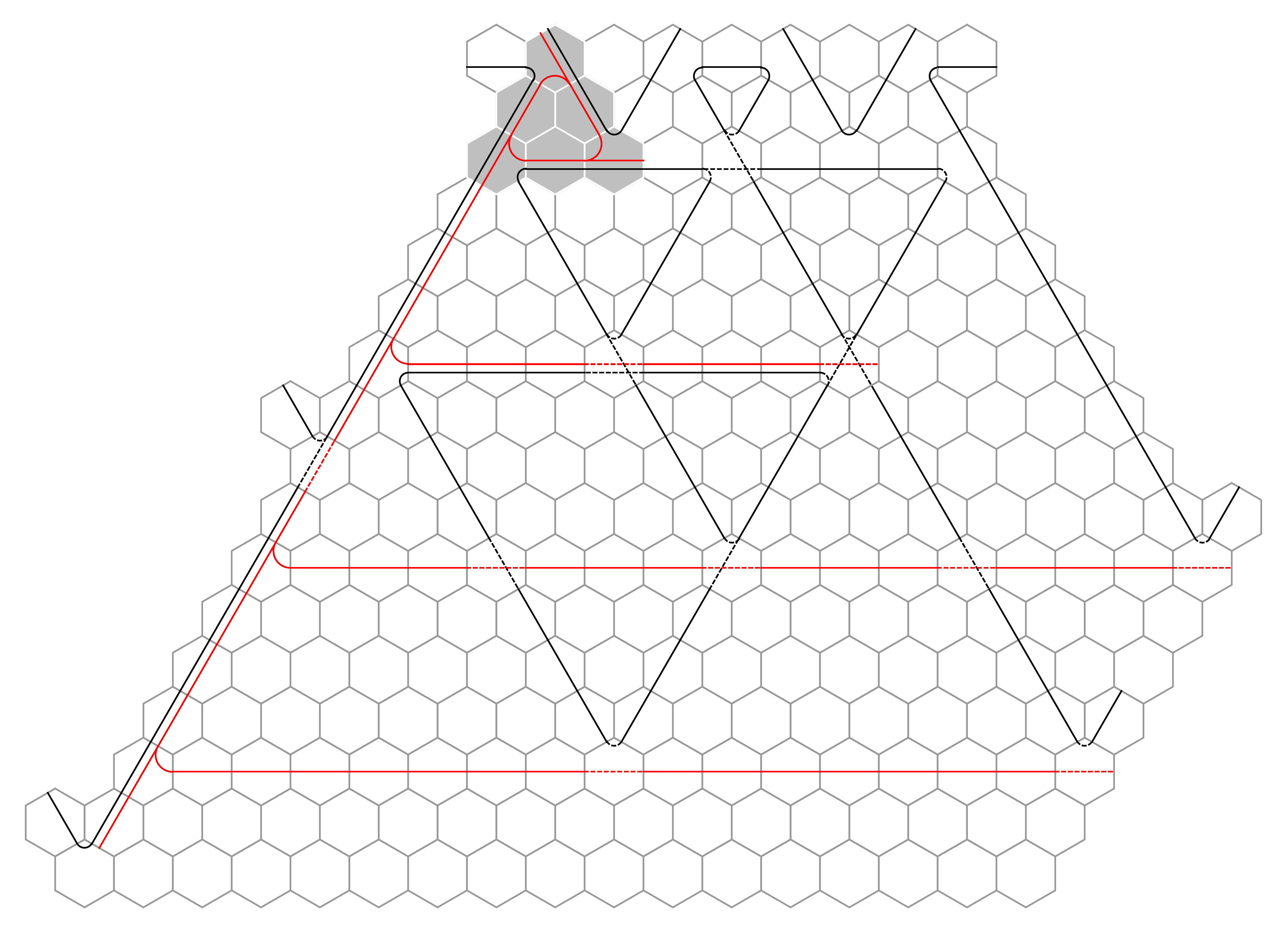}};
\node at (-5.4,0.7) {$ m=1 $};
\node at (-8,-3.9) {$ m=2 $};
\node at (-2.0,4.1) {\textbf{A}};
\node at (-0.35,4.65) {\textbf{B}};
\node at (2,2.8) {$ \mathbf{5} $};
\node at (4.7,-1.75) {$ \mathbf{13} $};
\node at (6,0.6) {$ \mathbf{9} $};
\end{tikzpicture}
\caption{The case for $n=1$ and $m=1,2$ in the induction from Lemma \ref{cyclic connected tree}}
\label{cyclic tree n2}
\end{figure}

\begin{lemma}\label{cyclic connected tree} 
Suppose $T$ is a tiling in $\CC$, then $T$ does not contain an R2-cycle.
\end{lemma}

\begin{proof}
We begin with a patch containing a R2-cycle of length $2^n -1$, and show that any tiling that extends the patch fails to satisfy \textbf{R2}. Fix $n \in \{0,1,2,\ldots\}$, and suppose we start with the patch on the left-hand side of Figure \ref{illegal patches}.

Since \textbf{R2} implies that the R2-tree must be infinite, at least one branch of the R2-tree leaving the central R2-cycle must be infinite. We will refer to the R1-triangle associated with the infinite tree as \emph{triangle} $\mathbf{A}$. Lemma \ref{no infinite triangles} implies that triangle $\mathbf{A}$ cannot be infinite. We will show that triangle $\mathbf{A}$ cannot be finite either. Lemma \ref{triangle sizes} implies that triangle $\mathbf{A}$ must have length $2^{n+m+1}-1$ for some $m \in \{1,2,\ldots\}$. Since the union of R2-trees terminates at the R1-corner of $\mathbf{A}$, there must be an infinite branch leaving the main tree. All R2-branches towards the interior of $\mathbf{A}$ are finite, so any infinite branch must turn away from triangle $\mathbf{A}$. A straightforward, but geometrically technical, induction proves that if triangle $\mathbf{A}$ has length $2^{m+n+1}-1$, then the next R1-triangle clockwise in the R2-cycle of length $(2^n -1)$ (labelled $\mathbf{B}$ in Figures \ref{cyclic tree fig} and \ref{cyclic tree n2}) forces an R1-triangle of length $2^n(2^{m+1}-(2^k-1))-1$ for some $k \in \{1,\dots,m\}$. Lemma \ref{triangle sizes} implies that such triangles cannot exist in a tiling, giving us the desired contradiction. So $T$ cannot belong to $\CC$.
\end{proof}

\begin{remark}
A comment on the omitted geometric induction argument from the proof of Lemma \ref{cyclic connected tree} is in order. Figure \ref{cyclic tree fig} shows the argument for arbitrary $n \in \N$ and $m=1$. Figure \ref{cyclic tree n2} shows the geometric argument for $n=1$ and $m=1,2$. Using Figure \ref{cyclic tree n2}, the general argument for $m=1,2$ follows by using dashed tiles of length $2^n-1$, as depicted in Figures \ref{dotted patch} and \ref{dotted fit}.
\end{remark}

\begin{lemma}\label{cyclic disconnected tree} 
Suppose $T$ is a tiling in $\CC$, then $T$ does not contain an R2-anticycle.
\end{lemma}

\begin{proof}
We begin with a patch containing a R2-anticycle of length $2^n -1$, and show that any tiling that extends the patch does not belong to $\CC$. Fix $n \in \{0,1,2,\ldots\}$, and suppose we start with the patch on the right-hand side of Figure \ref{illegal patches}.

\begin{figure}[ht]
\centering
\begin{tikzpicture}
\node at (0,0) {\includegraphics[width=0.6\textwidth]{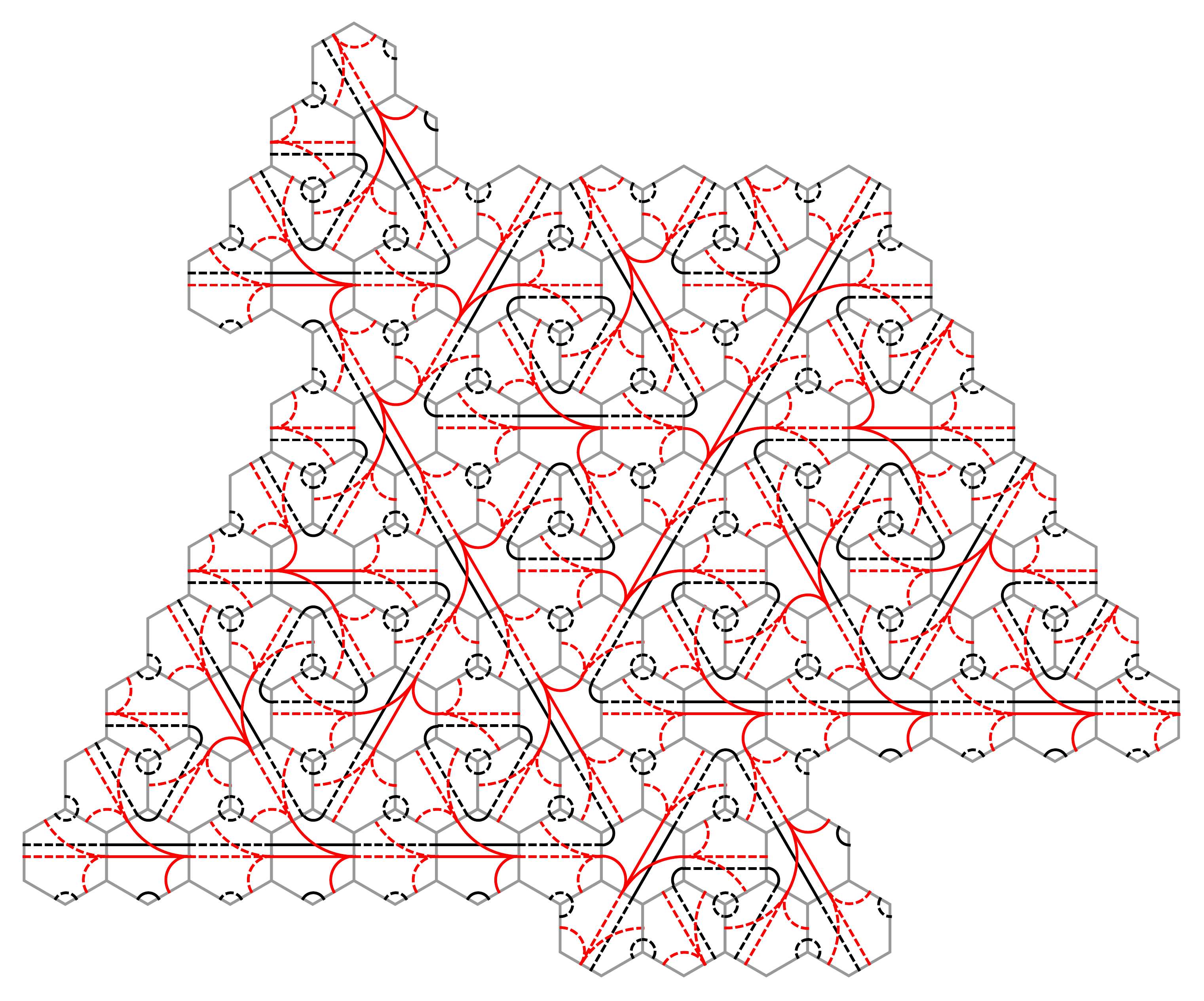}};
\node (a) at (-4.5,1.2) {$ x $};
\node at (-2.5,1.2) {} edge[<-] (a);
\node (b) at (3.5,-2.5) {R2-anticycle};
\node at (0.4,-2.2) {} edge[<-] (b);
\end{tikzpicture}
\caption{The pictorial idea of the proof of Lemma \ref{cyclic disconnected tree}, for $n \in \N$ and $m=2$}
\label{cyclic disconnected tree fig}
\end{figure}
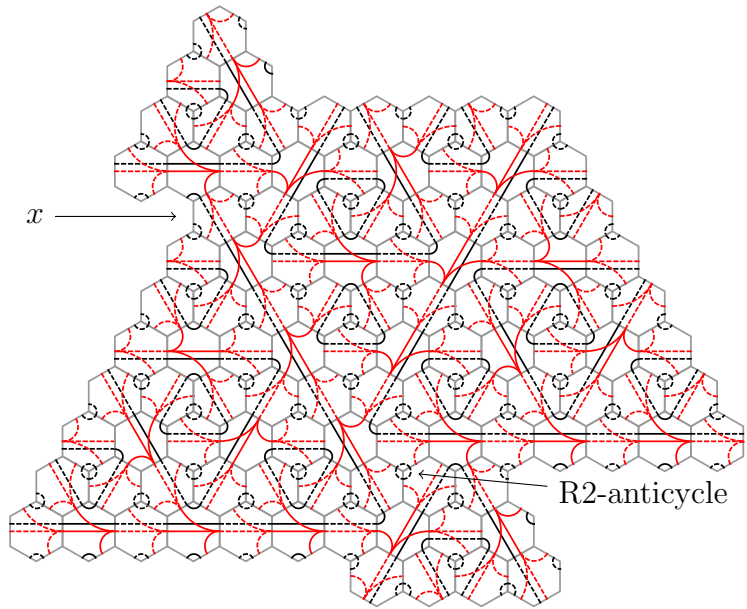

Rule \textbf{R2} implies that all three branches of the R2-anticycle must be infinite and must all be connected. Let us concentrate on just one of these branches. Lemma \ref{no infinite triangles} implies that the R1-triangle associated with this branch cannot be infinite, and then Lemma \ref{triangle sizes} implies it must have length $2^{n+m}-1$ for some $m \in \{1,2,\ldots\}$. However, if this R1-triangle has length $2^{n+m}-1$, then an R2-cycle of length $2^n -1$ is forced where the R1-triangle associated with this branch has an R1-corner. Lemma \ref{cyclic connected tree} implies that $T$ is not in $\CC$. See Figure \ref{cyclic disconnected tree fig} for a pictorial representation, where $x$ is the location of the R2-cycle in the case $m=2$.
\end{proof}

\begin{figure}[ht]
\[
\begin{tikzpicture}
\begin{scope}
\node (tile3_1) at (0,0) [label=below:{Centre}] {\includegraphics[scale=0.135]{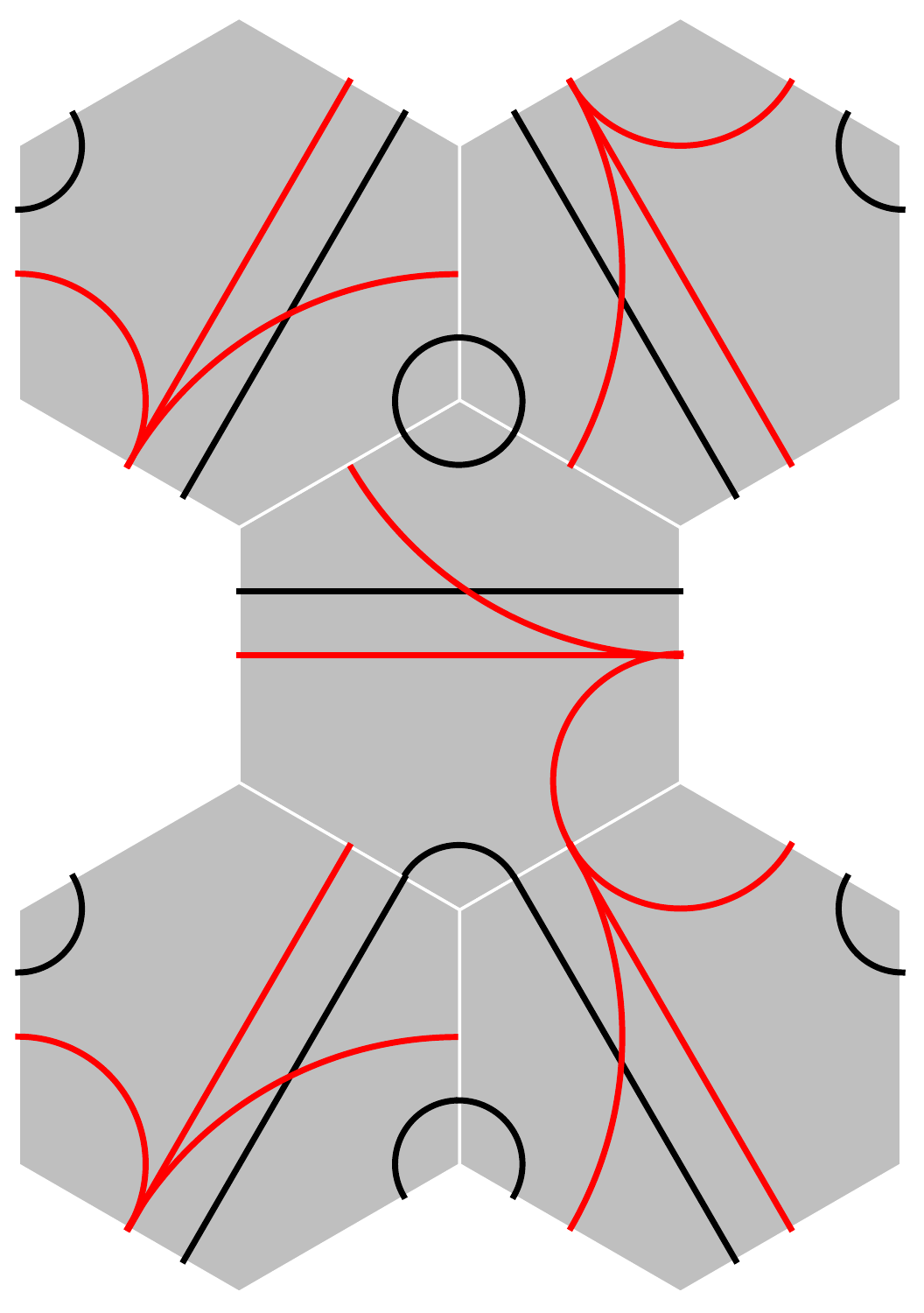}};
\node (tile3_2) at (-3,0) [label=below:{Left R1-anticycle}] {\includegraphics[scale=0.135]{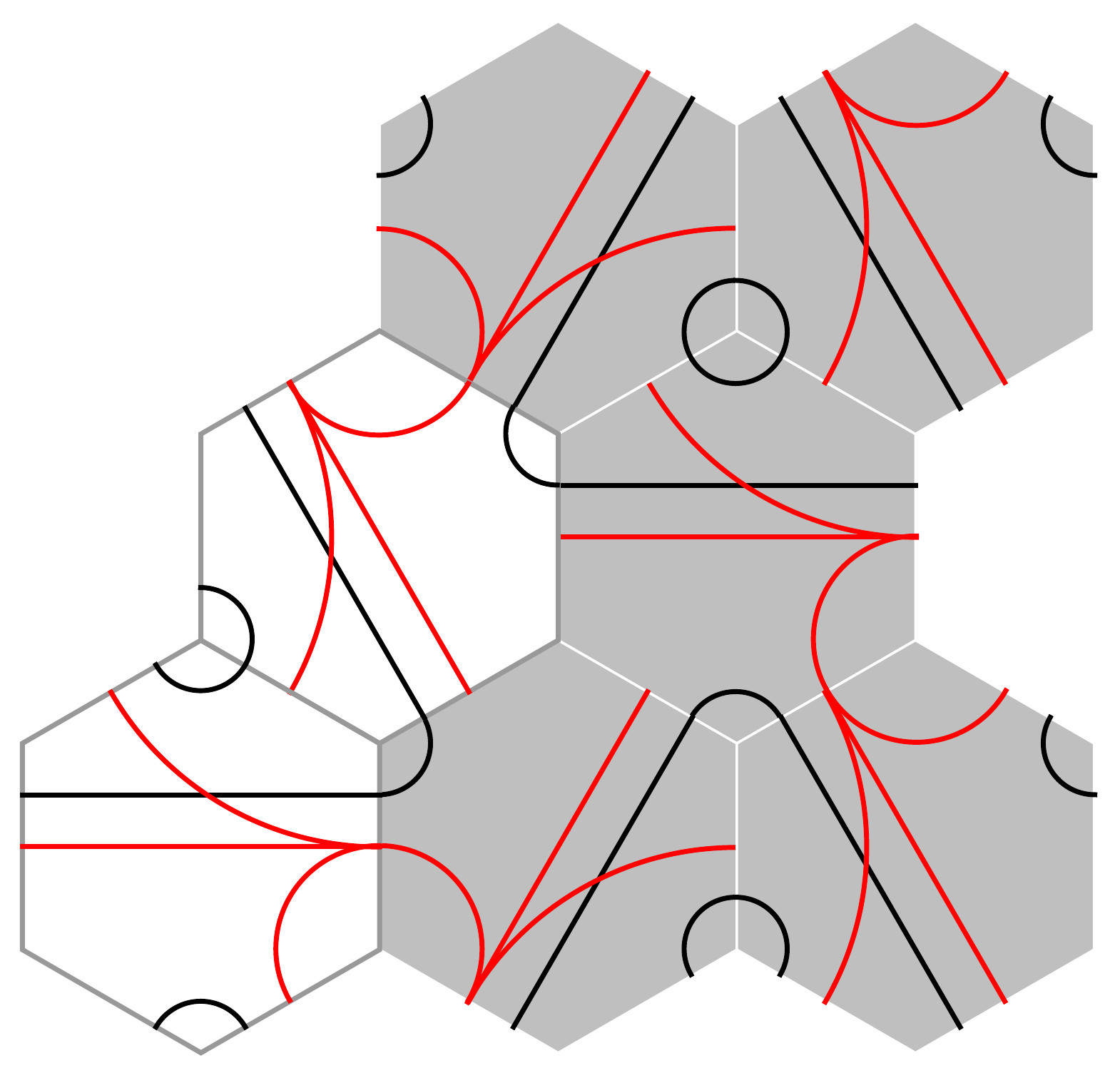}}
	edge[<-] (tile3_1);	
\node (tile3_3) at (3,0) [label=below:{Right R1-cycle}] {\includegraphics[scale=0.135]{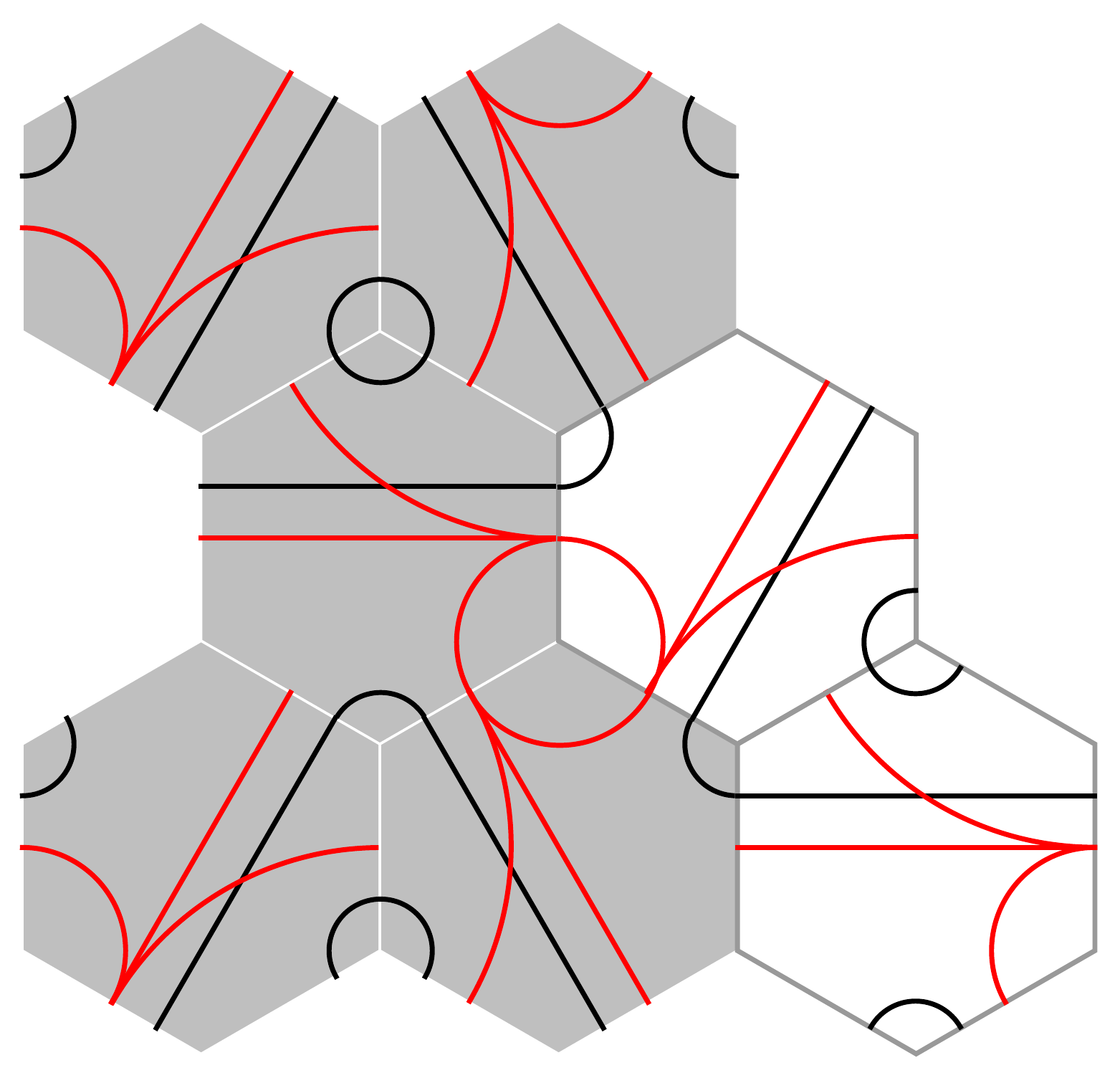}}
	edge[<-] (tile3_1);
\end{scope}
\begin{scope}[yshift=-3.5cm]
\node (tile4_1) at (0,0) [label=below:{Centre}] {\includegraphics[scale=0.135]{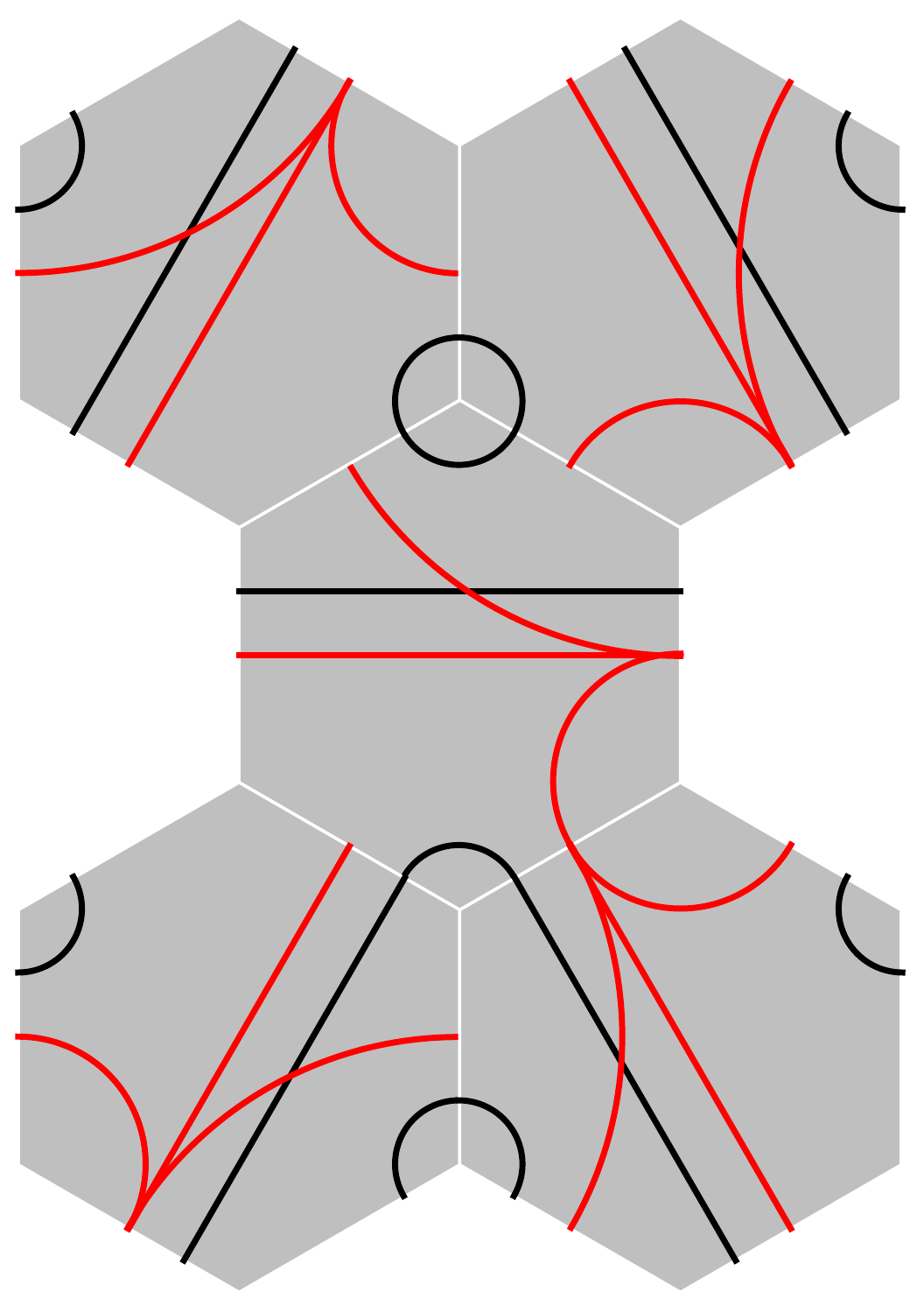}};
\node (tile4_2) at (-3.8,0) [label=below:{Left R1-anticycle}] {\includegraphics[scale=0.135]{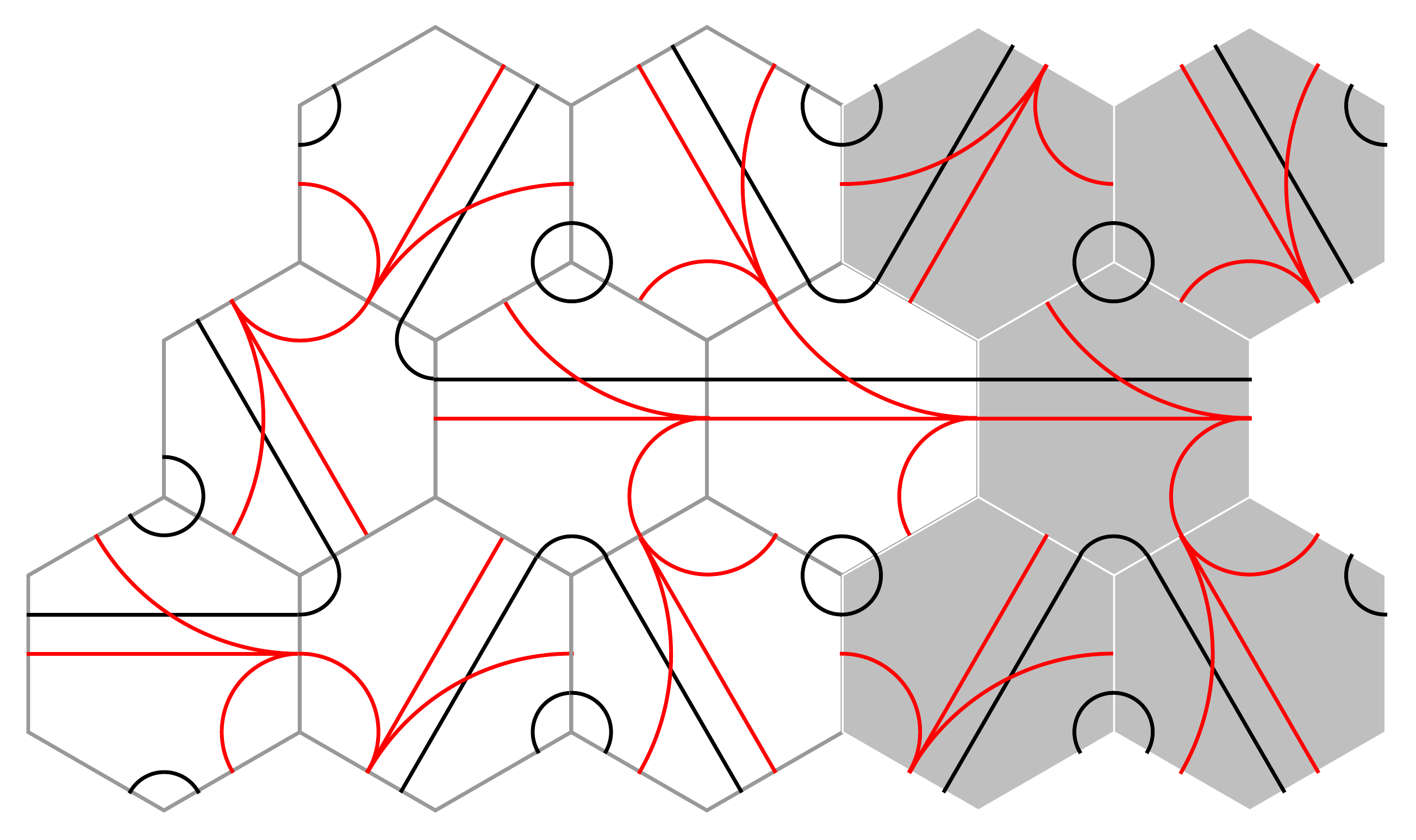}}
	edge[<-] (tile4_1);	
\node (tile4_3) at (3.8,0) [label=below:{Right R1-cycle}] {\includegraphics[scale=0.135]{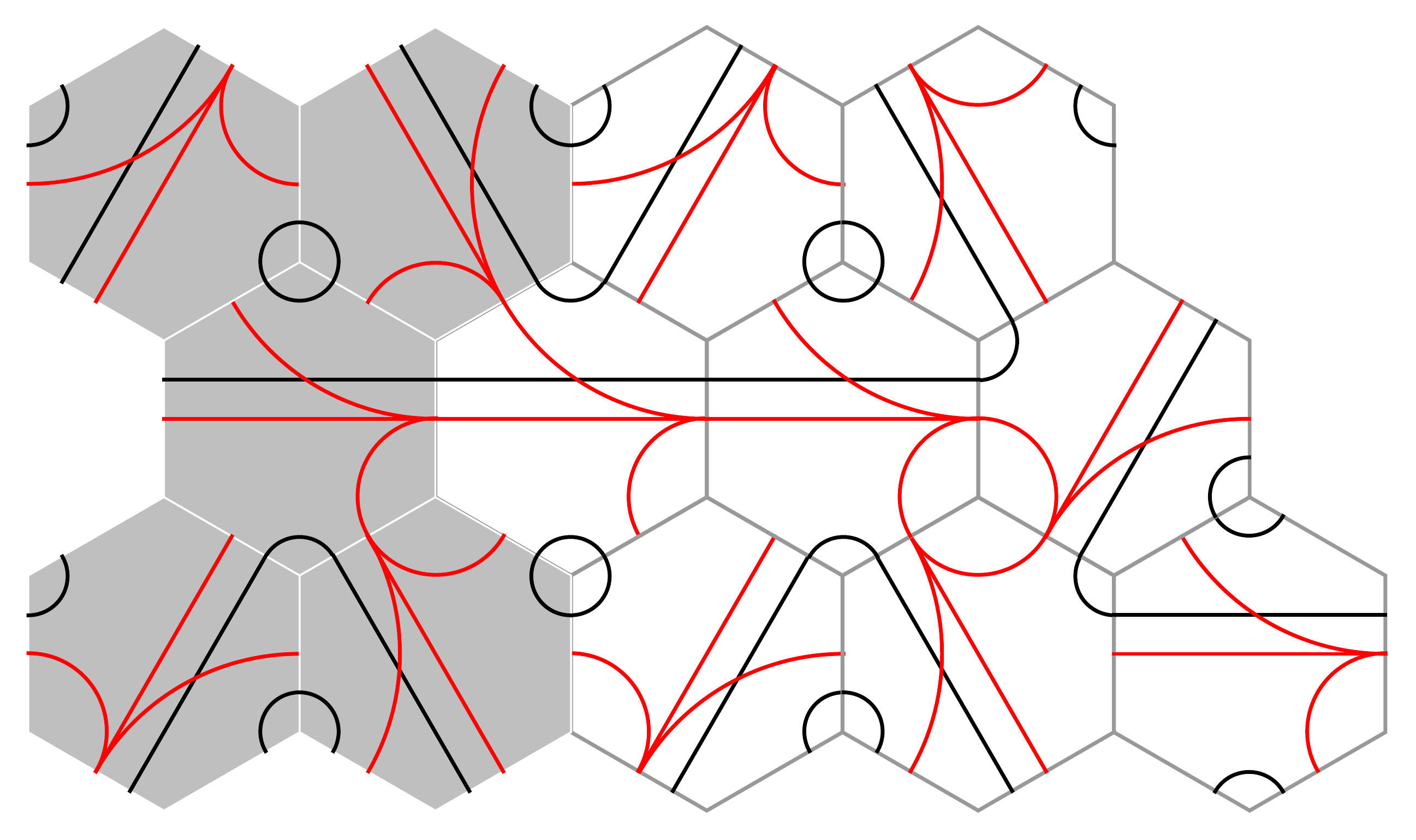}}
	edge[<-] (tile4_1);
\end{scope}
\begin{scope}[yshift=-7cm]
\node (tile4_1) at (0,0) [label=below:{Centre}] {\includegraphics[scale=0.135]{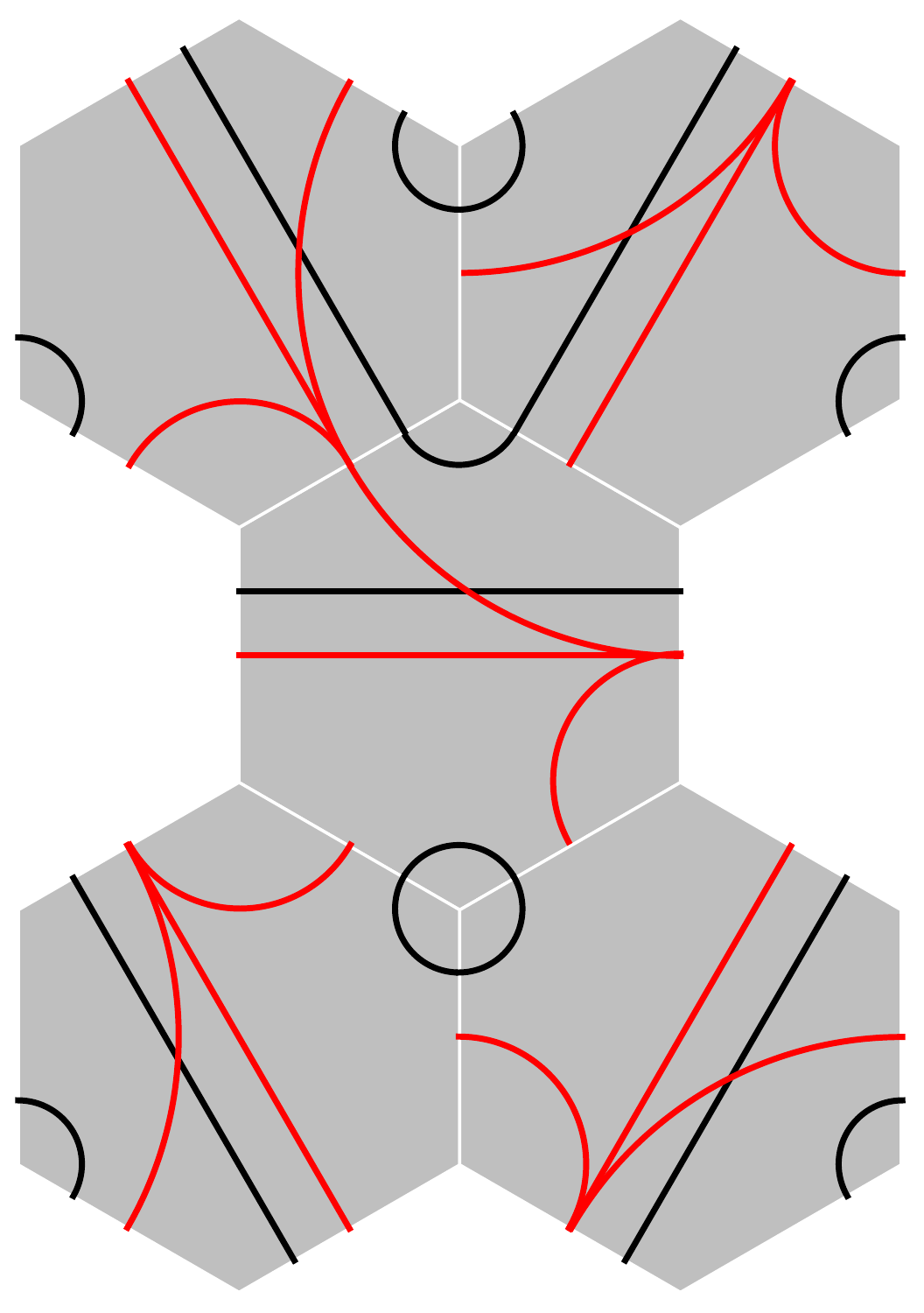}};
\node (tile4_2) at (-3.8,0) [label=below:{Left R1-anticycle}] {\includegraphics[scale=0.135]{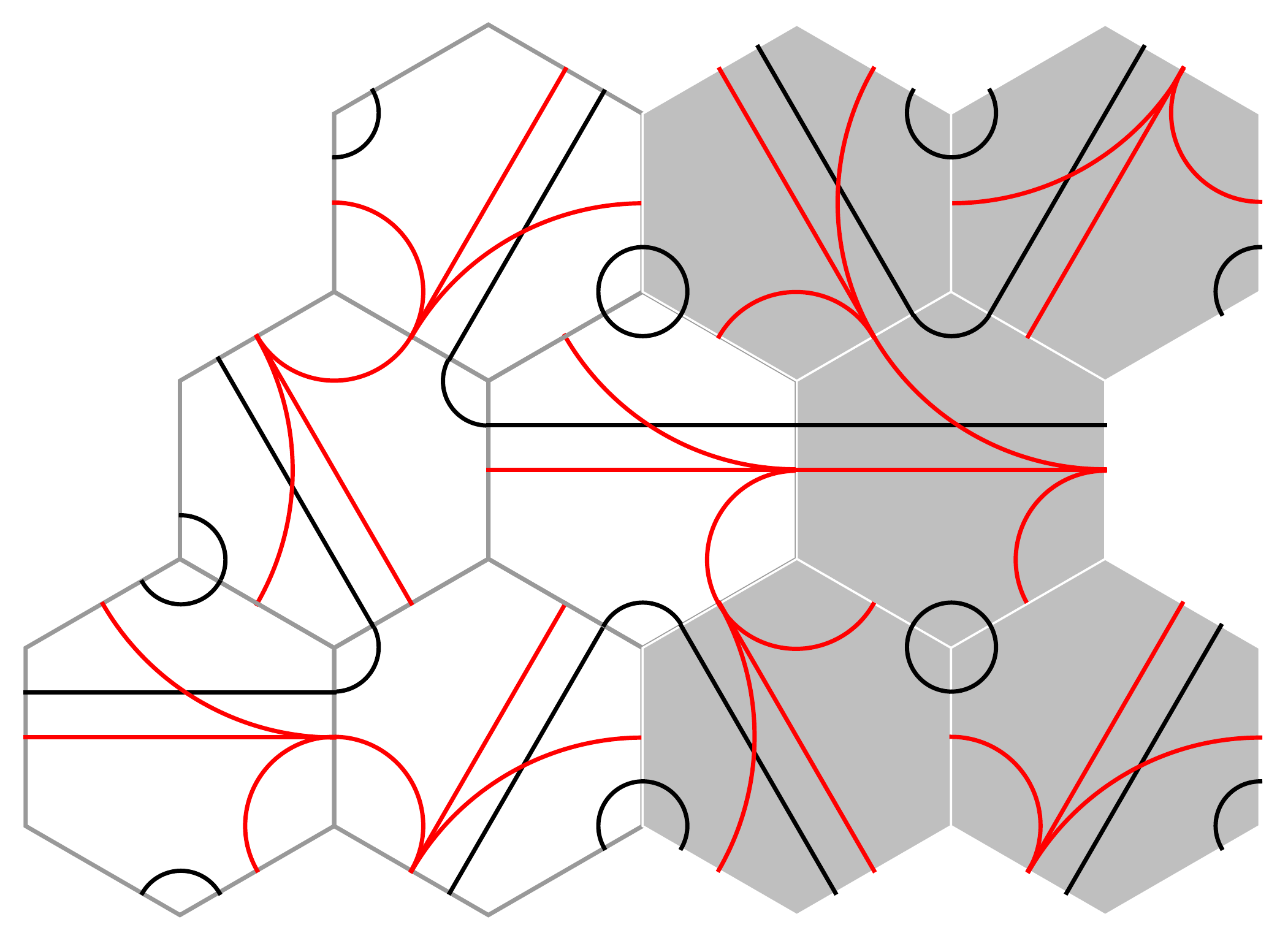}}
	edge[<-] (tile4_1);	
\node (tile4_3) at (3.8,0) [label=below:{Right R1-cycle}] {\includegraphics[scale=0.135]{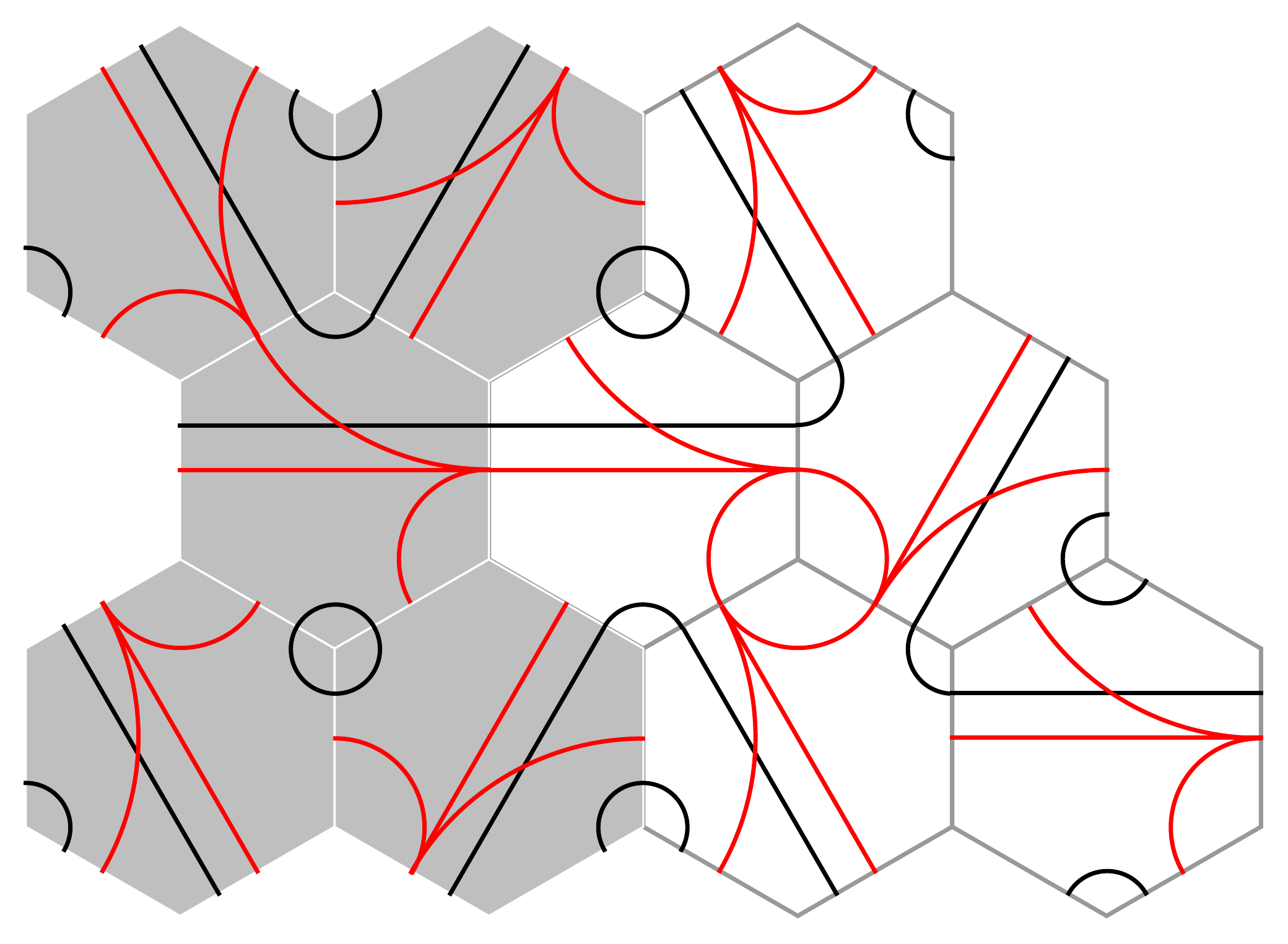}}
	edge[<-] (tile4_1);
\end{scope}
\end{tikzpicture}
\]
\caption{Three possibilities for $m=0$ and $n>0$ in the proof of Theorem \ref{main result monotile}}
\label{zero case}
\end{figure}

\begin{figure}[ht]
\[
\begin{tikzpicture}
\begin{scope}
\node (tile3_1) at (0,0) [label=below:{Centre}] {\includegraphics[scale=0.135]{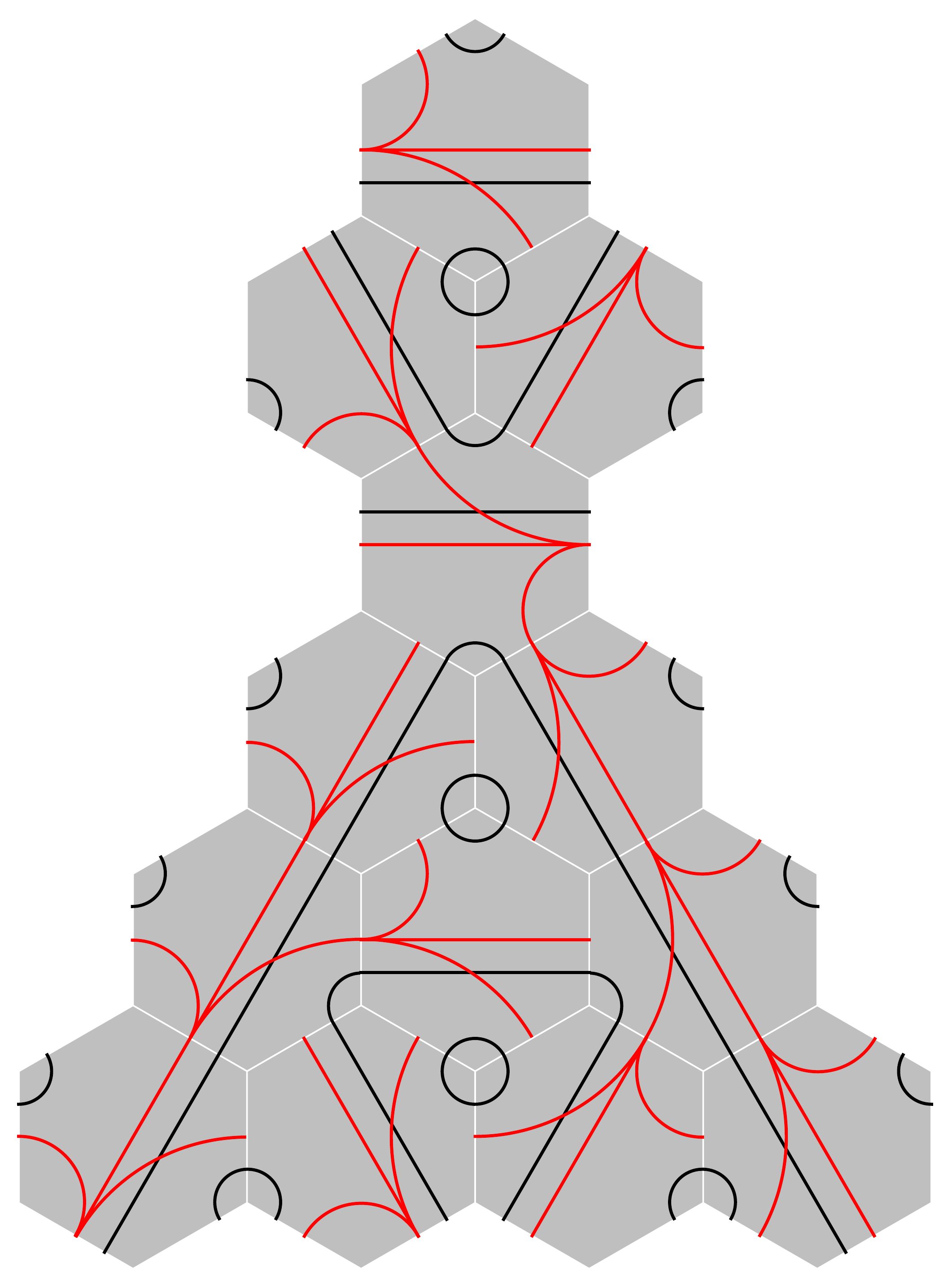}};
\node (tile3_2) at (-5,0) [label=below:{Left R1-anticycle}] {\includegraphics[scale=0.135]{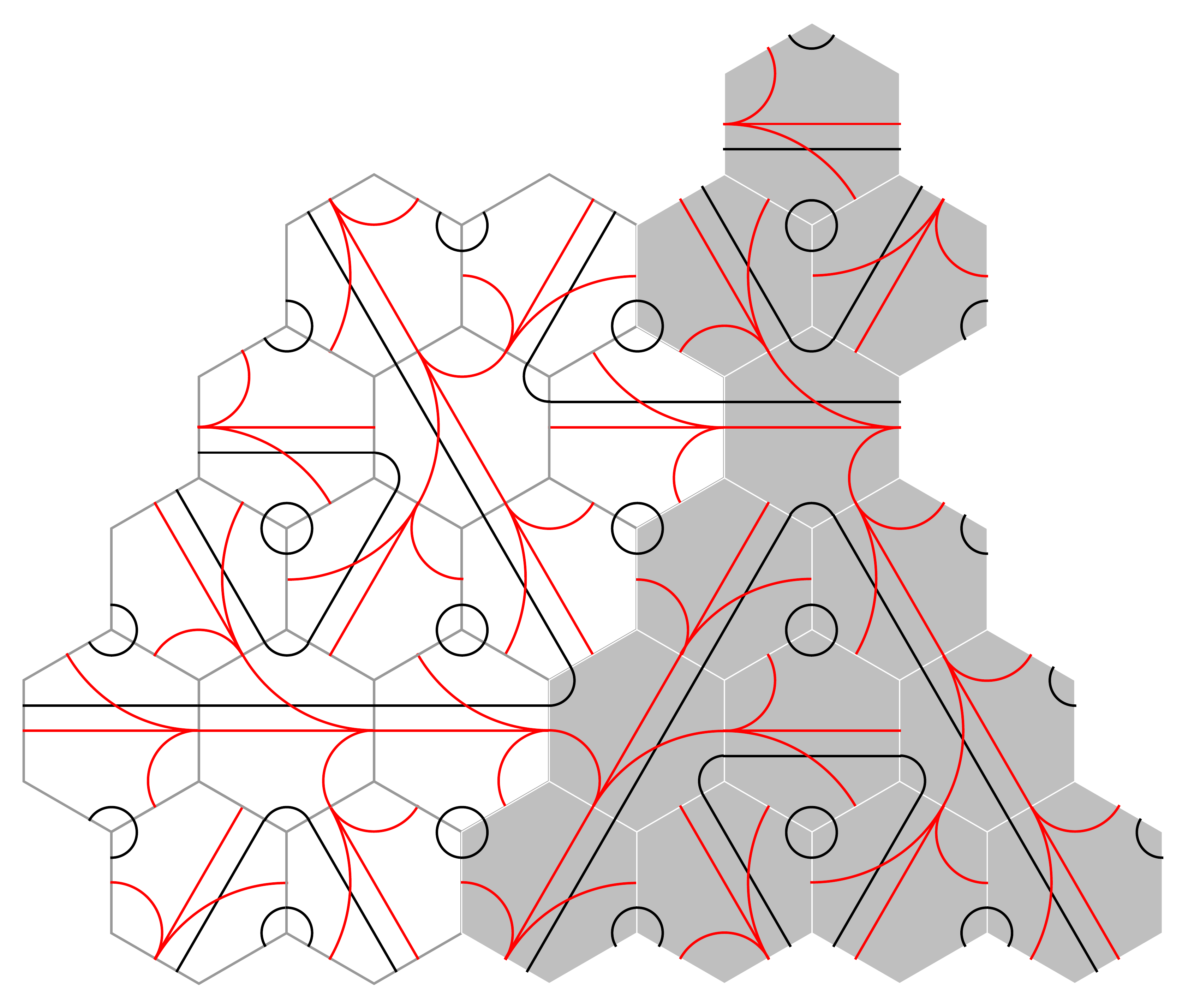}}
	edge[<-] (tile3_1);	
\node (tile3_3) at (5,0) [label=below:{Right R1-cycle}] {\includegraphics[scale=0.135]{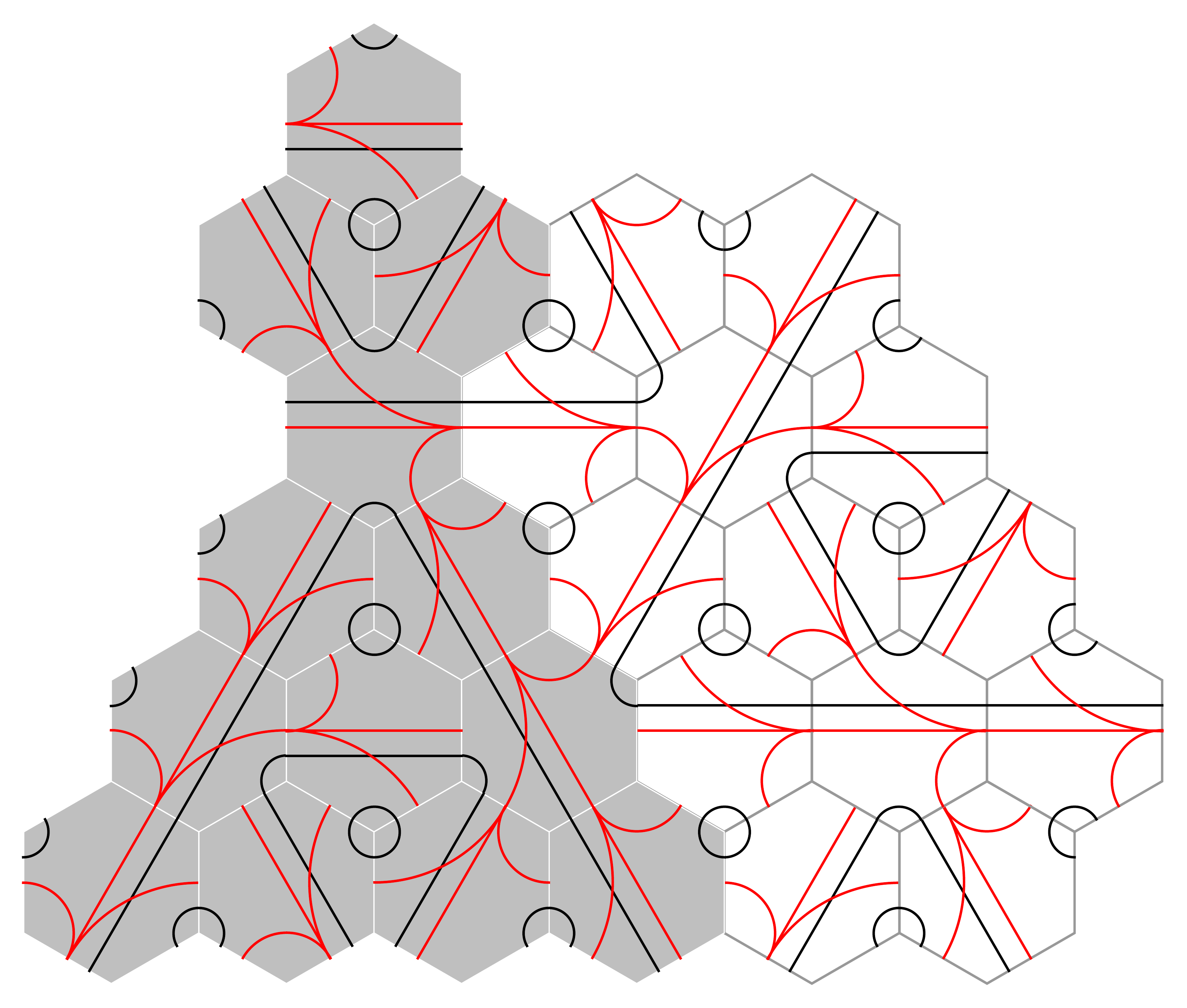}}
	edge[<-] (tile3_1);
\end{scope}	
\end{tikzpicture}
\]
\caption{One possibility for $m=1$ and $n>1$ in the proof of Theorem \ref{main result monotile}}
\label{one case}
\end{figure}

We are now able to tackle the proof of Theorem \ref{main result monotile}. The reader is encouraged to consider Figures \ref{zero case} and \ref{one case} while reading through the proof.

\begin{proof}[Proof of Theorem \ref{main result monotile}]
Propositions \ref{existence} and \ref{fault line} prove that $\CC_0$ and $\CC_1$ are both non-empty and only contain nonperiodic tilings. We will prove that $\CC=\CC_0 \cup \CC_1$, thereby proving the result.

By definition, $\CC_0$ and $\CC_1$ satisfy \textbf{R1} and \textbf{R2} so that $\CC_0 \cup \CC_1 \subseteq \CC$. We are left to prove the reverse inclusion. Suppose $T$ is in $\CC$. If all pairs of R1-triangles in $ T $ that meet at a common tile have the same length, then $T$ is in $\CC_0$.
If there exists a pair of R1-triangles which meet at a common tile, but do not have the same length, we claim that $T$ is in $\CC_1$, which would imply that $\CC \subseteq \CC_0 \cup \CC_1$.
To do this, we must show that such a tiling $T$ has a bi-infinite R1-line through it.

Suppose that $T$ contains a tile where a pair of R1-triangles meet that have lengths $2^m-1$ and $2^n-1$ for $m \neq n$.
Since these two R1-triangles meet at their respective R1-corners on a common tile, they are separated by the R1-line segment in this tile. We will argue that this line segment must extend indefinitely in both directions. For the sake of contradiction, suppose the extended R1-line segment has an R1-corner, which must occur at length $2^{m+j}$ along the straight R1-line segment from the corner of the $2^m-1$ R1-triangle, for $j \in \{1,2,\ldots\}$. At any such R1-corner, there is either an R2-cycle or an R2-anticycle of length $(2^k -1)$ for $k \in \{0,\ldots,\min\{m,n\}\}$. Since Lemma \ref{cyclic connected tree} and Lemma \ref{cyclic disconnected tree} imply that R2-cycles and R2-anticycles cannot exist in $ T $, the R1-line segment must extend infinitely in both directions.
Thus, $ T \in \CC_1 $ so that $ \CC \subseteq \CC_0 \cup \CC_1$, as required.
\end{proof}

\section{The continuous hull of our aperiodic monotile}

We conclude the paper with a brief discussion of the tiling space, or continuous hull, of the tilings in the class $\CC$. Recall that the \emph{continuous hull} of a collection of tilings $\Lambda$ is the completion of $\{T+\mathbb{R}^d \mid T \in \Lambda\}$ in the tiling metric, typically denoted by $\Omega$. Under mild assumptions, the continuous hull is a compact topological space endowed with a continuous $\mathbb{R}^d$ action, making $(\Omega,\mathbb{R}^d)$ a dynamical system. For further details see \cite[Section 1.2]{Sad}. We are interested in the continuous hull of $\CC=\CC_0 \cup \CC_1$.

\begin{thm}
All tilings in the continuous hull of $\CC$ are nonperiodic and satisfy $\textbf{R1}$. Moreover, in any such tiling, there are at most three connected components of R2-trees, and each component crosses an infinite number of tiles.
\end{thm}

\begin{proof}
We begin by considering the class $\CC_1$. As described in Proposition \ref{fault line}, a tiling is in $\CC_1$ if it contains a bi-infinite R1-line with intertwined $2^n$-periodic patterns of R1-triangles of length $2^n-1$ meeting either side of the R1-line, which typically have no length relationship with their opposite across the bi-infinite R1-line. See Figure \ref{hull of C1} for a typical patch that extends to a tiling in $\CC_1$. Note that on each side of the bi-infinite R1-line, tilings in $\CC_1$ look locally like tilings in $\CC_0$, so the completion of $\CC_1$ contains tilings in $\CC_0$. However, we will handle the tilings in $\CC_0$ later, so we ignore these elements of the completion for now.

\begin{figure}[ht]
	\centering
	\includegraphics[width=0.8\textwidth]{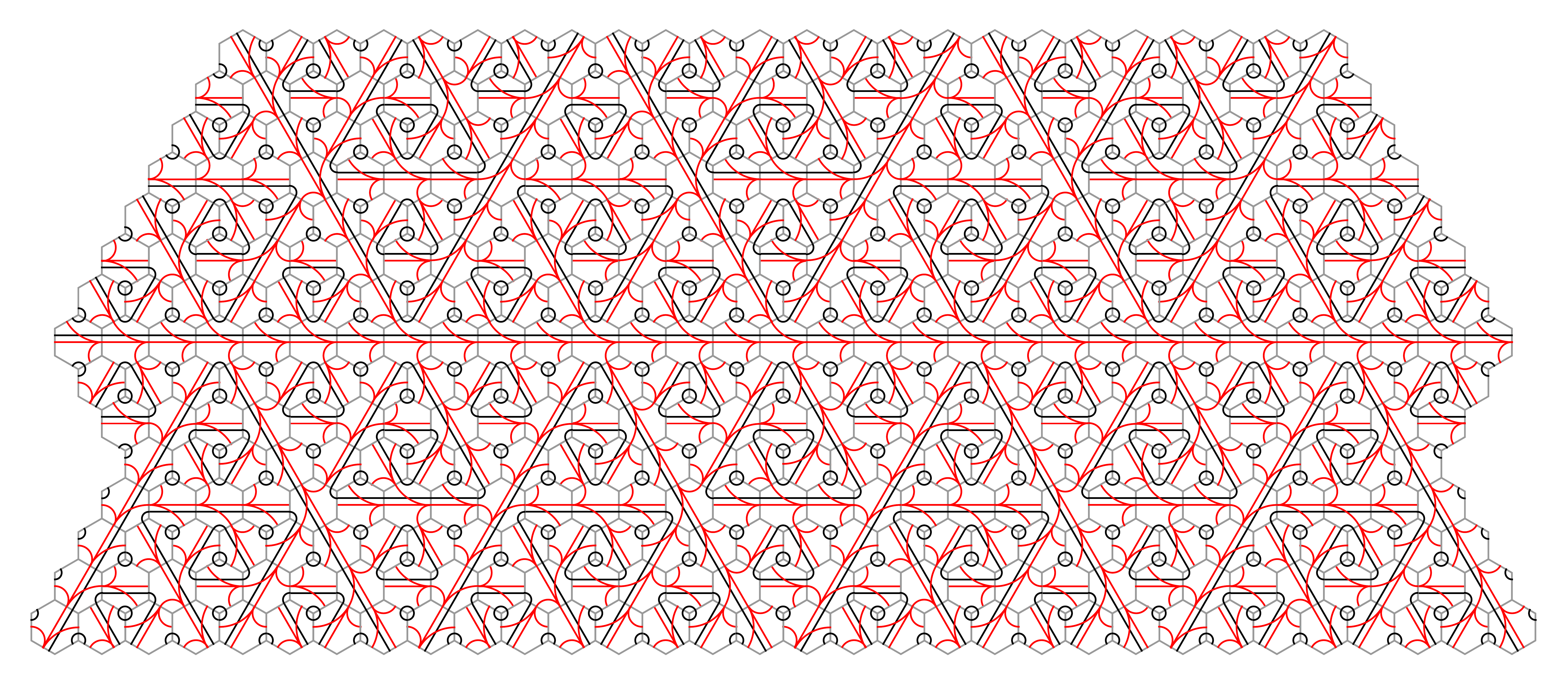}	
	\caption{A patch that extends to a typical tiling in $\CC_1$, note that R1-triangles that meet at a tile on the bi-infinite R1-line need not have the same length.}
	\label{hull of C1}
\end{figure}

We now consider tilings in the completion of $\CC_1$ that contain a bi-infinite R1-line. Since we can have arbitrarily large R1-triangles on either side of the bi-infinite R1-line, the completion of $\CC_1$ contains tilings with an infinite R1-triangle whose corner meets the R1-line. Such an infinite R1-triangle meeting a bi-infinite R1-line forces a half plane of R1-triangles with exactly two connected R2-components as shown in the proof of Lemma \ref{no infinite triangles}, and both of these components cross an infinite number of tiles. Since the behaviour of R1-triangles above and below the bi-infinite R1-line are independent, there are tilings with infinite R1-triangles on one or both sides of the bi-infinite R1-line in the completion. Thus, tilings in the completion of $\CC_1$ that contain a bi-infinite R1-line have at most three connected R2-components.

\begin{figure}[ht]
\[
\begin{tikzpicture}
\node (tile3_1) at (0,0) {\includegraphics[width=0.45\textwidth]{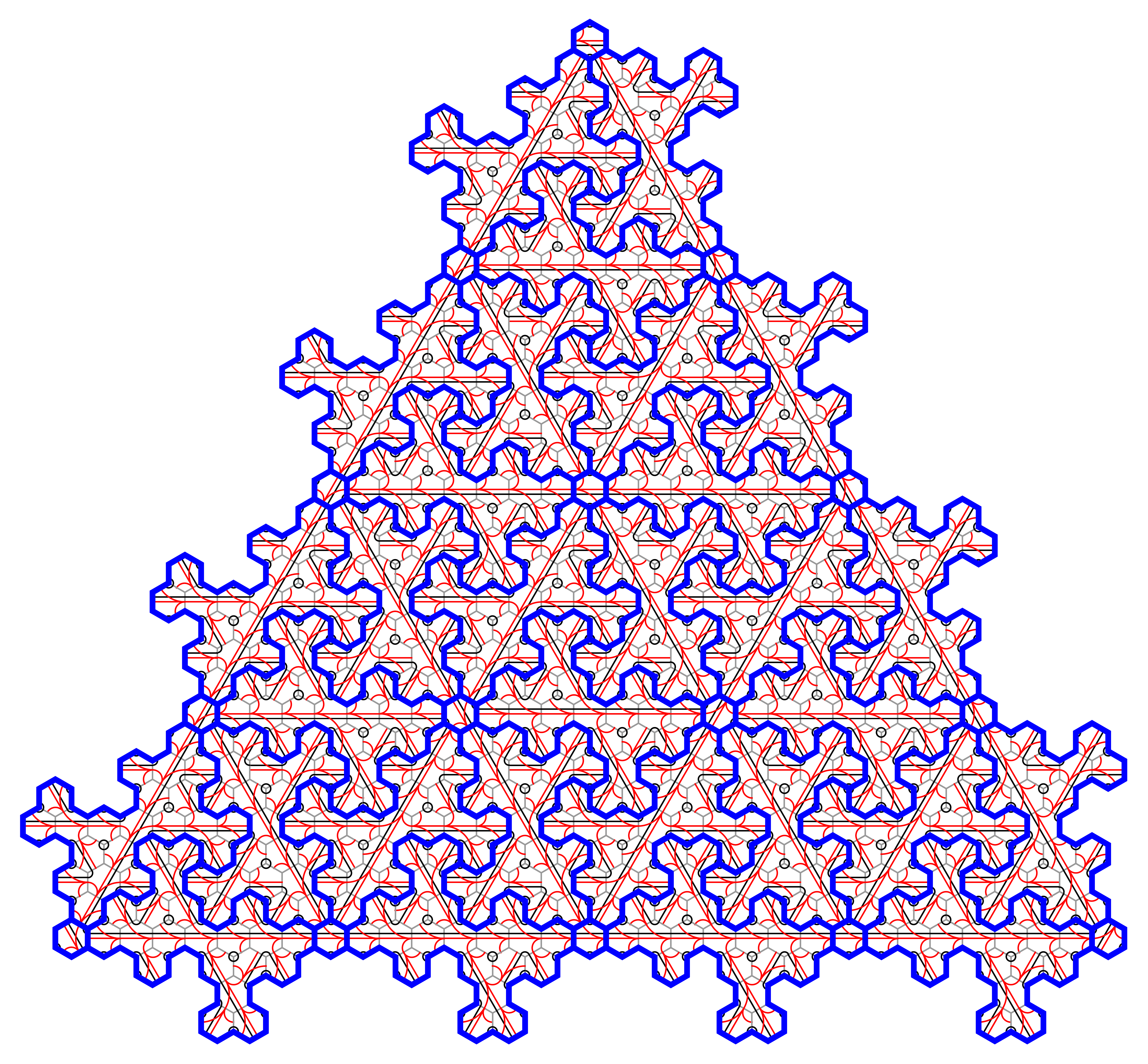}};
\node (tile3_2) at (8,0) {\includegraphics[width=0.45\textwidth]{Fractal}};	
\end{tikzpicture}
\]
\caption{All possible patches of $P_n$ connected by a single tile in $\CC_0$ are represented on the left, and a representation of the partial tiling $P$ with infinitely small tiles appears on the right}
\label{Fractal}
\end{figure}

\begin{figure}[ht]
\[
\begin{tikzpicture}
\node (tile3_1) at (0,0) {\includegraphics[scale=0.14]{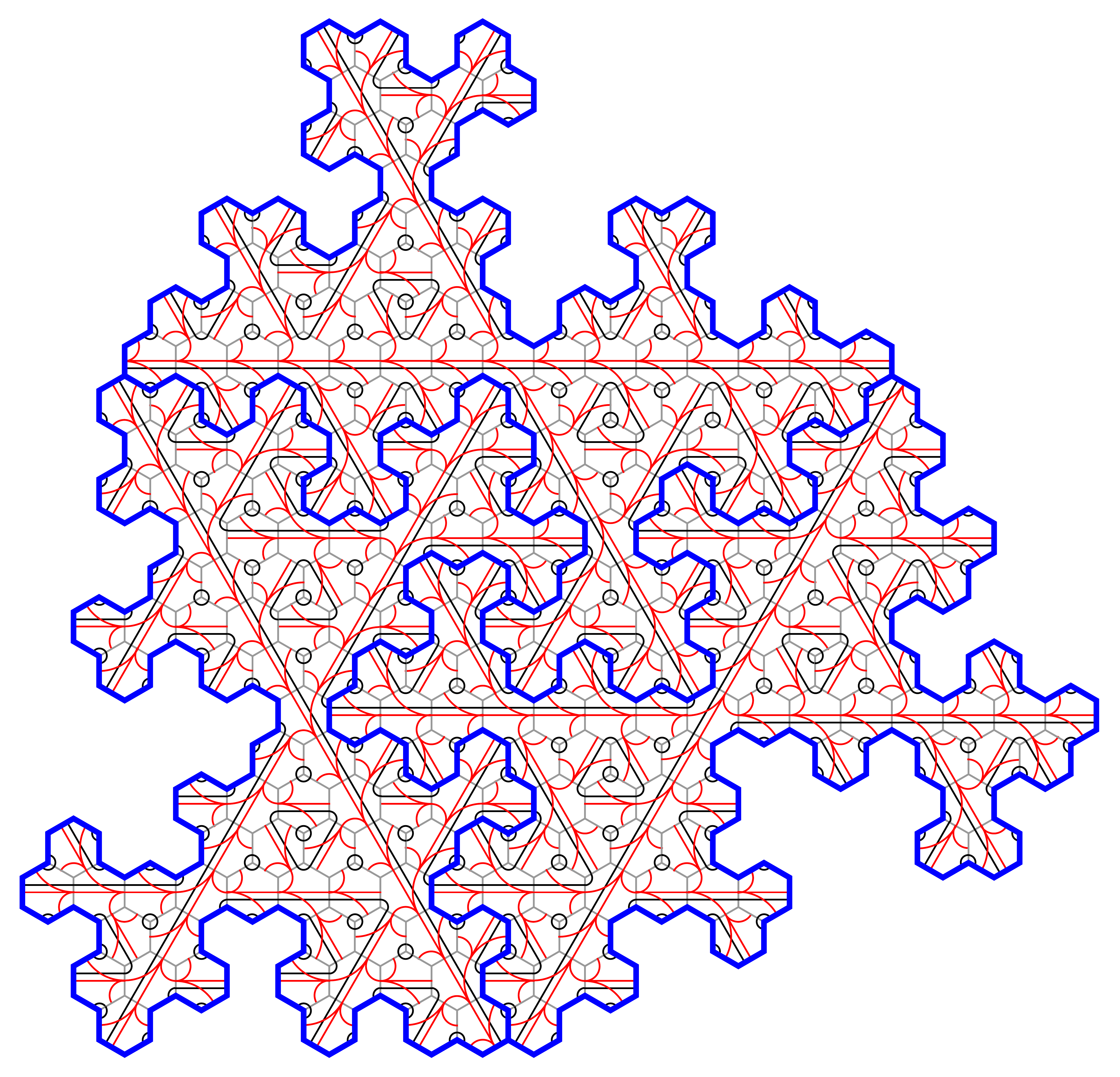} };
\node at (-0.30,-0.35) {$\bullet$};
\node (tile3_2) at (8,0) {\includegraphics[scale=0.14]{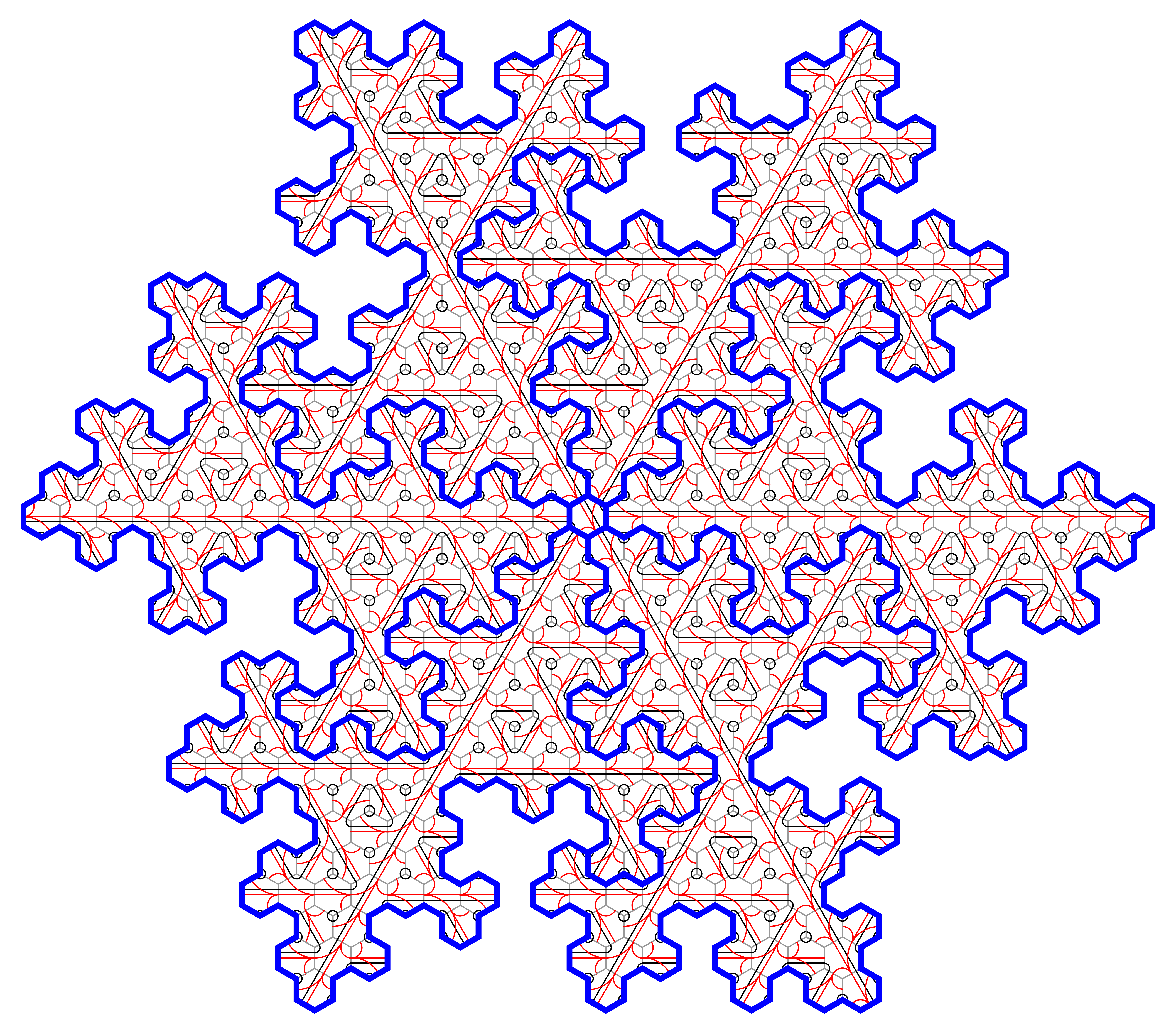}};	
\node at (8,0) {$\bullet$};
\end{tikzpicture}
\]
\caption{Cauchy sequences of the above configurations of patches $P_n$, with the central dots on the origin, converge to tilings $R$ and $S$ in the continuous hull of $\CC_0$. Each of these tilings has three connected components of R2-trees}
\label{hull of C0}
\end{figure}

We now consider the continuous hull of $\CC_0$. As described in Proposition \ref{existence}, such tilings have successively larger hexagonal grids of interlaced R1-triangles. For any $n \in \N$, a tiling of interlaced R1-triangles decomposes into patches $P_n$ connected by single tiles, where $P_n$ is defined in Proposition \ref{existence} (also see Figure \ref{dotted fit}). The image on the left-hand side of Figure \ref{Fractal} shows all possible arrangements of $P_2$ connected by a single tile. Up to direct isometry, there are exactly two configurations. These are shown for $P_3$ in Figure \ref{hull of C0}. For $n \in \N$, let us call these two patches $R_n$ and $S_n$. Placing the origin at the centre of each patch, as shown in Figure \ref{hull of C0}, we see that $R_n \subset R_{n+1}$ and $S_n \subset S_{n+1}$. For example, the reader can compare the patches in Figure \ref{Fractal}, where $n=2$, with the patches in Figure \ref{hull of C0}, where $n=3$. Therefore,
\[
R:=\bigcup_{n=0}^\infty R_{n} \qquad \text{ and } \qquad S:=\bigcup_{n=0}^\infty S_{n}
\]
are tilings satisfying \textbf{R1}. Thus, in addition to tilings in $\CC_0$, the continuous hull of $\CC_0$ contains direct isometries of the tilings $R$ and $S$. As $n$ tends to infinity, the patches $P_n$ converge to a partial tiling, which can be scaled down at each step so as to be depicted as the fractal on the right-hand side of Figure \ref{Fractal}. The fractal nature of this partial tiling implies that, up to direct isometry, $R$ and $S$ are the only two additional elements in the completion of $\CC_0$. Notice that the tilings $R$ and $S$ have exactly three connected components of R2-trees and each component crosses an infinite number of tiles, as desired.
\end{proof}

The astute reader will notice that the hull of $\CC$ contains all possible tilings satisfying \textbf{R1} with the condition that either $\CC_0$ or $\CC_1$ holds. That is, the dendrite rule \textbf{R2} ensures that there are no periodic tilings, but does not otherwise factor into the final description of the hull.

Finally, we comment on the differences between tilings coming from our rules and Socolar-Taylor tilings. In the Socolar-taylor tilings, their rule \textbf{R2} is designed to ensure that all the R1-patterns in their tilings fall into the class $\CC_0$. Since their rules are local matching rules, their hull is automatically complete. On the other hand, none of the R1-patterns in $\CC_1 \setminus \CC_0$ are possible in the Socolar-Taylor tilings, and hence we have two genuinely different classes of tilings that are not MLD to one another.

\end{document}